\documentclass[11pt,a4paper]{article}
\usepackage{amssymb,amsmath,amsthm}
\usepackage{enumitem}
\usepackage[nosort]{cite}
\usepackage{pict2e,graphicx}
\numberwithin{equation}{section}
\newtheorem{thm}{Theorem}[section]
\newtheorem{df}[thm]{Definition}
\newtheorem{prop}[thm]{Proposition}
\newtheorem{lem}[thm]{Lemma}
\newtheorem{rem}[thm]{Remark}

\newtheorem{cor}[thm]{Corollary}
\let\oldproofname=\proofname
\renewcommand{\proofname}{\rm\bf{\oldproofname}}

\newcommand{\N}{\mathbb{N}}

\newcommand{\R}{\mathbb{R}}

\newcommand{\cD}{\mathcal{D}}

\newcommand{\cF}{\mathcal{F}}

\newcommand{\cI}{\mathcal{I}}

\newcommand{\cL}{\mathcal{L}}

\newcommand{\cO}{\mathcal{O}}
\newcommand{\cP}{\mathcal{P}}

\newcommand{\cR}{\mathcal{R}}
\newcommand{\cS}{\mathcal{S}}

\newcommand{\cX}{\mathcal{X}}
\newcommand{\cY}{\mathcal{Y}}
\newcommand{\cZ}{\mathcal{Z}}

\newcommand{\dd}{\,{\rm d}}
\newcommand{\D}{{\rm d}}
\renewcommand{\div}{\mathop{\mathrm{div}}\nolimits}
\newcommand{\curl}{\mathop{\mathrm{curl}}}

\newcommand{\supp}{\mathop{\mathrm{supp}}}
\newcommand{\1}{\mathbf{1}}
\newcommand{\weakto}{\rightharpoonup}
\newcommand{\app}{\mathrm{app}}
\newcommand{\corr}{\mathrm{cor}}
\newcommand{\TS}{\textstyle}

\newcommand{\BS}{\mathrm{BS}}
\newcommand{\QED}{\mbox{}\hfill$\Box$}
\renewcommand{\:}{\thinspace :}
\newcommand{\vt}{\vartheta}
\newcommand{\vf}{\varphi}
\newcommand{\Ker}{\mathop{\mathrm{Ker}}}
\newcommand{\Ran}{\mathop{\mathrm{Ran}}}
\newcommand{\Rem}{\mathrm{Rem}}
\newcommand{\Rey}{\mathrm{Re}}
\newcommand{\adv}{\mathrm{adv}}
\newcommand{\dif}{\mathrm{dif}}
\newcommand{\Ein}{\mathrm{Ein}}
\newcommand{\kin}{\mathrm{kin}}
\newcommand{\lin}{\mathrm{lin}}
\newcommand{\Rest}{\mathfrak{R}}
\newcommand{\ttC}{\mathsf C}

\begin{document}

\title{Vanishing viscosity limit for axisymmetric vortex rings}

\author{Thierry Gallay and Vladim\'ir \v{S}ver\'ak}

%\date{April 11, 2024}

\maketitle

\begin{abstract}
For the incompressible Navier-Stokes equations in $\R^3$ with low viscosity
$\nu>0$, we consider the Cauchy problem with initial vorticity $\omega_0$ that
represents an infinitely thin vortex filament of arbitrary given strength
$\Gamma$ supported on a circle. The vorticity field $\omega(x,t)$ of the
solution is smooth at any positive time and corresponds to a vortex ring of
thickness $\sqrt{\nu t}$ that is translated along its symmetry axis due to
self-induction, an effect anticipated by Helmholtz in 1858 and quantified by
Kelvin in 1867. For small viscosities, we show that $\omega(x,t)$ is
well-approximated on a large time interval by $\omega_\lin(x-a(t),t)$, where
$\omega_\lin(\cdot,t)=\exp(\nu t\Delta)\omega_0$ is the solution of the heat
equation with initial data $\omega_0$, and $\dot a(t)$ is the instantaneous
velocity given by Kelvin's formula. This gives a rigorous justification of the
binormal motion for circular vortex filaments in weakly viscous fluids. The
proof relies on the construction of a precise approximate solution, using a
perturbative expansion in self-similar variables. To verify the stability of
this approximation, one needs to rule out potential instabilities coming from
very large advection terms in the linearized operator.  This is done by adapting
V.~I.~Arnold's geometric stability methods developed in the inviscid case
$\nu=0$ to the slightly viscous situation. It turns out that although the
geometric structures behind Arnold's approach are no longer preserved by the
equation for $\nu > 0$, the relevant quadratic forms behave well on larger
subspaces than those originally used in Arnold's theory and interact favorably
with the viscous terms.
\end{abstract}

\section{Introduction and main result}\label{sec1}

We consider the Cauchy problem for the 3d incompressible Navier-Stokes equations
in the vorticity form
\begin{align}
  \partial_t \omega + u\cdot\nabla \omega - \omega\cdot\nabla u  \,&=\, \nu\Delta\omega
  \quad \mbox{in}~\, \R^3\times(0,\infty)\,,\label{CP1}\\
  \omega|_{t=0}\,&=\,\omega_0\quad\quad \mbox{in} ~\,\R^3\,,\label{CP2}
\end{align}
where we use the familiar notation $\omega(x,t)$ for the vorticity of the fluid,
and the velocity $u(x,t)$ is given by the Biot-Savart law
$u(x,t) = (4\pi)^{-1}\int_{\R^3}\omega(y,t)\wedge(x-y)\, |x-y|^{-3}\dd y\,.$ Our
focus is on the special case where the initial vorticity $\omega_0 = \Gamma \delta_{\ttC}$
is an idealized vortex filament given by a current\footnote{Here the term
  {\it current} can be understood in its heuristic meaning but also in the
  technical meaning of the geometric measure theory, such as in~\cite{Federer}.}
of strength $\Gamma$ concentrated on an oriented circle $\ttC\subset\R^3$. More
precisely, $\omega_0$ is the vector-valued measure on $\R^3$ defined by the identity
\begin{equation}\label{i2}
  \langle \omega_0\,,\,\vf\rangle \,=\, \Gamma \sum_{i=1}^3\int_{\ttC}\vf_i\dd x_i\,,
\end{equation}
which is assumed to hold for any continuous test function
$\vf=(\vf_1,\vf_2,\vf_3)$. In the well-known analogy between fluid mechanics and
electromagnetism, $\omega_0$ can be thought of as an electric current of
intensity $\Gamma$ flowing through an infinitely thin wire represented by the
circle $\ttC$; the direction of the current is then given by the orientation of the
circle and the sign of $\Gamma$.  Vortex filaments were already considered in
the classical 1858 paper of Helmholtz~\cite{Helmholtz} which deals with the
inviscid case $\nu=0$ corresponding to the Euler equation. Helmholtz argued that
a circular vortex filament of zero thickness would move with infinite speed
along its symmetry axis. In 1867 Kelvin~\cite{Kel} established the following
formula for the speed of a vortex ring of small but finite thickness $d>0$ and
radius $r_0 \gg d$\:
\begin{equation}\label{i3}
  V \,\approx\, \frac{\Gamma}{4\pi r_0}\Bigl(\log\frac{8r_0}{d} - C\Bigr)\,,
\end{equation}
where $C \in \R$ is a dimensionless constant that depends on the distribution of vorticity
inside a cross section of the ring. If the distribution is uniform, which is probably
the assumption made by Kelvin, the relevant value is $C = \frac14$, see
\cite[\S 163]{Lam}. 

\begin{figure}[ht]
  \begin{center}
  \begin{picture}(400,140)% width and height of the picture
  \put(50,-10){\includegraphics[width=1.00\textwidth]{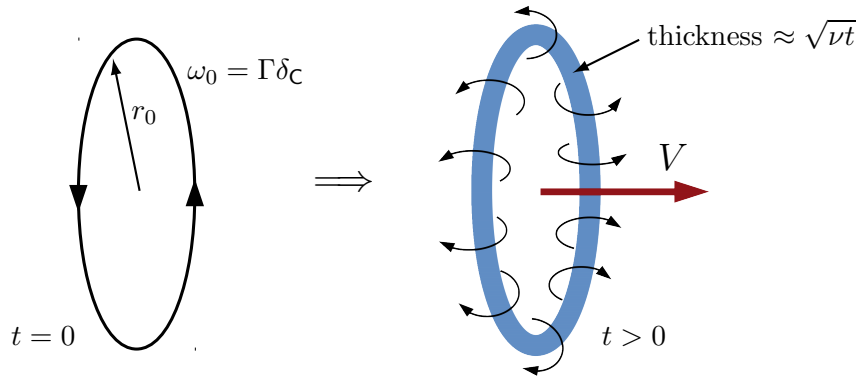}}
  \put(284,76){\Large $V$}
  \put(87,95){$r_0$}
  \put(280,123){thickness $\approx \sqrt{\nu t}$}
  \put(42,10){$t = 0$}
  \put(263,10){$t > 0$}
  \put(108,110){$\omega_0 = \Gamma \delta_{\ttC}$}
  \put(155,68){\Large $\Longrightarrow$}
  \thicklines
  \put(90,68){\vector(-0.2,1){10.0}}  
  \put(277,125){\vector(-1,-0.5){25.0}}
  \end{picture}
  \caption{{\small
  An illustration our main result. Starting from a vortex filament
  supported on an oriented circle $\ttC$, the solution of the Navier-Stokes equation
  evolves into a viscous vortex ring of thickness proportional to $\sqrt{\nu t}$ which
  moves along the symmetry axis at a speed $V$ given by the Kelvin-Saffman formula
  \eqref{i4}. In the right picture, the vortex lines are circles that fill the solid
  torus depicted in blue, whereas the black arrows denote the trajectories of the 
  fluid particles. }}
  \label{fig1}
  \end{center}
\end{figure}

In the viscous case $\nu > 0$, the solution originating from the singular
filament $\omega_0 = \Gamma\delta_{\ttC}$ becomes smooth for any positive time
$t > 0$ and is expected to represent a viscous vortex ring of thickness
proportional to $\sqrt{\nu t}$, as long as that quantity is small compared to
the radius $r_0$ of the ring.  Based on Kelvin's formula one anticipates that
the vortex ring will move at the (time-dependent) speed~\eqref{i3} with
$d=\sqrt{\nu t}$ and $C$ corresponding to a Gaussian distribution of vorticity
inside the core. To the best of our knowledge, the relevant value of the constant
$C$ was first determined by Saffman in~\cite{Sa}. The computation gives $C =
\frac32 \log(2) +\frac12(1-\gamma_E)$, where $\gamma_E\approx 0.5772...$ is
Euler's constant.\footnote{ Fraenkel's paper~\cite{Fr1} contains formulae that
  can be used to obtain the same result. Tung and Ting in~\cite{TT} also give a
  formula for $C$ of a similar nature, which however needs a small correction.}
We will refer to the formula
\begin{equation}\label{i4}
  V(t) \,=\, \frac{\Gamma}{4\pi r_0}\left(\log\frac{8r_0}{\sqrt{\nu t}}
  -\frac32\log 2-\frac12(1-\gamma_E)\right)
\end{equation}
as the {\em Kelvin-Saffman formula} for the speed of a viscous vortex ring. 

When the initial circle $\ttC$ is parametrized by $(r_0\cos \theta,r_0\sin\theta,0)$
for $\theta\in[0,2\pi]$, with the orientation in the direction of increasing
$\theta$, the translational motion will be in the positive direction along the
$x_3$-axis if $\Gamma > 0$.

It is proved in~\cite{GS2} that the Cauchy problem~\eqref{CP1},~\eqref{CP2} with
$\omega_0=\Gamma\delta_{\ttC}$ has a unique solution in natural classes of
axisymmetric fields. The main result of the present paper, Theorem~\ref{main1} below,
describes the precise behavior of that solution in the low viscosity regime
where the {\em circulation Reynolds number} $\Rey := \Gamma/\nu$ is large. Our
description is valid on a time interval whose length is intermediate between the
{\em advection time} and the {\em diffusion time}, defined respectively as
\begin{equation}\label{timescales}
  T_\adv \,=\, \frac{r_0^2}{\Gamma}\,, \qquad  T_\dif \,=\, \frac{r_0^2}{\nu}\,.
\end{equation}
Note that $T_\adv \ll T_\dif$ when $\Rey \gg 1$. The leading term in our
approximation is exactly the one suggested by the Kelvin-Saffman formula
together with the simplest diffusion heuristics: The ring diffuses according to
the linear heat equation, and translates with speed \eqref{i4} along its
symmetry axis. Denoting by $\omega_\lin(x,t)$ the solution of the heat equation
$\partial_t\omega=\nu\Delta \omega$ with initial condition $\omega|_{t=0}=\omega_0=
\Gamma\delta_{\ttC}$, and defining $\|\eta\| = \|\eta/r\|_{L^1(\R^3)}$, where
$r=r(x)$ is the distance from $x$ to the symmetry axis, we can state our main
result as follows.

\begin{thm}\label{main1}
There exist dimensionless constants $K > 0$, $R_0 > 0$, and $\sigma \in (0,\frac13)$ 
such that, for all $\Gamma > 0$, all $r_0 > 0$, and all $\nu > 0$ satisfying 
$\Rey := \Gamma/\nu \ge R_0$, the following holds. If $\omega_0=\Gamma\,\delta_{\ttC}$ 
where $\ttC$ is an oriented circle of radius $r_0$, the unique axisymmetric solution 
$\omega$ of the Cauchy problem~\eqref{CP1},~\eqref{CP2} established in \cite{GS2} can 
be expressed for $t \in (0,T_\adv\,\Rey^\sigma)$ as 
\begin{equation}\label{nseapp}
  \omega(x,t) \,=\, \omega_\lin(x-a(t),t) \,+\, \omega_\corr(x,t)\,,
  \quad \hbox{with}\quad  \|\omega_\corr(\cdot\,,t)\| \,\le\, K\,\Gamma\,
  \biggl(\frac{\sqrt{\nu t}}{r_0}\biggr)^{\!1-3\sigma}\,,
\end{equation}
where $a(t)$ describes the translation of the ring along its symmetry axis
according to the Kelvin-Saffman formula \eqref{i4}. Specifically, if
$\ttC = \{(r_0\cos\theta,r_0\sin\theta,0)\,;\, \theta \in [0,2\pi]\}$ is
oriented positively, one has $a(t) = (0,0,a_3(t))$ where
$a_3(t) = \int_0^t V(s)\dd s$ and $V$ is given by \eqref{i4}.
\end{thm}

An extended version of our result, including a more precise approximate solution
and a much stronger control of the correction term, is formulated as
Theorem~\ref{main2} below, after the necessary notation has been introduced in
Section~\ref{sec2}. In particular, the exponent $1-3\sigma$ in~\eqref{nseapp}
can be improved to $1$ if we take into account higher-order corrections to the
Kelvin-Saffman formula.

In Theorem~\ref{main1}, the constants $K$ and $R_0$ are large, whereas the exponent
$\sigma > 0$ is taken small. We conjecture that an approximation result of the form
\eqref{nseapp} remains valid on longer time scales of order $T_{\rm adv} \Rey^{\sigma'}$ 
with $\sigma'$ close to $1$, but we have no proof so far. In view of
\eqref{i3}, the advection time $T_\adv$ can be interpreted as the time needed for a vortex
ring of circulation $\Gamma$ and small (but not infinitesimal) aspect ratio $d/r_0$ to
travel over a distance comparable to its radius $r_0$. In contrast, the diffusion time
$T_\dif = T_\adv\Rey$ is the time at which the diffusion length $\sqrt{\nu t}$ becomes
equal to the radius $r_0$, so that the vortex ring structure is essentially lost. The
assumption that $\Rey \gg 1$ means that the vortex ring can travel along its symmetry 
axis over a very long distance, compared to its radius $r_0$, before being destroyed by 
diffusion. In particular, on the time scale $T = T_\adv\,\Rey^\sigma$ where Theorem~\ref{main1}
provides a rigorous control we find, using \eqref{i4} and \eqref{timescales}, 
\[
  |a(T)| \,=\, \int_0^{T} V(t)\dd t \,=\, \frac{r_0}{4\pi}
  \,\Rey^\sigma\Bigl(\log\bigl(\Rey^{\frac{1-\sigma}{2}}\bigr) + C'\Bigr)\,,
\]
for some constant $C'$. Obviously the quantity in the right-hand side grows boundlessly 
as $\Rey \to +\infty$, even in the limiting case where $\sigma = 0$ and $T = T_\adv$. 

It is instructive to compare the situation for vortex rings with the case of a
rectilinear filament, where the vorticity is initially concentrated on a
straight line $\ell$ in $\R^3$. Let us denote this initial vorticity field by
$\omega_0 = \Gamma\delta_\ell$.  In that case the solution of the full vorticity
equation is given by $\omega(\,\cdot\,,t) = \Gamma e^{\nu t\Delta}\delta_\ell$,
because the nonlinear terms vanish identically due to symmetries when evaluated
on the solution of the heat equation $\partial_t\omega=\nu\Delta\omega$.
Although the evolution of the velocity and the vorticity fields does not look
very dramatic, the fluid particles in the vicinity of $\ell$ do move at very
large speeds when $\nu t$ is small, and the inertial forces in the fluid are
therefore significant.  However, these forces are exactly balanced by the
pressure gradient.

When the rectilinear filament is bent into a vortex ring (as already considered
in Helmholtz's 1858 paper), the inertial forces are no longer in balance and the
ring has to move. To achieve a relatively smooth motion, the bent vortex has to
be ``well-prepared" so that the inertial forces generated by the fast-moving
fluid particles are still mostly canceled and do not generate fast
oscillations. The initial condition $\omega_0= \Gamma\delta_{\ttC}$ has the
advantage of letting the equation to adjust the vorticity field into a
well-prepared state without trying to achieve this ``by hand". Quite remarkably,
this adjustment is made in exactly such a way that the oscillations are
avoided.\footnote{In the related situation of interacting vortices in $\R^2$,
  this was already observed in~\cite{Ga}.} The largest inertial forces still
cancel and the situation remains somewhat close to the rectilinear case with
only two significant differences: (a) some motion of the ring along its axis of
rotational symmetry is needed to balance the forces, but the speed of this
motion is much lower than the speed of the fast particles in the fluid; (b) once
the thickness of the ring becomes comparable to its radius, new effects (not
discussed in this work) appear.

Theorem~\ref{main2} can be compared with a result by Brunelli and Marchioro
\cite{BrunMar}, where the authors consider general axisymmetric vorticities that
are initially supported in a torus of major radius $r_0 > 0$ and minor radius
$0 < \rho_0 \ll r_0$. Under certain technical assumptions, they show that the
solution of the Navier-Stokes equations remains essentially concentrated in a
thin torus which moves along the symmetry axis according to Kelvin's law. If the
vortex strength $\Gamma$ is kept fixed, the solution is under control on a time
interval of length $T\log(r_0/\rho_0)^{-1}$, which therefore shrinks to zero as
$\rho_0 \to 0$. Also, the authors assume that the viscosity satisfies
$\nu T \le \rho_0^2$ (up to a logarithmic correction), so that the viscous
effects can be treated perturbatively. In the same spirit, the case of several
vortex rings with a common symmetry axis was recently considered in \cite
{ButCavMar1}, see also \cite{BenCagMar,ButCavMar2} for similar results in the
inviscid case. Our Theorem~\ref{main2} is restricted to specific initial data,
which correspond to $\rho_0 = 0$, but it provides a more precise control of the
solution on a much longer time scale, and the diffusive effects are not treated
perturbatively.

\subsection{Main ideas of the proof of Theorem~\ref{main1}}

Our analysis starts with the construction of a precise approximation of the
solution $\omega(x,t)$. This is achieved by writing the solution in suitable
self-similar coordinates that capture well the singular behavior of the solution
at $t=0$ through explicit rescalings of a smooth ``profile" $\eta$ that can be
thought of as a perturbation of a suitable Gaussian $\eta_0$. The perturbed
profile $\eta$ is expressed as an asymptotic series in the time-dependent
parameter $\epsilon=\sqrt{\nu t}/{\bar r}$, with $\bar r=\bar r(t)$ denoting the
instantaneous radius of the ring. To achieve a precision that is sufficient for
our purposes, we need an expansion up to the fourth order:
$\eta=\eta_0+\epsilon\eta_1+\epsilon^2\eta_2+\epsilon^3\eta_3+\epsilon^4\eta_4
+\eta_\corr$. The profiles $\eta_j$ with $j\ge 1$ are obtained by inverting
operators containing the small parameter $\delta = 1/\Rey =\nu/\Gamma$, and in
that sense we really deal with a two-parameter expansion. As far as we know,
this is somewhat different from other expansions in the literature, such as
\cite{CT,Sa,FM}.  A one-parameter formal expansion in $\epsilon$ would treat
$\delta$ as $\sim\epsilon^2$, in view of the relation
$\bar r^2 \epsilon^2=\delta\,\Gamma\,t$. Keeping both parameters makes it easier
to cover the regimes when $\epsilon^2$ and $\delta$ are not really comparable,
as is the case for very small and very large times. For the sake of
completeness, we mention that the vorticity profiles $\eta_j$ for $j\ge 1$ can
also depend on $\log\epsilon$. That phenomenon is well known, and the leading
term in the speed of the ring is precisely related to choosing a moving
coordinate system in which the terms with $\log\epsilon$ in $\eta_1$ are
eliminated.

The main difficulty in the proof of Theorem~\ref{main1}, however, is not in the
computation of an approximate solution, but in showing that the true solution
remains close to this approximation on a large time interval. This requires
fairly strong stability properties for the linearization of the vorticity
equation at the approximate solution, which is very singular in the low
viscosity regime. When the initial condition corresponds to a finite number of
parallel rectilinear vortices, a stability analysis was carried on in~\cite{Ga}
by using suitable weighted $L^2$ spaces adapted to the specific features of the
rectilinear vortices with Gaussian profiles.  In the vortex ring case the
nonlinearity of the equations starts affecting the formal expansions earlier and
it is unclear whether the setup in \cite{Ga} can be used to show that the vortex
ring will not disintegrate on time-scales approaching zero as $\nu\to 0$. A
recent important work~\cite{Bedrossian-et-al} extends some of the 2d methods for
proving stability to a relevant 3d situation, but the length of the time
interval over which the solution is under control may approach $0$ as $\nu$
tends to $0$.
 
In physical flows and numerical experiments one observes a remarkable degree of
stability of vortex rings as well as signs of instabilities with respect to
non-axisymmetric perturbations, see for example~\cite{Widnall-Sullivan,
  Maxworthy}. At a rigorous mathematical level the stability issues have not
been well understood.  In fact, when $\Gamma/\nu$ is not small, not only the
stability, but even the uniqueness of the solutions of the Cauchy problem above
with $\omega_0=\Gamma\,\delta_{\ttC}$ (and also with $\omega_0=\Gamma\,\delta_\ell$) is
open in classes of solutions that do not share the symmetry of the initial data.
 
In the 1960s, V.~I.~Arnold suggested a variational method for proving stability
of steady solutions to Euler's equation based on a geometric insight that can
be summarized as follows, using the Hamiltonian setup of~\cite{Marsden-Weinstein}: 
 
\smallskip\noindent
(a) The incompressible Euler equation can be viewed as a Poisson system with a
Hamiltonian function given by the usual kinetic energy. 
 
\smallskip\noindent
(b) The steady states are critical points of the energy on the symplectic
leaves. The latter coincide with the {\it coadjoint orbits}, called just {\it orbits} 
in what follows, of the group of the volume-preserving diffeomorphisms 
of the fluid domain acting by push-forward on the vorticity fields. 

\smallskip\noindent
(c) When the critical point is a local maximum or a local minimum on an orbit,
the corresponding steady state should be stable.

\smallskip
These ideas fit into a broader family of methods used for proving stability
of solutions of Hamiltonian systems by invoking extremality properties of
a conserved quantity under constraints given by other conserved quantities. For
example, a circular planetary orbit in the three-dimensional Kepler problem is
stable because it minimizes energy for a given angular momentum.\footnote{It is 
well-known that this is no longer the case in dimension four and higher \cite{Gol}.} 
In the applications to vortex rings, it is natural to restrict the analysis to 
{\em axisymmetric flows with no swirl}, which means that the velocity field is 
invariant under rotations about a symmetry axis and under reflection across any 
plane containing that axis. 

Arnold's method has found many applications to Euler flows in 2d (see, for
example, \cite{Khesin}), and has also been invoked in the work of
Benjamin~\cite{Ben} on inviscid vortex rings that is directly relevant for our
purposes here.  Although some arguments in~\cite{Ben} may not be fully rigorous,
they provide important suggestions for investigating stability of inviscid vortex
rings in the class of axisymmetric solutions. In a different direction, the conservation 
of energy, impulse, and vortex strength has been used to control the evolution 
of a general class of concentrated solutions of the Euler equations describing 
vortex rings, see for example~\cite{BenCagMar,ButCavMar2}.
 
There is voluminous literature on the stationary vortex ring solutions of the
Euler equation, starting with the explicit solution of Hill \cite{Hill},
see e.g. \cite{Ambrosetti-Struwe,AT,BB,Bu, Cao1, Cao2, Cao3, Cao4, Fr1,Fr2,Fraenkel-Berger,FT,Ni,Nor,VS}.
Many of these works rely in one way or another on variational aspects of the 
underlying PDEs that have connections to the work of Arnold and Benjamin, albeit
in an indirect way. Roughly speaking, if we compare Arnold's setup 
to the maximization of a function $f(x)$ under constraints $g_j(x)=c_j$, one can 
compare some of the variational approaches in the references above to searching 
for critical points of $f(x)-\lambda_1 g_1(x)-\dots-\lambda_mg_m(x)$ when the 
Lagrange multipliers $\lambda_1, \dots, \lambda_m$ are given. Readers interested 
in related links can find more details in~\cite{GS3}.

In our asymptotic expansions of the solutions of~\eqref{CP1},~\eqref{CP2}
inviscid vortex ring solutions can also be discerned. For each fixed time $t>0$
the third-order expansion in our parameter $\epsilon = \sqrt{\nu t}/{\bar r}$ is
a good approximation of an inviscid vortex ring, at least in the limiting case
where our second parameter $\delta=\nu/\Gamma$ is taken equal to zero. A part of
our stability analysis can be thus understood in terms of the stability
properties of this ring, see Remark~\ref{remEuler} and Section~\ref{ssec38} for
more details.
  
If one wishes to apply Arnold's ideas to the solutions of~\eqref{CP1},~\eqref{CP2},
there appears to be a non-trivial obstacle: The viscous flows do not preserve
the geometric structures that are at the basis of Arnold's considerations and
the influence of the viscosity is too large to treat the viscous terms
perturbatively. At first this may seem to be a serious problem: If the
preservation of the orbits and the Hamiltonian nature of the equations are
violated beyond the reach of the perturbative approach (such
as~\cite{BrunMar, ButCavMar1}), can the geometric structure relying on
maximization of the energy on symplectic leaves be helpful?  In our previous
work~\cite{GS3} we showed, in a much simpler situation, that the answer to this
question can be positive. It turns out that the quadratic forms coming up in
Arnold's stability analysis, although originally envisaged as quadratic forms on
the tangent spaces to the orbits, are often well-behaved on much larger
subspaces. This point can still be conceptually explained by the geometry of the
Euler equation.  What we find more surprising is that Arnold's forms also have
favorable behavior with respect to the dissipative term generated by the
viscosity. We can show this by direct calculation, but we do not have a good
conceptual explanation of this fortuitous circumstance.  In the paper~\cite{GS3}
we showed that the above ideas can be used to prove the stability of the
rectilinear vortex solution (in self-similar variables) with respect to
perturbations for which the vorticity field stays parallel to the original
vortex line.  This result has been established previously by a different method
\cite{GW2}. The new proof in~\cite{GS3} can be thought of as a proof of concept
that the ideas of Arnold can be applied even in the presence of viscosity. The
application to vortex rings presented here is more complicated, but we are not
aware of any simpler approach in that case.

To conclude this section, we mention a recent important work by D\'avila, Del
Pino, Musso, and Wei~\cite{DDMW2}, where the authors rigorously establish 
``leapfrogging'' of inviscid vortex rings. The construction uses ``gluing methods'' 
that were previously developed in \cite{DDMW1} to study the interaction of vortices 
in the plane. The approach shares similarities with ours, as it relies on the 
construction of accurate approximate solutions and their stability analysis.
The stability part also uses an Arnold-type energy functional, although the connection 
to Arnold's geometric viewpoint is not explicit.  In the inviscid case, the expansion 
parameter $\epsilon > 0$ does not need to change during the motion, and the solution is 
controlled on a time interval of size $T/|\log\epsilon|$. This is shorter than in 
Theorem~\ref{main1}, but our result is restricted to a single vortex ring, and uses 
viscous effects. One expects the viscosity to have a stabilizing role, but its interaction 
with the geometric structures of Arnold requires a careful analysis. One needs to show 
that the solutions will stay ``coherent" for a sufficiently long time and the viscous 
effects will not be enhanced too much by the high velocities inside and near the ring.

Another description of the leapfrogging motion of vortex rings, in a different parameter 
regime, can be found in \cite{ButCavMar3}.

\subsection{Comments on the local induction approximation for general filaments}

The problem studied in this paper can be considered as a special case of
the viscous version of the {\it local induction approximation conjecture}. In the
setup considered here the conjecture could be formulated as follows:
if we replace the circle $\ttC$ be a general smooth closed curve and
consider the Cauchy problem~\eqref{CP1}, \eqref{CP2} with $\omega_0=\Gamma\delta_\ttC$,
the motion of the filament $\ttC$ should still be determined essentially 
by two effects: the diffusion, which transforms the filament into a
vortex tube of thickness $d(t)\approx\sqrt{\nu t}$ at time $t$, and the
advection by the self-induced velocity field. The latter is described by
a geometric equation that represents an extension of Kelvin's formula to general
smooth curves, and was derived by Da Rios~\cite{DaRios} in 1906:
\begin{equation}\label{LIA}
  {\mathbf V} \,\approx\, \biggl(\frac{\Gamma}{4\pi r}
  \log\frac{8r}{d}\biggr)\,{\mathbf b}\,.
\end{equation}
Here ${\mathbf V}$ is the vector representing the local velocity of the filament, 
${\mathbf b}$ denotes the local binormal vector, $r$ is the local radius of
the curvature, and $d$ denotes the local thickness of the filament. (All these
quantities may be time- and position-dependent.) In the limit $\nu\to 0$
the approximation should be valid until the geometric evolution of the
curve by the binormal flow leads to a self-intersection. For
general initial curves $\ttC$ the time of the first self-intersection may be
approaching zero as $\nu$ approaches zero.  The first significant step towards
this general case, a local-in time well-posedness result for a fixed $\nu>0$,
was obtained in~\cite{Bedrossian-et-al}. Some formal computations related to the
conjecture are presented in~\cite{CT} and we also refer the reader to the
important conditional result in~\cite{Jerrard-Seis}.
Our result can be viewed as a proof of the viscous formulation of the conjecture 
in the special case where the curve $\ttC$ is a circle. 

For a general smooth curve $\ttC$ and a sufficiently small Reynolds number
$\Gamma/\nu$, the Cauchy problem~\eqref{CP1}, \eqref{CP2} is globally well-posed
as first shown in~\cite{GM} by a perturbation analysis, see also~\cite{KT} for a
more general result in the same spirit. Accurate calculations in the recent
noteworthy preprint~\cite{FontVega} suggest that even in these perturbative
regimes the motion by the local induction approximation can still be discerned,
although its effect is small and the distance traveled by the ring due to the
velocity field~\eqref{LIA} seems to be quite shorter than its thickness.

The general case of the local induction approximation conjecture for the setup
considered in this paper seems to be difficult. In fact, it is unclear whether
the strongest version of the conjecture is valid even for small perturbations
of the circle, as the perturbed filaments may perhaps become unstable to
general 3d perturbations before possible self-intersections. For example, the
instabilities studied in~\cite{Widnall-Sullivan, Maxworthy} may be relevant.

\section{Preliminaries and sketch of the proof}\label{sec2}

In this section we introduce the notation that is necessary to formulate our
result in its stronger form, and we give a pretty detailed sketch of the overall
proof. The construction of the approximate solution will be performed in
Section~\ref{sec3}, and the stability analysis in Section~\ref{sec4}.  Technical
calculations are postponed to Appendix~\ref{secA} and \ref{secB}.

\subsection{Formulation of the problem in  cylindrical coordinates}\label{ssec21}

In a suitable Cartesian coordinate system, the circle of radius $r_0 > 0$ which
represents the support of the initial vorticity \eqref{i2} is given by $\ttC =
\{(r_0\cos\theta,r_0\sin\theta,0)\,;\,\theta\in[0,2\pi]\}$. Due to the symmetries
of the problem, it is natural to introduce the standard cylindrical coordinates
$(r,\theta,z)$ defined by $x_1 = r\cos\theta$, $x_2 = r\sin\theta$, $x_3 = z$ and
to restrict our attention to velocity and vorticity fields of the form
\begin{equation}\label{uaxi}
  u(x,t) \,=\, u_r(r,z,t) e_r + u_z(r,z,t) e_z\,, \qquad
  \omega(x,t) \,=\, \omega_\theta(r,z,t)e_\theta\,,
\end{equation}
where $e_r, e_\theta, e_z$ denote unit vectors in the radial, azimuthal, and
vertical directions, respectively. In the usual terminology, we thus consider
{\em axisymmetric flows with no swirl}, see \cite{MB}. Due to the incompressibility
condition $\div u := r^{-1}\partial_r (ru_r) + \partial_z(u_z) = 0$, the velocity
components $u_r,u_z$ can be expressed in terms of
the Stokes stream function $\psi$\:
\begin{equation}\label{psidef}
  u_r \,=\, -\frac{1}{r}\, \partial_z \psi\,, \qquad 
  u_z \,=\, \frac{1}{r}\, \partial_r \psi\,. 
\end{equation}
With this notation the vorticity formulation of the Navier-Stokes equation~\eqref{CP1}
becomes
\begin{equation}\label{omeq}
  \partial_t \omega_\theta + \Bigl\{\psi,\frac{\omega_\theta}{r}\Bigr\} \,=\,
  \nu\Bigl[\bigl(\partial_r^2 + \partial_z^2)\omega_\theta +
  \partial_r \frac{\omega_\theta}{r}\Bigr]\,,
\end{equation}
where $\{\cdot,\cdot\}$ is the Poisson bracket defined by $\{\psi,\phi\}  = \partial_r \psi
\,\partial_z \phi - \partial_z \psi \,\partial_r \phi$. Eq.~\eqref{omeq} is to be solved
in the domain $\Omega = \{(r,z) \in \R^2\,|\, r > 0\}$. The smoothness of the fields in
the original variables imposes the ``boundary conditions" $\omega_\theta(r,z,t) =
r\zeta(r,z,t)$ and $\psi(r,z,t) = r^2\Psi(r,z,t)$ near $r = 0$, where $\zeta$ and $\Psi$
can be extended to smooth functions on $\R^2 \times \R_+$ that are even functions
of $r$. 

The Stokes stream function  can be represented in terms of the vorticity $\omega_\theta 
= \partial_z u_r - \partial_r u_z$ by the Biot-Savart law
\begin{equation}\label{BSpsi}
  \psi(r,z) \,=\, \frac{1}{2\pi}\int_\Omega \sqrt{r \bar r}\,
  F\left(\frac{(r-\bar r)^2 + (z-\bar z)^2}{r\bar r}\right)
  \omega_\theta(\bar r,\bar z)\dd\bar r \dd\bar z\,,
\end{equation}
where $F : (0,\infty) \to \R$ is
defined by 
\begin{equation}\label{Fdef}
  F(s) \,=\, \int_0^{\pi/2} \frac{1-2\sin^2\psi}{
  \sqrt{\sin^2\psi + s/4}}\dd\psi\,, \qquad s > 0\,.
\end{equation}
Formula \eqref{BSpsi} provides a solution to the equation
\begin{equation}\label{em-pom}
  \curl\curl\left(\frac{\psi}{r}\,e_{\theta}\right) \,=\, \omega_{\theta}\,e_{\theta}
  \quad \hbox{or, equivalently,}\quad -\partial_r\Bigl(\frac{\partial_r \psi}{r}\Bigr)
  - \frac{\partial_z^2\psi}{r} \,=\, \omega_\theta\,,
\end{equation}
which is familiar in magnetostatics, see for example \cite[\S 701]{Max}. The same
expression can also be found in the classical book\cite[\S 161]{Lam}.
It is well-known (and not hard to check) that
\begin{equation}\label{Fexpand}
  F(s)\,=\,
  \begin{cases} \log\frac{8}{\sqrt{s}} - 2 + \cO(s \log s)
  & \hbox{as }s \to 0\,,\\ \frac{\pi}{2s^{3/2}} + \cO(s^{-5/2})
  & \hbox{as }s \to \infty\,. \end{cases}
\end{equation}

Since we wish to solve the Cauchy problem \eqref{CP1}, \eqref{CP2} with initial
data $\omega_0 = \Gamma\delta_{\ttC}$, we assume that the vorticity $\omega_\theta$
in \eqref{uaxi} satisfies the initial condition
\begin{equation}\label{omin}
  \omega_{\theta}\Big|_{t=0} \,=\, \Gamma \delta_{(r_0,0)}\,,
\end{equation}
where $\delta_{(r_0,z_0)}$ denotes the Dirac mass at the location $(r_0,z_0) \in \Omega$. 
Our starting point is the following global well-posedness result for the vorticity
equation \eqref{omeq} with such initial data. 

\begin{thm}\label{uniquethm} {\bf \cite{GS2}}
For any $\Gamma > 0$, any $\nu > 0$, and any $(r_0,z_0) \in \Omega$,
the axisymmetric vorticity equation~\eqref{omeq} has a unique global mild
solution $\omega_\theta \in C^0((0,\infty),L^1(\Omega) \cap L^\infty(\Omega))$
such that
\begin{equation}\label{omcond}
  \sup_{t > 0} \|\omega_\theta(t)\|_{L^1(\Omega)} \,<\, \infty\,,
  \qquad \hbox{and}\qquad \omega_\theta(t)\dd r\dd z \weakto
  \Gamma \,\delta_{(r_0,z_0)} \quad \hbox{as }t \to 0\,.
\end{equation}
Moreover there exists a constant $C > 0$, depending only on the ratio
$\Gamma/\nu$, such that
\begin{equation}\label{short-time}
  \int_\Omega \,\Bigl|\omega_\theta(r,z,t) - \frac{\Gamma}{4\pi\nu t}
  \,e^{-\frac{(r-r_0)^2+(z-z_0)^2}{4\nu t}}\Bigr|\dd r \dd z \,\le\, C\,\Gamma
  \,\frac{\sqrt{\nu t}}{r_0}\,\log\Bigl(\frac{r_0}{\sqrt{\nu t}} + 1\Bigr)\,,
\end{equation}
whenever $t \in (0,T_\dif)$, where $T_\dif = r_0^2/\nu$. 
\end{thm}

Here and in what follows, it is understood that $L^1(\Omega) = L^1(\Omega,\D r\dd z)$,
and similarly for the other Lebesgue spaces. Theorem~\ref{uniquethm} establishes the
existence of a four-dimensional family of vortex ring solutions to \eqref{omeq}
parametrized by the circulation $\Gamma > 0$, the viscosity $\nu > 0$, the
initial radius $r_0 > 0$, and the initial vertical position $z_0 \in \R$. Due to
translation invariance in the vertical direction, we may assume without loss of
generality that $z_0 = 0$, and we can also take $r_0 = 1$ by rescaling the space
variables. Then a rescaling of time allows us to change the values of both $\nu$
and $\Gamma$, while keeping the ratio $\Gamma/\nu$ fixed. Therefore, up to symmetries,
the viscous vortex ring solutions we consider here form a {\em one-parameter family}
indexed by the circulation Reynolds number $\Rey := \Gamma/\nu$.

The uniqueness of the vortex ring solution under the minimal assumptions
\eqref{omcond} is discussed in some detail in \cite{GS2}, so we concentrate here
on the short-time estimate \eqref{short-time}, which is of limited use despite
appearances.  For a fixed value of the Reynolds number $\Rey = \Gamma/\nu$, the
right-hand side of \eqref{short-time} is small whenever $t \ll T_\dif$, which
means that the solution of \eqref{omeq} with initial data \eqref{omin} is well
approximated by a Gaussian vortex ring of thickness proportional to
$\sqrt{\nu t}$, located a the {\em initial position} $(r_0, z_0) \in
\Omega$. However, since the constant $C$ depends on the Reynolds number in
a very bad way, estimate \eqref{short-time} gives no information on the solution
at a fixed time $t > 0$ in the low viscosity regime $\nu \to 0$.  This limitation is
not surprising: due to the translational motion along the vertical axis
predicted by the Kelvin-Saffman formula \eqref{i4}, the vortex ring at time $t > 0$ is
actually located at a new position which is rather far from the initial one if
$\nu$ is small.

Our goal in this paper is to replace \eqref{short-time} by an improved estimate
of the form
\begin{equation}\label{fixed-time}
  \int_\Omega \,\Bigl|\omega_\theta(r,z,t) - \frac{\Gamma}{4\pi\nu t}
  \,e^{-\frac{(r-\bar r(t))^2+(z-\bar z(t))^2}{4\nu t}}\Bigr|\dd r \dd z
  \,\le\, K\,\Gamma\,\frac{\sqrt{\nu t}}{r_0}\,, \qquad t \in (0,T_\adv\,\Rey^\sigma)\,,
\end{equation}
where the constant $K$ is now independent of the Reynolds number, if
$\Rey \gg 1$.  Comparing with \eqref{short-time}, we observe that
\eqref{fixed-time} is valid up to the intermediate time
$T_\adv\,\Rey^\sigma$, for some $\sigma \in (0,\frac13)$, which is shorter than
$T_\dif \equiv T_\adv\,\Rey$.  But the main difference is that
\eqref{fixed-time} compares the solution $\omega_\theta(r,z,t)$ to a vortex ring
located at a {\em time-dependent position} $(\bar r(t), \bar z(t))$, which has
to be determined. As we shall see, we can take $\bar r(t), \bar z(t)$ to be  
continuous functions of time which are smooth for $t > 0$ and satisfy 
$\bar r(0) = r_0$, $\bar z(0) = z_0$. Moreover
\begin{equation}\label{vmotion}
  \dot{\bar r}(t) \,=\, \cO\Bigl(\frac{\nu}{r_0}\Bigr)\,, \qquad
  \dot{\bar z}(t) \,=\, \frac{\Gamma}{4\pi r_0}\Bigl(\log \frac{1}{\epsilon(t)}
  + \hat v\Bigr)\Bigl(1 + \cO\bigl(\epsilon(t)^2 + \delta^2\bigr)\Bigr)\,,
\end{equation}
where $\epsilon(t) = \sqrt{\nu t}/{\bar r(t)}$, $\hat v = \frac32\log(2) +
\frac12(\gamma_E - 1)$, and $\delta = \nu/\Gamma$. The first relation in
\eqref{vmotion} implies that $\bar r(t) = r_0\bigl(1 + \cO(\epsilon(t)^2)\bigr)$,
which means that the change in the radius of the vortex ring over the time
interval under consideration is much smaller than the diffusion length
$\sqrt{\nu t}$. The second equality coincides with the Kelvin-Saffman
formula~\eqref{i4}, up to higher order corrections.

\subsection{Self-similar variables}\label{ssec22}

From now on, we fix the circulation $\Gamma > 0$ and the position $(r_0,0) \in \Omega$
of the initial filament, and we consider the vortex ring solution given by
Theorem~\ref{uniquethm}, in the regime where the viscosity $\nu > 0$ is
small. In view of the approximation formula \eqref{fixed-time}, which is
our objective, it is natural to make the following self-similar Ansatz
for the axisymmetric vorticity and the associated Stokes stream function\:
\begin{equation}\label{etadef}
\begin{split}
  \omega_\theta(r,z,t) \,&=\, \frac{\Gamma}{\nu t}\,\eta
  \Bigl(\frac{r-\bar r(t)}{\sqrt{\nu t}}\,,\,\frac{z-\bar z(t)}{\sqrt{\nu t}}
  \,,\,t\Bigr)\,, \\ 
  \psi(r,z,t) \,&=\, \Gamma\, \bar r(t)\,\phi
  \Bigl(\frac{r-\bar r(t)}{\sqrt{\nu t}}\,,\,\frac{z-\bar z(t)}{\sqrt{\nu t}}
  \,,\,t\Bigr)\,,
\end{split}
\end{equation}
where the time-dependent position $(\bar r(t),\bar z(t)) \in \Omega$ has to be
determined. We introduce the important notation
\begin{equation}\label{RZeps}
  \delta \,=\, \frac{\nu}{\Gamma}\,, \qquad
  \epsilon \,=\, \frac{\sqrt{\nu t}}{\bar r(t)}\,, \qquad
  R \,=\, \frac{r-\bar r(t)}{\sqrt{\nu t}}\,, \qquad
  Z \,=\, \frac{z-\bar z(t)}{\sqrt{\nu t}}\,.
\end{equation}
The evolution equation for the rescaled vorticity $\eta(R,Z,t)$ is found to be
\begin{equation}\label{etaeq}
  t\partial_t \eta + \frac{\Gamma}{\nu}\,\Bigl\{\phi\,,\frac{\eta}{1+\epsilon R}\Bigr\}
  - \sqrt{\frac{t}{\nu}}\Bigl(\dot{\bar r}\,\partial_R \eta + \dot{\bar z}
  \,\partial_Z \eta\Bigr) \,=\, \cL \eta + \partial_R \Bigl(\frac{\epsilon
  \eta }{1+\epsilon R}\Bigr)\,,
\end{equation}
where $\bigl\{\phi,\chi\bigr\} = \partial_R \phi \,\partial_Z \chi - 
\partial_Z \phi \,\partial_R \chi$ is the Poisson bracket in the new variables
$(R,Z)$, and $\cL$ is the Fokker-Planck operator
\begin{equation}\label{cLdef}
  \cL \,=\, \partial_R^2 + \partial_Z^2 + \frac12\bigl(R\partial_R + Z\partial_Z
  \bigr) + 1\,. 
\end {equation}
Eq.~\eqref{etaeq} is to be solved in the time-dependent domain
\begin{equation}\label{Omeps}
  \Omega_\epsilon \,=\, \bigl\{(R,Z) \in \R^2\,\big|\, 
  1 + \epsilon R > 0\bigr\}\,,
\end{equation}
with the Dirichlet boundary condition $\eta(-1/\epsilon,Z,t) = 0$ for all $(Z,t) \in \R
\times \R_+$.

As in \cite{GS2}, it is useful to introduce the velocity field $U = (U_R,U_Z)$ defined by
\begin{equation}\label{Udef}
  U_R \,=\, -\frac{\partial_Z\phi}{1+\epsilon R}\,, \qquad 
  U_Z \,=\, \frac{\partial_R\phi}{1+\epsilon R}\,,
\end{equation}
in terms of which the nonlinearity in \eqref{etaeq} reads $\bigl\{\phi\,,\frac{\eta}{1+\epsilon R}
\bigr\} \,=\, \partial_R \bigl(U_R\,\eta) + \partial_Z (U_Z \,\eta)$. The stream 
function $\phi$ in \eqref{etaeq} satisfies the elliptic equation
\begin{equation}\label{phidef}
  \eta \,=\, \partial_Z U_R - \partial_R U_Z \,\equiv\, - \partial_R \Bigl(
  \frac{\partial_R\phi}{1+\epsilon R}\Bigr) -\frac{\partial_Z^2\phi}{1+
  \epsilon R}\,, \qquad (R,Z) \in \Omega_\epsilon\,,
\end{equation}
with boundary conditions $\phi(-1/\epsilon,Z,t) = \partial_R \phi(-1/\epsilon,Z,t) = 0$ for
all $(Z,t) \in \R \times \R_+$. Using \eqref{BSpsi}, we easily obtain the representation
formula \cite{GS2}
\begin{equation}\label{BSeps}
  \phi(R,Z) \,=\, \frac{1}{2\pi}\int_{\Omega_\epsilon}\sqrt{(1{+}\epsilon R)
  (1{+}\epsilon R')} \,F\biggl(\epsilon^2\frac{(R{-}R')^2 + (Z{-}Z')^2}{
  (1{+}\epsilon R)(1{+}\epsilon R')}\biggr)\eta(R',Z')\dd R' \dd Z'\,,
\end{equation}
where $F$ is as in \eqref{Fdef}. In what follows we write $\phi = \BS^\epsilon[\eta]$
when \eqref{BSeps} holds. 

The quantities introduced in \eqref{RZeps} are all dimensionless. The first one
is the inverse Reynolds number $\delta > 0$, a fixed parameter that is assumed
to be small.  The second one is the time-dependent aspect ratio $\epsilon > 0$,
which appears in the evolution equation \eqref{etaeq}, in the definition of the
domain \eqref{Omeps}, and in the Biot-Savart formula \eqref{BSeps}. Finally, the
variables $R$, $Z$ are self-similar coordinates centered at the time-dependent
location $(\bar r(t),\bar z(t))$ and normalized according to the size $\sqrt{\nu t}$
of the vortex core. Note that the rescaled functions $\eta, \phi$ defined in
\eqref{etadef} are also dimensionless. 

\begin{rem}\label{crossrem}
Recalling that $\delta = \nu/\Gamma$ and $T_\adv = r_0^2/\Gamma$, we observe that
\begin{equation}\label{crosseq}
  \epsilon^2 \,=\, \frac{\nu t}{r_0^2}\,\frac{r_0^2}{\bar r(t)^2} \,=\,
  \frac{\delta t}{T_\adv}\,\frac{r_0^2}{\bar r(t)^2} \,\approx\,
  \frac{\delta t}{T_\adv}\,,
\end{equation}
as long as the ratio $r_0/{\bar r}(t)$ remains close to unity, which will
always be the case thanks to \eqref{vmotion}. It follows in particular that
$\epsilon^2$ is comparable to $\delta$ whenever $t$ is comparable to
$T_\adv$. Our goal is to control the solution of \eqref{omeq} when
$t \le T_\adv \delta^{-\sigma}$ for some $\sigma \in (0,\frac13)$, and on that
interval it follows from \eqref{crosseq} that
$\epsilon^2 \lesssim \delta^{1-\sigma}$.
\end{rem}

\subsection{Approximate solution}\label{ssec23}

The first important step in our analysis is the construction of an approximate
solution of \eqref{etaeq} with initial data
\begin{equation}\label{etazero}
  \eta_0(R,Z) \,=\, \frac{1}{4\pi}\,e^{-(R^2 + Z^2)/4}\,, \qquad
  (R,Z) \in \Omega_0 = \R^2\,. 
\end{equation}
The associated stream function satisfies $-\Delta_0\phi_0 = \eta_0$, where
$\Delta_0 = \partial_R^2 + \partial_Z^2$. As $\eta_0,\phi_0$ are both radially
symmetric, it is clear that $\{\phi_0,\eta_0\} = 0$, and the Gaussian profile
\eqref{etazero} has the property that $\cL \eta_0 = 0$. Since $\epsilon = 0$
when $t = 0$ in view of \eqref{RZeps}, we conclude that equation \eqref{etaeq}
is satisfied at initial time if $\eta_0$ is given by \eqref{etazero}. 

For $t > 0$, we construct our approximate solution as a power series in 
the time-dependent parameter $\epsilon = \sqrt{\nu t}/\bar r$, the 
coefficients of which depend on the small parameter $\delta$. To this end, 
we multiply both sides of \eqref{etaeq} by $\delta$ and rewrite the equation 
in the equivalent form
\begin{equation}\label{etaeq2}
  \delta \,t\partial_t \eta + \Bigl\{\phi\,,\frac{\eta}{1+\epsilon R}\Bigr\} 
  -\frac{\epsilon \bar r}{\Gamma}\Bigl(\dot{\bar r}\,\partial_R \eta + 
  \dot{\bar z} \,\partial_Z \eta\Bigr) \,=\, \delta \Bigl[\cL \eta + \partial_R 
  \Bigl(\frac{\epsilon\eta }{1+\epsilon R}\Bigr)\Bigr]\,.
\end{equation}
This equation is defined on the time-dependent domain $\Omega_\epsilon$, but
expanding the factors $(1+\epsilon R)^{-1}$ in powers of $\epsilon$ we get at
each order a relation that can be solved in the whole plane $\Omega_0 =
\R^2$. The corresponding approximation for the stream function $\phi$ is
obtained in a self-consistent way by expanding the integrand in \eqref{BSeps} in
powers of $\epsilon$, and integrating order by order over the whole plane
$\R^2$. As is shown in Section~\ref{sec3}, this results in an asymptotic
expansion of the form
\begin{equation}\label{etaapp}
  \eta_\app(R,Z,t) \,=\, \sum_{m=0}^M \epsilon^m\,\eta_m(R,Z,\beta_\epsilon)\,, \qquad
  \phi_\app(R,Z,t) \,=\, \sum_{m=0}^M \epsilon^m\,\phi_m(R,Z,\beta_\epsilon)\,,
\end{equation}
where the dependence of the profiles $\eta_m$ and $\phi_m$ on
$\beta_\epsilon := \log(1/\epsilon)$ is polynomial.  The profiles also depend on
the small parameter $\delta$, but to make the notation lighter this dependence
is not indicated explicitly.  The velocity of the vortex center is not known a
priori, but can be approximated in a similar way as a power series in
$\epsilon$\:
\begin{equation}\label{rzapp}
  \dot{\bar r}(t) \,=\, \sum_{m=0}^{M-1} \epsilon^m\,\dot{\bar r}_m(\beta_\epsilon)\,, \qquad 
  \dot{\bar z}_*(t) \,=\, \sum_{m=0}^{M-1} \epsilon^m\,\dot{\bar z}_m(\beta_\epsilon)\,,
\end{equation}
where the quantities $\dot{\bar r}_m(\beta_\epsilon)$, $\dot{\bar z}_m(\beta_\epsilon)$
depend on $\delta$ and are polynomials in $\beta_\epsilon$. As will be explained below,
the quantity $\dot{\bar z}_*(t)$ in \eqref{rzapp} is only an initial
approximation of the vertical speed of the vortex ring; the final approximation
$\dot{\bar z}(t)$ will be obtained from it by a small adjustment. It is perhaps
worth emphasizing that, throughout the paper, the point $(\bar r(t),\bar z(t))$
is not necessarily the exact center of our vortex. Rather, it is its suitably
chosen approximation.

The outcome of the analysis carried out in Section~\ref{sec3} below is that,
if we want our expansions \eqref{etaapp}, \eqref{rzapp} to hold uniformly
with respect to the parameter $\delta$ in the limit where $\delta \to 0$, 
there is a {\em unique choice} of the profiles $\eta_m,\phi_m$ and of the velocities
$\dot{\bar r}_m$, $\dot{\bar z}_m$ such that\:

\smallskip\noindent
a) Both members of \eqref{etaeq2} agree up to order $\cO(\epsilon^{M+1})$, modulo
powers of $\beta_\epsilon$; 

\smallskip\noindent
b) The point $({\bar r}(t),{\bar z}_*(t)) \in \Omega$ is the center of the
vorticity distribution \eqref{etadef} when $\eta = \eta_\app$. 

\smallskip\noindent The integer $M$ in \eqref{etaapp}, \eqref{rzapp} determines the
accuracy of our approximate solution. It turns out that $M = 4$ will be
sufficient for our purposes. The velocities $\dot{\bar r}(t), \dot{\bar z}_*(t)$
given by \eqref{rzapp} are found to satisfy estimate \eqref{vmotion} with
$\delta = 0$. 

\begin{rem}\label{remEuler}
If we set $\delta = \dot{\bar r} = 0$, equation \eqref{etaeq2} reduces to
\begin{equation}\label{etaeqEuler} 
  \Bigl\{\phi\,,\frac{\eta}{1+\epsilon R}\Bigr\} 
  -\frac{\epsilon \bar r}{\Gamma}\,\dot{\bar z} \,\partial_Z \eta 
  \,\equiv\, \Bigl\{\phi - \frac{\bar r \dot{\bar z}}{2\Gamma}\,(1+\epsilon R)^2
  \,,\frac{\eta}{1+\epsilon R}\Bigr\} \,=\, 0\,,
\end{equation}
which is the relation satisfied by the vorticity $\eta$ and the stream
function $\phi$ of a stationary solution of the Euler equations in a frame
moving with speed $\dot{\bar z}\,e_z$. These are precisely the vortex rings
constructed, for instance, in \cite{Fr1,Fraenkel-Berger,FT,Ambrosetti-Struwe,Bu}.
In that situation the aspect ratio $\epsilon > 0$ is fixed and, as in
\eqref{RZeps}, the dimensionless variables $(R,Z)$ are defined so that
$(r,z) = (\bar r,\bar z) + \epsilon{\bar r}\, (R,Z)$. An approximate solution of
\eqref{etaeqEuler} can be constructed in the form of a power series in
$\epsilon$, as in \eqref{etaapp}, where all profiles $\eta_m, \phi_m$ are even
functions of the variable $Z \in \R$, since this is the case for the
coefficients in \eqref{etaeqEuler} and for the initial approximation
\eqref{etazero}. Returning to the approximate solution \eqref{etaapp}, we deduce
by uniqueness that $\eta_\app$, $\phi_\app$ are even functions of $Z$ in the
limit $\delta \to 0$, and that
$\dot{\bar r} = \frac{\Gamma\,}{r_0}\,\cO(\delta)$ as $\delta \to 0$.
\end{rem}

\begin{rem}\label{remepsdef}
In view of \eqref{RZeps} and \eqref{rzapp}, the function $\epsilon(t)$ is
implicitly defined by the relation
\begin{equation}\label{epsrel}
  \frac{\sqrt{\nu t}}{\epsilon(t)} \,=\, \bar r(t) \,=\, r_0 + 
  \sum_{m=0}^{M-1} \int_0^t \epsilon(s)^m\,\dot{\bar r}_m\bigl(\beta_{\epsilon(s)}
  \bigr)\dd s\,,
\end{equation}
which should hold when $0 < t \ll T_\dif$. As was mentioned in the previous remark,
the radial velocities $\dot{\bar r}_m$ are small when $\delta \ll 1$, so that
Eq.~\eqref{epsrel} will be easy to solve, see Section~\ref{ssec36}. 
\end{rem}

The asymptotic approximation $\eta_\app(R,Z,t)$ is defined on the whole 
plane and does not vanish on the boundary $\partial \Omega_\epsilon$. To obtain
a valid approximate solution of \eqref{etaeq}, we fix $\sigma_0 \in (0,1)$ and
we truncate $\eta_\app$ outside a large ball of radius $\epsilon^{-\sigma_0}$ by
setting
\begin{equation}\label{etaapp1}
  \eta_*(R,Z,t) \,=\, \chi_0\bigl(\epsilon^{\sigma_0}(R^2{+}Z^2)^{1/2}\bigr)\,
  \eta_\app(R,Z,t)\,, \qquad \phi_*(\cdot,t) \,=\, \BS^\epsilon[\eta_*(\cdot,t)]\,,
\end{equation}
where $\chi_0 : \R_+ \to [0,1]$ is a smooth function such that $\chi_0(r) = 1$ for $r \le 1$
and $\chi_0(r) = 0$ for $r \ge 2$. The remainder of that approximation is defined as 
\begin{equation}\label{Remdef}
  \Rem(R,Z,t) \,=\, \cL \eta_* + \partial_R \Bigl(\frac{\epsilon \eta_* }{1{+}
  \epsilon R}\Bigr) - t\partial_t \eta_* - \frac{1}{\delta}\,\Bigl\{\phi_*\,,
  \frac{\eta_*}{1{+}\epsilon R}\Bigr\} + \frac{\epsilon \bar r}{\delta \Gamma}
  \Bigl(\dot{\bar r}\,\partial_R \eta_* + \dot{\bar z}_* \,\partial_Z \eta_*\Bigr)\,.
\end{equation}
By construction this quantity depends on time only through the parameter
$\epsilon = \sqrt{\nu t}/{\bar r(t)}$. 

The accuracy of our approximate solution is quantified by the following
result, which is established in Section~\ref{ssec37} below\:

\begin{prop}\label{Remprop}
Given any $\gamma_0 < 1$ and any $\gamma_5 < 5$, there exist a constant $C > 0$
such that the remainder \eqref{Remdef} satisfies
\begin{equation}\label{Remest}
  \sup_{(R,Z) \in \Omega_\epsilon} e^{\gamma_0(R^2+Z^2)/4}\,|\Rem(R,Z,t)|
  \,\le\, C\bigl(\epsilon \delta + \epsilon^{\gamma_5}\delta^{-1}\bigr)\,,
\end{equation}
whenever the parameters $\epsilon, \delta$ are small enough.
\end{prop}

\subsection{Stability estimates}\label{ssec24}

In our previous work \cite{GS2}, the evolution equation \eqref{etaeq} was
carefully studied in the particular case where $\bar r(t) = r_0$ and
$\bar z(t) = z_0$. This does not make any substantial difference for
the initial value problem at fixed viscosity, and we can thus infer from the
results of \cite{GS2} that Eq.~\eqref{etaeq} has a unique solution $\eta(R,Z,t)$
with initial data $\eta_0$ given by \eqref{etazero}. Our purpose is to show
that, if the inverse Reynolds number $\delta = \nu/\Gamma$ is sufficiently
small, the solution $\eta(R,Z,t)$ remains close to the approximation
\eqref{etaapp1} on a long time interval of the form $(0,T_\adv\delta^{-\sigma})$,
for some small $\sigma > 0$. We use the decomposition\:
\begin{equation}\label{etapert}
  \eta(R,Z,t) \,=\, \eta_*(R,Z,t) + \delta\,\tilde \eta(R,Z,t)\,, \qquad  
  \phi(R,Z,t) \,=\, \phi_*(R,Z,t) + \delta\,\tilde \phi(R,Z,t)\,,
\end{equation}
where $\tilde \phi = \BS^\epsilon[\tilde\eta]$ in the sense of \eqref{BSeps}.
Similarly we assume that the vertical speed of the vortex ring takes the form
\begin{equation}\label{barzpert}
  \dot{\bar z}(t) \,=\, \dot{\bar z}_*(t) + \delta\,\dot{\tilde z}(t)\,,
\end{equation}
where $\dot{\bar z}_*(t)$ is given by \eqref{rzapp} and $\dot{\tilde z}(t)$
is a small correction which is chosen so that the perturbation $\tilde \eta$
has vanishing first order moment in the vertical direction, see
Section~\ref{ssec41}. The equation satisfied by $\tilde\eta$ then reads
\begin{equation}\label{tildeq}
\begin{split}
  t\partial_t \tilde\eta + \frac{1}{\delta}\Bigl\{\phi_*\,,\frac{\tilde\eta}{1
  +\epsilon R}\Bigr\} + \frac{1}{\delta}\Bigl\{&\tilde\phi\,,\frac{\eta_*}{1
  +\epsilon R}\Bigr\} + \Bigl\{\tilde\phi\,,\frac{\tilde\eta}{1
  +\epsilon R}\Bigr\} 
  -\frac{\epsilon \bar r}{\delta\Gamma}\Bigl(\dot{\bar r}\,\partial_R \tilde\eta + 
  \dot{\bar z}_* \,\partial_Z \tilde\eta\Bigr)\\ \,&=\, \cL \tilde\eta + \partial_R 
  \Bigl(\frac{\epsilon\tilde\eta }{1+\epsilon R}\Bigr) + \frac{1}{\delta}
  \,\Rem(R,Z,t) \,+\, \frac{\epsilon \bar r}{\delta\Gamma}\,\dot{\tilde z}
  \,\partial_Z \eta\,.
\end{split}
\end{equation}

Since $\eta_*(R,Z,0) = \eta_0(R,Z)$, the nonlinear evolution equation
\eqref{tildeq} is to be solved with zero initial data. The solution is therefore
driven by the source term $\delta^{-1}\Rem(R,Z,t)$, which is small in view of
Proposition~\ref{Remprop} and Remark~\ref{crossrem} if the parameter $\sigma$ is
small enough. As long as $\tilde\eta$ stays small, the nonlinear term
$\{\tilde \phi,(1{+}\epsilon R)^{-1}\tilde\eta\}$ is of course harmless. The
most serious difficulty in controlling $\tilde\eta$ using~\eqref{tildeq} comes
from the linear terms with a large prefactor $\delta^{-1} = \Gamma/\nu$. These
terms could conceivably trigger violent instabilities that might lead to strong
amplification of $\tilde\eta$ in a short time. Our goal is to show that this
scenario does not occur, due to the special structure of the advection terms in
\eqref{tildeq}. A similar strategy was applied in the previous work \cite{Ga}
devoted to the vanishing viscosity limit of interacting vortices in the plane,
but the specific estimates used there do not seem to be easily adaptable to the
present situation.

To control the time evolution of the solution of \eqref{tildeq}, we use the
energy functional
\begin{equation}\label{Edef}
  E_\epsilon(t) \,=\, \frac12 \int_{\Omega_\epsilon} W_\epsilon\,\tilde \eta^2 \dd R\dd Z
  \,-\, \frac12 \int_{\Omega_\epsilon} \tilde \phi\,\tilde \eta\dd R\dd Z\,,
\end{equation}
where $W_\epsilon : \Omega_\epsilon \to (0,+\infty)$ is a weight function
that will be described below. The first term in the right-hand side of
\eqref{Edef} is a weighted $L^2$ integral of the vorticity $\tilde\eta$,
similar to weighted enstrophies that were used for the same purposes in
\cite{GW2,Ga,GS2}, for instance. The second term is just the kinetic
energy associated with the vorticity perturbation $\tilde\eta$, as can
be seen by invoking \eqref{Udef}, \eqref{phidef} and integrating by parts\:
\[
  \frac12 \int_{\Omega_\epsilon} \tilde \phi\,\tilde \eta\dd R\dd Z \,=\,
  \frac12 \int_{\Omega_\epsilon} \frac{|\partial_R \tilde\phi|^2 + |\partial_Z
  \tilde \phi|^2}{1+\epsilon R}\dd R\dd Z \,=\, \frac12 \int_{\Omega_\epsilon}
  \bigl(|\tilde U_R|^2 + |\tilde U_Z|^2\bigr)(1+\epsilon R)\dd R\dd Z\,.
\]

To construct the weight $W_\epsilon$ in \eqref{Edef}, we consider three
different regions\: 

\medskip\noindent{\bf 1)} The {\em inner region} where $\rho := (R^2{+}Z^2)^{1/2} 
\lesssim \epsilon^{-\sigma_1}$, for some small $\sigma_1 > 0$. Here we choose
\begin{equation}\label{Wepsdef}
  W_\epsilon \,=\, \frac{1}{1+\epsilon R}\, \Phi_\epsilon'\Bigl(
  \frac{\eta_*}{1+\epsilon R}\Bigr)\,,
\end{equation}
where $\eta_*$ is the approximate solution \eqref{etaapp1} and $\Phi_\epsilon
: (0,+\infty) \to \R$ is a smooth function with the property that,
in the region under consideration, 
\begin{equation}\label{Phieps}
  \phi_* - \frac{\bar r \dot{\bar z}_*}{2\Gamma}\,(1+\epsilon R)^2 \,=\,
  \Phi_\epsilon\Bigl(\frac{\eta_*}{1+\epsilon R}\Bigr) \,+\,
  \cO(\epsilon\delta + \epsilon^{\gamma_3})\,,
\end{equation}
for some $\gamma_3 < 3$ that can be arbitrarily close to $3$. 
It is not difficult to understand intuitively why such a function should exist.
Indeed, in the dimensionless variables \eqref{RZeps}, the left-hand side of
\eqref{Phieps} is nothing but the stream function of the approximate solution
$\eta_*$ in a frame moving with constant speed $\dot{\bar z}_*$ in the vertical
direction, see Remark~\ref{remEuler}. If we drop the remainder term
$\cO(\epsilon\delta + \epsilon^{\gamma_3})$ and consider $\epsilon > 0$ as a fixed
parameter, Eq.~\eqref{Phieps} expresses a functional relation between the
potential vorticity $\zeta_* := (1+\epsilon R)^{-1} \eta_*$ and the stream
function, which implies that $\eta_*$ is a stationary solution of the Euler
equation in the moving frame. This is not exactly true, of course, but the
estimate on the remainder $\Rem(R,Z,t)$ in Proposition~\ref{Remprop} ensures
that the approximate solution $\eta_*$ (for a fixed value of $\epsilon > 0$) is
not far from a stationary solution of Euler, and in Section~\ref{ssec38} we
verify that this implies the existence of a function $\Phi_\epsilon$ satisfying
\eqref{Phieps}. Moreover, an easy calculation shows that
\[
  \frac{1}{1+\epsilon R}\,\Phi_\epsilon'\Bigl(\frac{\eta_*}{1+\epsilon R}\Bigr)
  \,=\, \frac{4}{\rho^2}\bigl(e^{\rho^2/4}-1\bigr) + \cO(\epsilon)\,, \qquad
  \rho \,:=\, \sqrt{R^2+Z^2} \,\le\, \epsilon^{-\sigma_1}\,.
\]

\smallskip\noindent{\bf 2)} The {\em intermediate region} where $
\epsilon^{-\sigma_1} \lesssim \rho \le \epsilon^{-\sigma_2}$, for some $\sigma_2 > 1$.
In this area we assume that the weight is approximately constant in space, with value
$W_\epsilon \approx \exp(\epsilon^{-2\sigma_1}/4)$. 

\smallskip\noindent{\bf 3)} The {\em far field region} where $\rho \ge
\epsilon^{-\sigma_2}$. Here we take $W_\epsilon \approx \exp(\rho^{2\gamma}/4)$,
where $\gamma = \sigma_1/\sigma_2$. 

\begin{figure}[ht]
  \begin{center}
  \begin{picture}(300,180)% width and height of the picture
  \put(30,0){\includegraphics[width=1.00\textwidth]{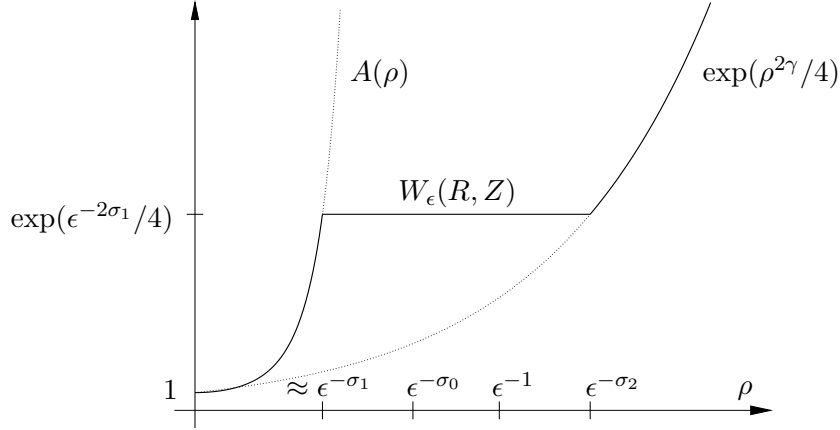}}
  \put(81,21){$\approx \epsilon^{-\sigma_1}$}
  \put(126,21){$\epsilon^{-\sigma_0}$}
  \put(158,21){$\epsilon^{-1}$}
  \put(193,21){$\epsilon^{-\sigma_2}$}
  \put(250,22){$\rho$}
  \put(35,20){$1$}
  \put(-22,85){$\exp(\epsilon^{-2\sigma_1}/4)$}  
  \put(123,95){$W_\epsilon(R,Z)$}  
  \put(105,140){$A(\rho)$}
  \put(237,140){$\exp(\rho^{2\gamma}/4)$}
  \end{picture}
  \caption{{\small When $\epsilon > 0$ is small, the weight $W_\epsilon(R,Z)$
  entering the energy functional \eqref{Edef} is close to a piecewise smooth
  radially symmetric function, which satisfies $W_\epsilon \approx A(\rho) := (4/\rho^2)
  \bigl(e^{\rho^2/4}-1\bigr)$ in the inner region where $\rho := (R^2{+}Z^2)^{1/2}
  \lesssim \epsilon^{-\sigma_1}$. When $W_\epsilon$ reaches the threshold value
  $\exp(\epsilon^{-2\sigma_1}/4)$, the weight is taken approximately constant until
  $\rho = \epsilon^{-\sigma_2}$, and outside that region we set $W_\epsilon \approx
  \exp(\rho^{2\gamma}/4)$ with $\gamma = \sigma_1/\sigma_2$. The dashed lines reflect
  the fact that $\exp(\rho^{2\gamma}/4) \lesssim W_\epsilon \lesssim
  A(\rho)$ where the implicit constants do not depend on the parameter $\epsilon$.
  The intermediate scales $\epsilon^{-\sigma_0}$, where the truncation \eqref{etaapp1}
  occurs, and $\epsilon^{-1}$, which is the distance from the origin to the boundary
  $\partial\Omega_\epsilon$, are indicated for completeness.}}\label{fig2}
  \end{center}
\end{figure}

\medskip 
The actual construction of the weight is more complicated, and ensures 
that $W_\epsilon$ is Lipschitz continuous at the boundaries of the three
regions under consideration, see Section~\ref{sec4} below for details. For the
moment, we just mention that our choice of the energy functional in the inner
region is related to Arnold's variational characterization of the steady states
of the Euler equation, as discussed in our previous work \cite{GS3}. In fact, if
we suppose that $\eta_*$ is a stationary solution of the axisymmetric Euler equation 
in a moving frame (which not exactly true), then the functional \eqref{Edef} with 
the weight \eqref{Wepsdef} corresponds, up to a constant factor, to the second 
variation of the kinetic energy on the isovortical surface, which is the set 
of (potential) vorticities $\zeta := (1+\epsilon R)^{-1}\eta$ that are
measure-preserving rearrangements of $\zeta_*$ \cite{Arn,GS3}. Since the kinetic
energy is conserved under the inviscid dynamics, the advection terms involving
$\delta^{-1}$ in \eqref{tildeq}, which originate from the linearization of
Euler's equation at the ``steady state'' $\zeta_*$, do not contribute to the
time evolution of the functional $E_\epsilon$. In reality $\zeta_*$ is only an
approximate steady state of Euler, and the cancellations alluded to above only
occur up to correction terms of order $\cO(\epsilon\delta + \epsilon^{\gamma_3})$,
but this is sufficient to cancel the dangerous factors $\delta^{-1}$ in
\eqref{tildeq}. On the other hand, away from the inner region, the last term in
\eqref{Edef} is extremely small, so that our functional $E_\epsilon$ reduces to
a weighted enstrophy. We assume that the weight $W_\epsilon$ is approximately constant
in the intermediate region, so that the advection terms in \eqref{tildeq} do not
contribute to the evolution of $E_\epsilon$, and in the far field region the dynamics 
is dominated by the diffusion operator $\cL$ in \eqref{tildeq} so that we can just 
take any radially symmetric weight with appropriate growth at infinity.

A technical difficulty inherent to our approach is the fact that the functional
$E_\epsilon$ is not coercive, unless the perturbed vorticity $\tilde\eta$
satisfies some moment conditions. The problem comes from the inner region, where
the last term in \eqref{Edef} plays an important role. The results established
in \cite[Section~2]{GS3} indicate that $E_\epsilon$ is positive definite
provided $\tilde\eta$ has zero mean and vanishing first order moments with
respect to the space variables $R,Z$. In practice this means that, in addition
to the information provided by the energy $E_\epsilon$, we must control the
integral and the first order moments of the perturbed vorticity $\tilde\eta$. It
turns out that $\int\tilde \eta\dd R\dd Z$ is always extremely small, of the
order of $\cO(\exp(-c/\epsilon^2))$ for some $c > 0$. The radial moment
$\int\!R\,\tilde\eta\dd R\dd Z$ may take larger values, but can be controlled
using the conservation of the total impulse of the vortex ring. Finally, we
choose the correction $\dot{\tilde z}(t)$ of the vertical speed \eqref{barzpert}
in such a way that $\int\!Z\,\tilde\eta\dd R\dd Z = 0$, see Section~\ref{ssec41}
below for further details. This correction thus plays the role of a ``modulation
parameter'', see \cite{Wei,MM} for a similar idea in the context of the
stability analysis of solitary waves.

Disregarding these technical questions for the moment, we briefly indicate 
how the argument is concluded. If we differentiate $E_\epsilon$ with respect to
time, and use the evolution equation \eqref{tildeq} together with the estimate
\eqref{Remest} on the source term, we obtain after lengthy calculations a
differential inequality of the form
\begin{equation}\label{Ediffeq}
  t E_\epsilon'(t) \,\le\, -c_1 E_\epsilon(t) + c_2\Bigl(\epsilon^2 +
  \frac{\epsilon^{2\gamma_3}}{\delta^2}\Bigr)\,, \qquad t \in (0,T_\adv\delta^{-\sigma})\,,
\end{equation}
for some positive constants $c_1, c_2$. Here we assume that
$\epsilon^{2\gamma_3} \ll \delta^2$ so that the source term in \eqref{Ediffeq}
is small. Since $\epsilon^2 \lesssim \delta^{1-\sigma}$ by
Remark~\ref{crossrem}, this is the case if $\sigma < 1 - 2/\gamma_3$, which is
possible if $\sigma < 1/3$ and $\gamma_3$ is close enough to $3$.  Integrating
\eqref{Ediffeq} with initial condition $E_\epsilon(0) = 0$, we find
\begin{equation}\label{Einteq}
  E_\epsilon(t) \,\le\, c_3\Bigl(\epsilon^2 + \frac{\epsilon^{2\gamma_3}}{\delta^2}\Bigr)
  \,, \qquad t \in (0,T_\adv\delta^{-\sigma})\,,
\end{equation}
and using in addition the bounds on the moments of $\tilde\eta$ that are
obtained by a different argument we arrive at an estimate of the form
$\delta\|\tilde \eta(t)\|_{\cX_\epsilon}\le c(\epsilon\delta + \epsilon^{\gamma_3})$,
where $\cX_\epsilon$ is the weighted $L^2$ space equipped with the norm
\begin{equation}\label{Xepsdef}
  \|\tilde \eta\|_{\cX_\epsilon} \,=\, \biggl(\int_{\Omega_\epsilon} 
  W_\epsilon(R,Z)\,|\tilde\eta(R,Z)|^2\dd R\dd Z\biggr)^{1/2}\,.
\end{equation}
This space depends on time through the parameter $\epsilon > 0$, but we recall
that the weight function satisfies a uniform lower bound of the form $W_\epsilon(R,Z)
\gtrsim \exp(\rho^{2\gamma}/4)$, see Figure~\ref{fig2}.

The main result of this paper can now be formulated as follows\:

\begin{thm}\label{main2}
For any $\gamma_3 \in (2,3)$, there exist constants $K > 0$, $\delta_0 > 0$,
and $\sigma \in (0,\frac13)$ such that, for all $\Gamma > 0$, all $r_0 > 0$, and
all $\nu > 0$ satisfying $\delta := \nu/\Gamma \le \delta_0$,
the following holds. There exist continuous functions $\bar r(t), 
\bar z(t)$ which are smooth for positive times and satisfy \eqref{vmotion} with 
$\bar r(0) = r_0, \bar z(0) = 0$ such that the unique solution $\eta$ of \eqref{etaeq} 
with initial data \eqref{etazero} satisfies 
\begin{equation}\label{main2est}
  \|\eta(t) - \eta_*(t)\|_{\cX_\epsilon} \,\le\, K \bigl(\epsilon\delta + \epsilon^{\gamma_3}
  \bigr)\,, \qquad t \in (0,T_\adv\delta^{-\sigma})\,,
\end{equation}
where $\epsilon = \sqrt{\nu t}/{\bar r}(t)$ and $\eta_*$ is the approximate solution
defined by \eqref{etaapp}, \eqref{etaapp1}.
\end{thm}

We recall that estimate \eqref{vmotion} for the radial velocity $\dot{\bar r}$
implies that $\bar r(t) = r_0\bigl(1 + \cO(\epsilon^2)\bigr)$, meaning that
the major radius of the vortex ring remains essentially constant on the
time interval $(0,T_\adv\delta^{-\sigma})$. As for the vertical velocity, it is 
given by \eqref{barzpert} where the approximate speed $\dot{\bar z}_*$ defined 
in \eqref{rzapp} and the correction $\dot{\tilde z}$ satisfy
\begin{equation}\label{tzbdd}
  \dot{\bar z}_* \,=\, \frac{\Gamma}{4\pi r_0}\Bigl(\log \frac{1}{\epsilon}
  + \hat v\Bigr)\Bigl(1 + \cO\bigl(\epsilon^2)\Bigr)\,, \qquad
  \dot{\tilde z}  \,=\, \frac{\Gamma}{r_0}\,\cO\Bigl(\bigl(\epsilon +
  \frac{\epsilon^{\gamma_3}}{\delta}\bigr)\log\frac{1}{\epsilon} + \delta\Bigr)\,.
\end{equation}
This gives the announced formula \eqref{vmotion} for the full velocity 
$\dot{\bar z} = \dot{\bar z}_* + \delta\dot{\tilde z}$.

It is not difficult to verify that Theorem~\ref{main2} implies Theorem~\ref{main1},
see Section~\ref{ssec49} for details. Here we just show how to derive estimate
\eqref{fixed-time}, which is essentially a reformulation of \eqref{nseapp}. 
By construction, our approximate solution satisfies $\|\eta_*(t) - \eta_0\|_{\cX_\epsilon}
= \cO(\epsilon)$, where $\eta_0$ is the Gaussian function \eqref{etazero}. Moreover,
the lower bound $W_\epsilon(R,Z) \gtrsim \exp(\rho^{2\gamma}/4)$ implies that
$\cX_\epsilon\hookrightarrow L^1(\Omega_\epsilon)$ uniformly in $\epsilon$. It thus follows
from \eqref{main2est} that
\[
  \|\eta(t) - \eta_0\|_{L^1(\Omega_\epsilon)} \,\le\, C_1 \Bigl(\|\eta(t) - \eta_*(t)\|_{\cX_\epsilon}
  + \|\eta_*(t) - \eta_0\|_{\cX_\epsilon}\Bigr) \,\le\, C_2\epsilon\,,
\]
for any $t \in (0,T_\adv\delta^{-\sigma})$, and returning to the original variables we
arrive at estimate \eqref{fixed-time}.

\begin{rem}\label{expandrem}
It follows from \eqref{etaapp} and \eqref{main2est} that the solution of
\eqref{etaeq} satisfies
\begin{equation}\label{expand}
  \eta(R,Z,t) \,=\, \eta_0(R,Z) + \epsilon \eta_1(R,Z) + \epsilon^2 \eta_2(R,Z,\beta_\epsilon)
  + \cO\bigl(\delta\epsilon  + \epsilon^{\gamma_3}\bigr)\,,
\end{equation}
where the remainder term is estimated in the topology of $\cX_\epsilon$ as $\epsilon \to 0$.
Here $\eta_0$ is the Gaussian function \eqref{etazero}, and the vorticity profiles
$\eta_1, \eta_2$ are explicitly constructed in Section~\ref{sec3}. Since $\delta
\lesssim \epsilon^2$ except for very small times, see Remark~\ref{crossrem}, we
see that \eqref{expand} determines the shape of the vortex core up to third order
in $\epsilon$. 
\end{rem}

\section{Construction of the approximate solution}\label{sec3}

In this section we construct perturbatively an approximate solution of
\eqref{etaeq2} such that the corresponding remainder satisfies
\eqref{Remest}. Approximations of vortex rings with varying degrees
of accuracy were obtained by many authors, and typically rely on
matched asymptotics expansions where the inner core of the vortex and
the outer region are considered separately, see \cite{Kel,Hi,Dy,Fr1,Fr2}
in the inviscid case and \cite{TT,CT,FM} in the viscous case. 
Here we rather follow the direct approach introduced in \cite{Ga} for
interacting vortices in the plane, which does not rely on matched
asymptotics techniques. 

\subsection{Expansion of the Biot-Savart formula}\label{ssec31}

Our first task is to compute an accurate asymptotic expansion of 
the function $F(s)$ defined by \eqref{Fdef} in the limit where 
$s \to 0$. This can be done by expressing $F$ in terms of elliptic 
integrals, a procedure initiated in the early references 
\cite{Helmholtz,Max}.

\begin{lem}\label{Flem}
For $0 < s < 4$ we have the power series representation
\begin{equation}\label{Fexp}
  F(s) \,=\, \log\Bigl(\frac{8}{\sqrt{s}}\Bigr) \sum_{m=0}^\infty A_m s^m \,+\, 
  \sum_{m=0}^\infty B_m s^m\,, 
\end{equation}
where $A_m$, $B_m$ are real numbers. Moreover
\begin{equation}\label{Fcoeff}
  A_0 = 1\,, \quad A_1 = \frac{3}{16}\,, \quad A_2 = -\frac{15}{1024}\,, \quad
  B_0 = -2\,, \quad B_1 = -\frac{1}{16}\,, \quad B_2 = \frac{31}{2048}\,.
\end{equation}
\end{lem}

\begin{proof}
If $s > 0$ and $k = 2/\sqrt{s+4} \in (0,1)$, it is straightforward to 
verify that 
\begin{equation}\label{Fexp1}
  F(s) \,=\, \int_0^{\pi/2} \frac{1-2\sin^2\psi}{
  \sqrt{\sin^2\psi + s/4}}\dd\psi \,=\, \frac{2-k^2}{k}\,
  K(k) - \frac{2}{k}\,E(k)\,,
\end{equation}
where $K(k)$, $E(k)$ are the complete elliptic integrals with modulus
$k$\:
\[
  K(k) \,=\, \int_0^{\pi/2} \frac{1}{\sqrt{1-k^2 \sin^2\theta}}\dd \theta\,, 
  \qquad  E(k) \,=\, \int_0^{\pi/2} \sqrt{1-k^2 \sin^2\theta}\dd \theta\,.
\]
We are interested in the limit where $s \to 0$, namely $k \to 1$. Introducing
the complementary modulus $\kappa = \sqrt{1-k^2}$, we have 
the power series expansions (see \cite{Car})
\begin{equation}\label{KEexp}
  \begin{split}
  K(k) \,&=\, \sum_{m=0}^\infty a_m^2\,\kappa^{2m} \Bigl(\log\frac{1}{\kappa} + 
  2 b_m\Bigr)\,, \\
  E(k) \,&=\, 1 + \sum_{m=0}^\infty \frac{2m+1}{2m+2}\,a_m^2\,\kappa^{2m+2} 
  \Bigl(\log\frac{1}{\kappa} + b_m + b_{m+1}\Bigr)\,, 
  \end{split}
\end{equation}
where $a_0 = 1$, $b_0 = \log(2)$, and 
\[
  a_m \,=\, \frac{1}{2} \cdot \frac{3}{4}\cdot \,\ldots\, \cdot 
  \frac{2m{-}1}{2m}\,, \qquad b_m \,=\, \log(2) + \sum_{\ell=1}^{2m} 
  \frac{(-1)^\ell}{\ell}\,, \qquad m \in \N^*\,.
\]
Combining \eqref{Fexp1}, \eqref{KEexp}, we obtain a representation of the form
\begin{equation}\label{Fexp2}
  F(s) \,=\, \frac{1+\kappa^2}{\sqrt{1-\kappa^2}}\,K(k) - 
  \frac{2}{\sqrt{1-\kappa^2}}\,E(k) \,=\, \log\Bigl(\frac{4}{\kappa}\Bigr) 
  \sum_{m=0}^\infty C_m \kappa^{2m} \,+\, \sum_{m=0}^\infty D_m \kappa^{2m}\,,
\end{equation}
which converges for $0 < \kappa < 1$. Moreover, a direct calculation 
shows that
\begin{equation}\label{Fcoeff2}
  C_0 = 1\,, \quad C_1 = \frac{3}{4}\,, \quad C_2 = \frac{33}{64}\,, \quad
  D_0 = -2\,, \quad D_1 = -\frac{3}{4}\,, \quad D_2 = -\frac{81}{128}\,.
\end{equation}
As $\kappa^2 = s/(s+4)$, the right-hand side of \eqref{Fexp2} can 
be written in the form \eqref{Fexp}, and using \eqref{Fcoeff2} we see 
that the first coefficients satisfy \eqref{Fcoeff}. 
\end{proof}

\begin{rem}\label{phiexpand}
Various asymptotic expansions of the stream function given by the Biot-Savart
law \eqref{BSpsi} can be found in the literature \cite{Hi,Dy,Lam,TT,FM}, and are 
easily recovered using Lemma~\ref{Flem}.
\end{rem}

We next consider the rescaled Biot-Savart formula \eqref{BSeps}, which can
be written in the equivalent form
\begin{equation}\label{BSeps2}
  \phi(R,Z) \,=\, \frac{1}{2\pi}\int_{\Omega_\epsilon} K_\epsilon(R,Z;R',Z')
  \,\eta(R',Z')\dd R' \dd Z'\,,
\end{equation}
where
\begin{equation}\label{Kepsdef}
  K_\epsilon \,=\, \sqrt{(1{+}\epsilon R) (1{+}\epsilon R')} 
  ~F\biggl(\frac{\epsilon^2 D^2}{(1{+}\epsilon R)(1{+}\epsilon R')}\biggr)\,,
  \qquad D^2 \,=\, (R{-}R')^2 + (Z{-}Z')^2\,.
\end{equation}
To simplify the notations below, we define
\begin{equation}\label{betaLdef}
  \beta_\epsilon \,=\, \log\frac{1}{\epsilon}\,, \qquad L(R,Z;R',Z') \,=\, 
  \log\Bigl(\frac{8}{D}\Bigr)\,.
\end{equation}

\begin{lem}\label{Kexpansion}
For any $(R,Z)$, $(R',Z') \in \R^2$ with $(R,Z) \neq (R',Z')$ and any sufficiently 
small $\epsilon > 0$, we have the expansion
\begin{equation}\label{Kexp}
  K_\epsilon \,=\, (\beta_\epsilon + L) \sum_{m=0}^\infty \epsilon^m P_m \,+\, 
  \sum_{m=0}^\infty \epsilon^m Q_m\,,
\end{equation}
where $P_m(R,Z;R',Z')$, $Q_m(R,Z;R',Z')$ are homogeneous polynomials of 
degree $m$ in the three variables $R$, $R'$, and $Z-Z'$. Moreover
\begin{equation}\label{PQexp}
  \begin{array}{rcl} 
  P_0 \!&=&\! 1 \\[1mm]
  P_1 \!&=&\! \frac{1}{2}(R+R') \\[1mm]
  P_2 \!&=&\! \frac{1}{16}(R-R')^2 + \frac{3}{16}(Z-Z')^2
  \end{array} \qquad
  \begin{array}{rcl}
  Q_0 \!&=&\! -2 \\[1mm]
  Q_1 \!&=&\! -\frac{1}{2}(R+R') \\[1mm]
  Q_2 \!&=&\! \frac{1}{4}(R^2 + R'^2) - \frac{1}{16}\,D^2\,.
  \end{array}
\end{equation}
\end{lem}

\begin{proof}
If $(R,Z)$, $(R',Z')$ are as in the statement, we take $\epsilon > 0$ small enough 
so that 
\begin{equation}\label{sdef1}
   \max\bigl(|R|,|R'|\bigr) \,<\, \frac{1}{\epsilon}\,, \qquad \hbox{and} \qquad
   s \,:=\, \frac{\epsilon^2 D^2}{(1{+}\epsilon R)(1{+}\epsilon R')} \,<\, 4\,.
\end{equation}
As $D \neq 0$ by assumption, we have $0 < s < 4$, so that we can apply
expansion \eqref{Fexp} to the quantity $F(s)$ in \eqref{Kepsdef}. In 
view of definitions \eqref{betaLdef} we have
\begin{equation}\label{auxlog}
  \log\Bigl(\frac{8}{\sqrt{s}}\Bigr) \,=\, \beta_\epsilon + L + 
  \frac{1}{2}\log(1+\epsilon R) +  \frac{1}{2}\log(1+\epsilon R')\,. 
\end{equation}
We observe that the last two terms in \eqref{auxlog}, as well as the prefactor
$\sqrt{(1{+}\epsilon R) (1{+}\epsilon R')}$ in \eqref{Kepsdef} and each monomial
$s^m$ in the series \eqref{Fexp}, can be expanded into a power series in the
three variables $\epsilon R$, $\epsilon R'$, and $\epsilon (Z-Z')$. Thus,
combining \eqref{Fexp} and \eqref{Kepsdef}, we obtain a representation of the
form \eqref{Kexp}, where the first homogeneous polynomials $P_m$, $Q_m$ are
easily computed using the explicit values \eqref{Fcoeff}.
\end{proof}

\begin{rem}\label{kernelrem}
In what follows, with a slight abuse of notation, we denote by $L$ the
integral operator on $\R^2$ given by the kernel \eqref{betaLdef}. For any 
continuous and rapidly decreasing function $\eta : \R^2 \to \R$, 
we thus have
\begin{equation}\label{Lopdef}
  \bigl(L \eta\bigr)(R,Z) \,=\, \int_{\R^2} \log\biggl(\frac{8}{\sqrt{(R{-}R')^2 + 
  (Z{-}Z')^2}}\biggr)\,\eta(R',Z')\dd R' \dd Z'\,.
\end{equation}
Similarly, we associate integral operators to the homogeneous polynomials 
$P_m$, $Q_m$ in \eqref{Kexp}, and to the functions $L P_m$ for all $m \in \N^*$. 
\end{rem}

\begin{df} Using the notation introduced in Remark~\ref{kernelrem}, we define
the linear operators
\begin{equation}\label{BSmdef}
  \BS_0 \,=\, \frac{1}{2\pi}\,L\,, \quad \hbox{and} \quad 
  \BS_m \,=\, \frac{1}{2\pi}\Bigl(\beta_\epsilon P_m  + L P_m + Q_m\Bigr)\,,
  \quad \hbox{for all } m \in \N^*\,.
\end{equation}
\end{df}

Note that, for $m \ge 1$, the linear operator $\BS_m$ depends on the parameter
$\epsilon$ through the constant factor $\beta_\epsilon = \log(1/\epsilon)$, but
for simplicity this mild dependence is not indicated explicitly. For
convenience, we do not include the constant term
$\beta_\epsilon P_0 + Q_0 \equiv \beta_\epsilon - 2$ in the definition of
$\BS_0$, because the stream function is only defined up to an additive
constant. It is important to observe that, in \eqref{Lopdef} and in the
corresponding definition of the integral operators $P_m$, $Q_m$, and $L P_m$,
the integration is performed on the whole plane $\R^2$, rather than on the
half-plane $\Omega_\epsilon$.  This is justified because these operators will
always be applied to functions that decay rapidly at infinity, so that the
integration on $\R^2 \setminus \Omega_\epsilon$ gives a contribution of
order $\cO(\epsilon^\infty)$ as $\epsilon \to 0$, which can be neglected in our
perturbative expansion. If $\eta : \R^2 \to \R$ is compactly supported, then
according to Lemma~\ref{Kexpansion} the following equality holds in any bounded
region of $\R^2$:
\begin{equation}\label{BSexp}
  \BS^\epsilon[\eta] \,=\, \frac{\beta_\epsilon - 2}{2\pi} \int_{\R^2} \eta(R',Z') 
  \dd R' \dd Z' \,+\, \sum_{m = 0}^\infty \epsilon^m\,\BS_m[\eta]\,,
\end{equation}
provided $\epsilon > 0$ is sufficiently small. As was already mentioned, the first 
term in the right-hand side of \eqref{BSexp} is a constant that can be omitted. 

\subsection{Function spaces and linear operators}\label{ssec32}

We next introduce the function spaces in which we shall construct our approximate
solution of \eqref{etaeq2}. These spaces consist of functions that are defined
on the whole space $\R^2$, and not just on the half-plane $\Omega_\epsilon$. 
Indeed, at each step of the approximation, the vorticity profile
$\eta_m(R,Z,\beta_\epsilon)$ and the stream function $\phi_m(R,Z,\beta_\epsilon)$ in
\eqref{etaapp} are defined for all $(R,Z) \in \R^2$. To simplify the writing we often
denote $X = (R,Z)$, and we use polar coordinates $(\rho,\vt)$ in $\R^2$ defined by
the relations $R = \rho\cos\vt$, $Z = \rho\sin\vt$. 

Following \cite{GW1,GW2} we introduce the weighted $L^2$ space
\begin{equation}\label{Ydef}
  \cY \,=\, \Bigl\{\eta \in L^2(\R^2) \,\Big|\, \int_{\R^2}
  |\eta(X)|^2 \,e^{|X|^2/4}\dd X < \infty\Bigr\}\,,
\end{equation}
equipped with the scalar product $(\eta_1,\eta_2)_\cY = \int_{\R^2} \eta_1(X) 
\eta_2(X) \,e^{|X|^2/4}\dd X$ and the associated norm. We also introduce the
differential operator $\cL : D(\cL) \to \cY$ corresponding to \eqref{cLdef}, 
namely
\begin{equation}\label{Ldef}
  \cL \eta \,=\, \Delta \eta + \frac{1}{2}\,X \cdot \nabla\eta 
  + \eta\,, \qquad \eta \in D(\cL) \,=\, \Bigl\{\eta \in \cY\,\Big|\, 
  \Delta \eta \in \cY\,, ~X\cdot\nabla \eta \in \cY\Bigr\}\,,
\end{equation}
as well as the integro-differential operator $\Lambda : D(\Lambda) \to \cY$
defined by 
\begin{equation}\label{Lamdef}
  \Lambda \eta \,=\, \frac{1}{2\pi}\Bigl(\bigl\{L \eta_0\,,\eta\bigr\} + 
  \bigl\{L\eta\,,\eta_0\bigr\}\Bigr)\,, \qquad \eta \in D(\Lambda) \,=\, 
  \Bigl\{\eta \in \cY\,\Big|\, \bigl\{L\eta_0\,,\eta\bigr\} \in 
  \cY\Bigr\}\,,
\end{equation}
where $\eta_0$ is the Gaussian function \eqref{etazero} and $L$ denotes the
integral operator \eqref{Lopdef}. Here and in what follows the Poisson bracket
is understood with respect to the rescaled variables $(R,Z)$, so that
$\{\phi,\eta\} = \partial_R \phi \,\partial_Z \eta - \partial_Z \phi
\,\partial_R \eta$. We recall the following well-known properties\:

\begin{prop}\label{LLamprop} {\bf\cite{GW1,GW2,Ma1}}\\[0.5mm]
1) The linear operator $\cL$ is {\em self-adjoint} in $\cY$, with purely 
discrete spectrum
\[
  \sigma(\cL) \,=\, \Bigl\{-\frac{n}{2}\,\Big|\, n = 0,1,2, \dots\Bigr\}\,.
\]
The kernel of $\cL$ is one-dimensional and spanned by the Gaussian function $\eta_0$. 
More generally, for any $n \in \N$, the eigenspace corresponding to the 
eigenvalue $\lambda_n = -n/2$ is spanned by the $n+1$ Hermite functions
$\partial^\alpha \eta_0$ where $\alpha = (\alpha_1,\alpha_2) \in \N^2$ and 
$\alpha_1 + \alpha_2 = n$. \\[1mm]
2) The linear operator $\Lambda$ is {\em skew-adjoint} in $\cY$, so that 
$\Lambda^* = -\Lambda$. Moreover,
\begin{equation}\label{KerLam}
  \Ker(\Lambda) \,=\, \cY_0 \oplus \bigl\{\beta_1\partial_R \eta_0 + \beta_2
  \partial_Z \eta_0 \,\big|\, \beta_1, \beta_2 \in \R\bigr\}\,,
\end{equation}
where $\cY_0 \subset \cY$ is the subspace of all radially symmetric 
elements of $\cY$. 
\end{prop}

A crucial observation is that both operators $\cL$, $\Lambda$ are 
invariant under rotations about the origin in $\R^2$. It is therefore
advantageous to decompose the space $\cY$ into a direct sum
\begin{equation}\label{Yndef}
  \cY \,=\, \mathop{\oplus}\limits_{n=0}^\infty \cY_n\,, 
\end{equation}
where $\cY_0 \subset \cY$ is as in Proposition~\ref{LLamprop} and, for 
all $n \ge 1$, the subspace $\cY_n \subset \cY$ consists of all functions 
$\eta \in \cY$ such that $\eta(\rho\cos\vt,\rho\sin\vt) = a_1(\rho)\cos(n\vt) + 
a_2(\rho)\sin(n\vt)$ for some $a_1,a_2 : \R_+ \to \R$. It is clear that 
$\cY_n \perp \cY_{n'}$ if $n \neq n'$. In particular, in view of 
\eqref{KerLam}, we have $\cY_n \in \Ker(\Lambda)^\perp$ for all $n \ge 2$. 
When $n = 1$, the functions $\partial_R \eta_0$, $\partial_Z \eta_0$ belong to 
$\cY_1 \cap \Ker(\Lambda)$, and we define
\begin{equation}\label{Y1def}
  \cY_1' \,=\, \cY_1 \cap \Ker(\Lambda)^\perp \,=\, \biggl\{\eta \in \cY_1\,\bigg|\, 
  \int_{\R^2} \eta(R,Z) R \dd R\dd Z = \int_{\R^2} \eta(R,Z) Z\dd R\dd Z = 0
  \biggr\}\,.
\end{equation}

Since $\Lambda$ is skew-adjoint, we know that $\Ker(\Lambda)^\perp = \overline{\Ran
(\Lambda)}$, but the image of $\Lambda$ is not dense in $\cY$ and, therefore,
we cannot solve the equation $\Lambda \eta = f$ for any
$f \in \Ker(\Lambda)^\perp$.  As is shown in \cite{GW3,Ga}, the problem
disappears if one assumes in addition that $f$ belongs to a smaller space such
as
\begin{equation}\label{Zdef}
  \cZ \,=\, \Bigl\{\eta : \R^2 \to \R\,\Big|\, e^{|X|^2/4} \eta \in 
  \cS_*(\R^2)\Bigr\} \,\subset\, \cY\,,
\end{equation}
where $\cS_*(\R^2)$ denotes the space of all smooth functions which are slowly
growing at infinity. More precisely, a smooth function $w : \R^2 \to \R$ belongs
to $\cS_*(\R^2)$ if, for any $\alpha = (\alpha_1,\alpha_2) \in \N^2$,
there exist $C > 0$ and $N \in \N$ such that $|\partial^\alpha w(X)| \le
C(1+|X|)^N$ for all $X \in \R^2$.

\begin{rem}\label{rem:topo}
We do not need to specify the topology of the space $\cZ$, but the following
notation will be useful. If $f \in \cZ$ depends on a small parameter
$\epsilon > 0$, we say that $f = \cO(\epsilon)$ in $\cZ$ if, for any $\alpha =
(\alpha_1,\alpha_2) \in \N^2$, there exist $C > 0$ and $N \in \N$ such that
$|\partial^\alpha f(X)| \le C\epsilon (1+|X|)^N e^{-|X|^2/4}$ for all $X \in \R^2$.
\end{rem}

To formulate the main technical result of this section, we introduce the notation
\begin{equation}\label{phihdef}
  \vf(\rho) \,=\, \frac{1}{2\pi \rho^2}\bigl(1 - e^{-\rho^2/4}\bigr)\,, \qquad 
  h(\rho) \,=\, \frac{\rho^2/4}{e^{\rho^2/4}-1}\,, \qquad \rho > 0~.
\end{equation}
The following statement is a slight extension of \cite[Lemma~4]{Ga}. For the
reader's convenience, we give a short proof of it in Section~\ref{ssecA1}, 
emphasizing the case $n = 1$ which was not treated in \cite{Ga}.

\begin{prop}\label{Laminv}
If $n \ge 2$ and $f \in \cY_n \cap \cZ$, or if $n = 1$ and  $f \in \cY'_1 \cap \cZ$, 
there exists a unique $\eta \in \cY_n \cap \cZ$ (respectively, a unique $\eta \in \cY_1' 
\cap \cZ$) such that $\Lambda \eta = f$.  In particular, if $f = a(\rho)\sin(n\vt)$, 
then $\eta = \omega(\rho) \cos(n\vt)$, where
\begin{equation}\label{omexp}
  \omega(\rho) \,=\, h(\rho)\Omega(\rho) + \frac{a(\rho)}{n\vf(\rho)}\,, 
  \qquad \rho > 0~,
\end{equation}
and where $\Omega : (0,\infty) \to \R$ is the unique solution of the differential
equation
\begin{equation}\label{Omexp}
  -\Omega''(\rho) -\frac{1}{\rho}\,\Omega'(\rho) + \Bigl(\frac{n^2}{\rho^2} - h(\rho)\Bigr)
  \Omega(\rho) \,=\, \frac{a(\rho)}{n\vf(\rho)}\,, \qquad \rho > 0\,,  
\end{equation}
such that $\Omega(\rho) = \cO(\rho^n)$ as $\rho \to 0$ and $\Omega(\rho) = 
\cO(\rho^{-n})$ as $\rho \to \infty$. 
\end{prop}

\begin{rem}\label{sincos}
As was observed in \cite{Ga}, if $f = a(\rho)\cos(n\vt)$, then $\eta = -\omega(\rho) 
\sin(n\vt)$, where $\omega$ is still given by \eqref{omexp}, \eqref{Omexp}. The 
general case where $f = a_1(\rho)\cos(n\vt) + a_2(\rho)\sin(n\vt)$ follows 
by linearity. 
\end{rem}

In the construction of an approximate solution of \eqref{etaeq2}, we shall
encounter linear equations of the form
\begin{equation}\label{deltaLam}
  \delta \bigl(\kappa - \cL\bigr)\eta^{\delta} + \Lambda \eta^{\delta} \,=\, f\,,
\end{equation}
where $\kappa > 0$ is fixed and the parameter $\delta > 0$ can be arbitrarily small.
Proposition~\ref{LLamprop} implies that the linear operator $\delta(\kappa - \cL) +
\Lambda$ is invertible in $\cY$, so that \eqref{deltaLam} has a unique solution 
$\eta^{\delta}$ for any $f \in \cY$. In general, the best estimate we
can hope for is
\begin{equation}\label{resolbd}
  \|\eta^\delta\|_\cY \,=\, \bigl\|\bigl(\delta(\kappa - \cL) + \Lambda\bigr)^{-1}
  f\bigr\|_{\cY} \,\le\, \frac{1}{\kappa \delta}\,\|f\|_{\cY}\,.
\end{equation}
However, if $f$ satisfies the assumptions of Proposition~\ref{Laminv}, the solution
$\eta^\delta$ admits a regular expansion in powers of the small parameter $\delta$.
More precisely\:

\begin{prop}\label{Laminv2}
Assume that $n \ge 2$ and $f \in \cY_n \cap \cZ$, or that $n = 1$ and $f \in
\cY'_1 \cap \cZ$. For any fixed $\kappa > 0$ and any $\delta > 0$, equation
\eqref{deltaLam} has a unique solution $\eta^{\delta} \in \cY_n$ (respectively,
$\eta^{\delta} \in \cY_1'$). Moreover, for each nonzero $N \in \N$, there exists
a constant $C > 0$, depending only on $f$ and $N$, such that
\begin{equation}\label{deltacomp}
  \Bigl\|\eta^{\delta} - \sum_{m=0}^{N-1} \delta^m \hat\eta_m\Bigr\|_\cY \,\le\, C\delta^N\,,
\end{equation}
where the profiles $\hat\eta_m \in \cY_n \cap \cZ$ (respectively, $\hat\eta_m \in
\cY_1' \cap \cZ$) are determined by the relations $\Lambda \hat\eta_0 = f$ and
$\Lambda \hat\eta_m = (\cL - \kappa)\hat\eta_{m-1}$ for $m \ge 1$.
\end{prop}

\begin{proof}
Assume first that $n \ge 2$. Since the space $\cY_n$ is invariant under the action
of both operators $\cL$ and $\Lambda$, it is clear that $\eta^\delta \in \cY_n$ if
$f \in \cY_n$. If we suppose in addition that $f \in \cZ$, Proposition~\ref{Laminv} shows
that there is a unique $\hat\eta_0 \in \cY_n \cap \cZ$ such that $\Lambda\hat\eta_0 = f$.
A direct calculation then shows that $(\cL-\kappa)\hat\eta_0 \in \cY_n \cap \cZ$, so
that we can define $\hat\eta_1 \in \cY_n \cap \cZ$ as the unique solution of $\Lambda
\hat\eta_1 = (\cL - \kappa)\hat\eta_0$. Repeating this procedure, we construct the 
profiles $\hat\eta_m$ for $m = 0,\dots,N$, and we define 
$\tilde \eta = \eta^{\delta}
- \bigl(\hat\eta_0 + \delta \hat\eta_1 + \dots + \delta^N \hat\eta_N\bigr)$, so that
\begin{equation}\label{deltadiff}
  \Bigl(\delta (\kappa - \cL) + \Lambda\Bigr)\tilde \eta \,=\, f -
  \Bigl(\delta \bigl(\kappa - \cL\bigr) + \Lambda\Bigr) \sum_{m=0}^N \delta^m
  \hat\eta_m \,=\, \delta^{N+1} \bigl(\cL-\kappa\bigr)\hat\eta_N\,.
\end{equation}
Estimate \eqref{resolbd} then gives the crude bound $\|\tilde\eta\|_\cY \le C\delta^N$, 
which nevertheless implies \eqref{deltacomp}. The proof is identical if $n = 1$ and 
$f \in \cY_1' \cap \cZ$. 
\end{proof}

\subsection{First order approximation}\label{ssec33}

We now begin the construction of an approximate solution of \eqref{etaeq2} in
the form \eqref{etaapp}, \eqref{rzapp}. We recall that, for an exact solution,
the stream function is determined by the relation \eqref{BSeps}, which we
write in the compact form $\phi = \BS^\epsilon[\eta]$. For our approximate
solution, we expand the Biot-Savart operator as in \eqref{BSexp}, omitting
the constant term in the right-hand side. We thus obtain the formal relation
\[
  \sum_{m=0}^\infty \epsilon^m\,\BS_m\biggl[\,\sum_{m=0}^M 
  \epsilon^m \eta_m\biggr] \,=\, \sum_{m=0}^M \epsilon^m \phi_m + 
  \cO\bigl(\epsilon^{M+1}\bigr)\,,
\]
which we assume to be satisfied order by order in $\epsilon$, up to order $M$. 
This leads to the relations $\phi_0 = \BS_0[\eta_0]$, $\phi_1 = 
\BS_0[\eta_1] + \BS_1[\eta_0]$, and more generally
\begin{equation}\label{BScompat}
  \phi_m \,=\, \BS_0[\eta_m] + \BS_1[\eta_{m-1}] + \dots + 
  \BS_{m-1}[\eta_1] + \BS_m[\eta_0]\,.
\end{equation}
In particular, in view of \eqref{etazero} and \eqref{BSmdef}, the leading order
of our approximation is
\begin{equation}\label{phi0def}
  \eta_0(R,Z) \,=\, \frac{1}{4\pi}\,e^{-(R^2 + Z^2)/4}\,, \qquad
  \phi_0(R,Z) \,=\, \frac{1}{2\pi}\bigl(L \eta_0\bigr)(R,Z)\,,
\end{equation}
where $L$ is the integral operator \eqref{Lopdef}. The stream function
$\phi_0$ has the expression
\begin{equation}\label{phi0exp}
  \phi_0(R,Z) \,=\, \phi_0(0) - \frac{1}{4\pi}\Ein\biggl(\frac{R^2{+}Z^2}{4}
  \biggr)\,, \qquad \hbox{where}\quad \Ein(x) \,=\, \int_0^x
  \frac{1-e^{-t}}{t}\dd t\,,
\end{equation}
so that $\phi_0$ is radially symmetric and $\phi_0(R,Z) \sim -(2\pi)^{-1}
\log\rho$ as $\rho := (R^2+Z^2)^{1/2} \to +\infty$. The value at the
origin does not play a big role in our analysis, but can be
computed too, see Section~\ref{ssecA2}\:
\[
  \phi_0(0) \,=\, \frac{\log(2)}{\pi} + \frac{\gamma_E}{4\pi}\,,
  \qquad \hbox{where } \gamma_E\,\hbox{ is Euler's constant.}
\]

Before proceeding further, we estimate the time derivative of the
quantity $\epsilon = \sqrt{\nu t}/\bar r(t)$ introduced in \eqref{RZeps}.
In view of \eqref{rzapp}, we have
\begin{equation}\label{dereps}
  t\dot \epsilon \,=\, \frac{\epsilon}{2} - \frac{\epsilon t \dot{\bar r}}{\bar r}
  \,=\, \frac{\epsilon}{2} - \frac{\epsilon t}{\bar r} \sum_{m=0}^{M-1} \epsilon^m
  \,\dot{\bar r}_m\,.
\end{equation}
At this stage the radial velocity profiles $\dot{\bar r}_m$ are not
determined yet, but in view of Remark~\ref{remEuler} we can anticipate
the fact that $|\dot{\bar r}| = (\Gamma/r_0)\cdot\cO(\delta)$ as $\delta \to 0$.
Since $\delta t = (r_0^2/\Gamma)\cdot\cO(\epsilon^2)$ by Remark~\ref{crossrem}, 
it follows that $\bar r(t) = r_0\bigl(1 + \cO(\epsilon^2)\bigr)$ and that $t\dot
\epsilon = \epsilon/2 + \cO(\epsilon^3)$ as $\epsilon \to 0$. 

With that observation in mind, we substitute the expansions \eqref{etaapp},
\eqref{rzapp} into the evolution equation \eqref{etaeq2}, keeping only the
terms that are exactly of order $\epsilon$ or $\epsilon \beta_\epsilon$.
This gives the relation
\begin{equation}\label{approx1}
  \bigl\{\phi_1\,,\eta_0\bigr\} + \bigl\{\phi_0\,,\eta_1\bigr\} + \eta_0
  \partial_Z \phi_0 -\frac{{r_0}}{\Gamma}\Bigl(\dot{\bar r}_0\,\partial_R 
  \eta_0 + \dot{\bar z}_0 \,\partial_Z \eta_0\Bigr) \,=\, \delta \Bigl[
  \partial_R \eta_0 + \bigl(\cL - {\TS\frac12}\bigr) \eta_1 -t\partial_t
  \eta_1\Bigr]\,.
\end{equation}
To solve \eqref{approx1} we first impose the relation  
\begin{equation}\label{dotr0}
  \dot{\bar r}_0 \,=\, -\frac{\Gamma \delta}{r_0}\,,
\end{equation}
which ensures that the terms involving $\partial_R \eta_0$ cancel exactly.
We also assume that $\eta_1$ does not depend on $\beta_\epsilon$, so
that $\partial_t \eta_1 = 0$ (this property will be verified later). On the
other hand, from \eqref{BScompat} with $m = 1$ we deduce that $\{\phi_1\,,\eta_0\}
= \{\BS_0[\eta_1]\,,\eta_0\} + \{\BS_1[\eta_0]\,,\eta_0\}$, where $\BS_0$, $\BS_1$
are defined in \eqref{BSmdef}. Using \eqref{phi0def} and the definition \eqref{Lamdef} 
of the linear operator $\Lambda$, we thus find
\begin{align*}
  \bigl\{\phi_1\,,\eta_0\bigr\} + \bigl\{\phi_0\,,\eta_1\bigr\} \,&=\, 
  \frac{1}{2\pi}\Bigl( \bigl\{L \eta_1\,,\eta_0\bigr\} + 
  \bigl\{L \eta_0\,,\eta_1\bigr\}\Bigr) + \bigl\{\BS_1[\eta_0]\,,\eta_0\bigr\} \\
  \,&=\, \Lambda \eta_1 + \frac{\beta_\epsilon -1}{2\pi}\,\bigl\{P_1 \eta_0\,,
  \eta_0\bigr\} + \frac{1}{2\pi}\bigl\{L P_1 \eta_0\,,\eta_0\bigr\}\,,
\end{align*}
where in the second line we used the definition \eqref{BSmdef} of $\BS_1$
and the fact that $Q_1 = -P_1$ in view of \eqref{PQexp}. Now, elementary 
calculations that are reproduced in Section~\ref{ssecA2} show that
\begin{equation}\label{LP1}
  \bigl\{P_1 \eta_0\,,\eta_0\bigr\} \,=\, \frac12\, \partial_Z \eta_0\,,
  \quad \hbox{and}\quad 
  \frac{1}{2\pi}\bigl\{L P_1 \eta_0\,,\eta_0\bigr\} \,=\, \frac{1}{2}
  \,\partial_Z \bigl(\phi_0\eta_0\bigr)\,. 
\end{equation}
It follows that we can write \eqref{approx1} in the equivalent form
\begin{equation}\label{eta1eq}
  \Lambda \eta_1 + \delta \bigl({\TS\frac12} - \cL\bigr) \eta_1 \,=\, 
  \Bigl(\frac{{r_0}}{\Gamma}\,\dot{\bar z}_0 - \frac{\beta_\epsilon 
  - 1}{4\pi}\Bigr)\partial_Z \eta_0 - \frac{3}{2}\,(\partial_Z \phi_0)\eta_0
  - \frac{1}{2}\,\phi_0\partial_Z \eta_0\,.
\end{equation}

Using the explicit expressions \eqref{phi0def}, \eqref{phi0exp} of the
profiles $\eta_0, \phi_0$, it is straightforward to verify that the right-hand
side of \eqref{eta1eq}, which we denote by $-\cR_1$, belongs to $\cY_1 \cap \cZ$,
where $\cY_1$, $\cZ$ are the function spaces defined in \eqref{Yndef},
\eqref{Zdef}. Therefore, according to Proposition~\ref{Laminv2}, the linear
equation \eqref{eta1eq} has a unique solution $\eta_1 \in \cY_1$ for any
$\delta > 0$, and that solution has a well-defined limit as $\delta \to 0$ if and
only if $\cR_1 \in (\ker \Lambda)^\perp$, namely if $\cR_1 \in \cY_1'$. In view of
\eqref{Y1def}, this gives the solvability condition $\int_{\R^2} \cR_1 Z\dd R
\dd Z = 0$, which determines uniquely the value of the constant $\dot{\bar z}_0$
in \eqref{eta1eq}. The calculations are reproduced in Section~\ref{ssecA2}, and
yield the following expression of the vertical velocity to leading order\:
\begin{equation}\label{dotz0}
  \dot{\bar z}_0 \,=\, \frac{\Gamma}{4\pi r_0}\Bigl(\beta_\epsilon 
  -1 + 2v\Bigr)\,, \qquad \hbox{where}\quad
  v \,=\, \frac{3}{4}\,\log(2) + \frac{1}{4}\,\gamma_E + \frac{1}{4}\,.
\end{equation}
Here again $\gamma_E = 0,5772\dots$ denotes Euler's constant.  

\begin{rem}\label{speedrem}
The formula \eqref{dotz0}, including the leading term $\beta_\epsilon 
= \log(1/\epsilon)$ and the correct value of the constant $2v -1$, was 
established by Saffman \cite{Sa}, see also Fukumoto \& Moffatt \cite{FM}. 
\end{rem}

We assume henceforth that $\dot{\bar z}_0$ is given by \eqref{dotz0}, so
that \eqref{eta1eq} reduces to
\begin{equation}\label{eta1eq2}
  \Lambda \eta_1 + \delta \bigl({\TS\frac12} - \cL\bigr) \eta_1 \,=\, 
  \frac{v}{2\pi}\,\partial_Z \eta_0 - \frac{3}{2}\,(\partial_Z \phi_0)\eta_0
  - \frac{1}{2}\,\phi_0\partial_Z \eta_0\,,
\end{equation}
where the right-hand side $-\cR_1$ now belongs to $\cY_1' \cap \cZ$ and is
independent of $\epsilon$. Equation \eqref{eta1eq2} is of the form \eqref{deltaLam},
and can be solved using Proposition~\ref{Laminv2}. For our purposes, it is
sufficient to consider the {\em approximate} solution corresponding to the
choice $N = 2$ in \eqref{deltacomp}, which reads
\begin{equation}\label{eta1exp}
  \eta_1(R,Z) \,=\, R\,\eta_{10}(\rho) + \delta Z\,\eta_{11}(\rho)\,,
  \qquad  \rho = \sqrt{R^2 + Z^2}\,,
\end{equation}
where $\Lambda (R\,\eta_{10}) = -\cR_1$ and $\Lambda(Z\,\eta_{11}) = (\cL - \frac12)
(R\,\eta_{10})$. Note that $\eta_1 \in \cY' \cap \cZ$, which implies
in particular that the functions $\eta_{10}$, $\eta_{11}$ are smooth
and have a Gaussian decay at infinity. The corresponding stream
function $\phi_1 = \BS_0[\eta_1] + \BS_1[\eta_0]$ is computed in
Section~\ref{ssecA2} and takes the form
\begin{equation}\label{phi1exp}
  \phi_1(R,Z,\beta_\epsilon) \,=\, \frac{\beta_\epsilon - 1}{4\pi}\,R + 
  \frac{R}{2}\,\phi_0 - \partial_R \phi_0 + R\,\phi_{10}(\rho) + 
  \delta Z\,\phi_{11}(\rho)\,,  
\end{equation}
where $R\,\phi_{10} = \BS_0[R\,\eta_{10}]$ and $Z\,\phi_{11} = \BS_0[Z\,\eta_{11}]$.
One can check that the functions $\phi_{10}, \phi_{11}$ are smooth and decay at least like
$1/\rho^2$ as $\rho \to +\infty$. Note that $\phi_1$ involves the time-dependent
term $\beta_\epsilon = \log(1/\epsilon)$, so that $\partial_t \phi_1 \neq
0$. With the choices \eqref{dotr0}, \eqref{dotz0}, \eqref{eta1exp}, and
\eqref{phi1exp}, the relation \eqref{approx1} is not satisfied exactly, but the
difference of both members is $\cO(\delta^2)$ in the topology of $\cZ$, which is
all we need.

\subsection{Second order approximation}\label{ssec34}

We next compute the second order terms in the asymptotic expansion 
\eqref{etaapp}. As we shall see, it is consistent at this stage to 
take 
\begin{equation}\label{dotr1z1}
  \dot{\bar r}_1 \,=\, \dot{\bar z}_1 \,=\, 0\,,
\end{equation}
so we make that assumption from now on. As before, we deduce from \eqref{dereps}, 
\eqref{dotr0}, \eqref{dotr1z1} that $\bar r(t) = r_0\bigl(1 + \cO(\epsilon^2)\bigr)$
and $t\dot \epsilon = \epsilon/2 + \cO(\epsilon^3)$ as $\epsilon \to 0$. 
Substituting \eqref{etaapp}, \eqref{rzapp} into \eqref{etaeq2} and keeping 
only the terms involving $\epsilon^2$ or $\epsilon^2 \beta_\epsilon$, we
obtain the relation
\begin{equation}\label{approx2a}
\begin{split}
  \bigl\{\phi_2\,,\eta_0\bigr\} + \bigl\{\phi_1\,,\eta_1 - R\eta_0\bigr\} 
  + \bigl\{\phi_0\,,\eta_2 &- R\eta_1 + R^2 \eta_0\bigr\} 
  - \frac{r_0}{\Gamma}\Bigl(\dot{\bar r}_0\,\partial_R \eta_1 + 
   \dot{\bar z}_0\,\partial_Z \eta_1\Bigr) \\
  \,&=\, \delta \Bigl[\bigl(\cL - 1\bigr) \eta_2 + \partial_R (\eta_1 - R\eta_0)
  - t\partial_t \eta_2\Bigr]\,.
\end{split}
\end{equation}
In view of \eqref{dotr0}, the terms involving $\partial_R \eta_1$ cancel 
exactly. Moreover, we know from \eqref{BSmdef}, \eqref{BScompat} that
\begin{equation}\label{phi2def}
  \phi_2 \,=\, \frac{1}{2\pi} \Bigl(L \eta_2 + \bigl(\beta_\epsilon P_1 + LP_1 
  + Q_1\bigr)\eta_1 + \bigl(\beta_\epsilon P_2 + LP_2 + Q_2\bigr)\eta_0\Bigr)\,,
\end{equation}
where the notations are introduced in Lemma~\ref{Kexpansion}. Recalling the
definition \eqref{Lamdef} of the operator $\Lambda$, we can thus write 
\eqref{approx2a} in the equivalent form
\begin{equation}\label{approx2b}
  \Lambda \eta_2 + \delta \Bigl(t\partial_t \eta_2 + \bigl(1-\cL\bigr) \eta_2 \Bigr)
  + \cR_2 \,=\, 0\,,
\end{equation}
where
\begin{equation}\label{R2def}
\begin{split}
  \cR_2 \,=\, &\frac{1}{2\pi}\bigl\{(\beta_\epsilon - 1)P_1 \eta_1 + LP_1 \eta_1
  \,, \eta_0\bigr\} + \frac{1}{2\pi}\bigl\{\beta_\epsilon P_2\eta_0 + LP_2 \eta_0 
  + Q_2 \eta_0\,, \eta_0\bigr\} \\
  &+ \bigl\{\phi_1\,,\eta_1\bigr\} + (\partial_Z \phi_1)\eta_0 + (\partial_Z \phi_0)
  \eta_1 - R \Bigl(\bigl\{\phi_1\,,\eta_0\bigr\} + \bigl\{\phi_0\,,\eta_1\bigr\} + 
  2 (\partial_Z \phi_0)\eta_0 \Bigr) \\
  &+ \delta \partial_R(R \eta_0) - \frac{r_0 \dot{\bar z}_0}{\Gamma}
  \,\partial_Z \eta_1\,.
\end{split}
\end{equation}

We have the following result, whose proof is postponed to Section~\ref{ssecA3}\:

\begin{lem}\label{R2lem}
The function $\cR_2$ defined in \eqref{R2def} belongs to $(\delta\cY_0 + \cY_2) \cap 
\cZ$ and satisfies
\begin{equation}\label{R2exp}
  \cR_2 \,=\, \frac{3\beta_\epsilon}{16\pi}\,RZ \eta_0 + RZ\chi_{20} + 
  \delta \Bigl(\chi_{21} + (R^2 - Z^2)\chi_{22}\Bigr) + \delta^2RZ\chi_{23}\,,
\end {equation}
for some (time-independent) radially symmetric functions $\chi_{20}, \chi_{21},
\chi_{22}, \chi_{23} \in \cY_0 \cap \cZ$. 
\end{lem}

In view of \eqref{R2exp}, we look for a solution of \eqref{approx2b}
in the form $\eta_2 = \beta_\epsilon \hat \eta_{20} + \hat\eta_{21} + 
\hat\eta_{22}$, where $\hat\eta_{20}, \hat\eta_{21} \in \cY_2$ and 
$\hat\eta_{22} \in \cY_0$ do not depend on $\beta_\epsilon$. Inserting
this ansatz into \eqref{approx2b} and using the fact that $t\partial_t
\beta_\epsilon = -1/2 + \cO(\epsilon^2)$, we obtain the system
\begin{equation}\label{eta2sys}
\begin{split}
  &\Lambda \hat\eta_{20} + \delta\bigl(1-\cL\bigr) \hat\eta_{20} + 
   \frac{3}{16\pi}\,RZ \eta_0 \,=\, 0\,, \\
  &\Lambda \hat\eta_{21} + \delta\bigl(1-\cL\bigr) \hat\eta_{21} - 
  \frac{\delta}{2}\,\hat\eta_{20} + \cP_2 \Bigl(\cR_2 - \frac{3\beta_\epsilon}{16\pi}
  \,RZ \eta_0\Bigr)\,=\, 0\,, \\[1mm] 
  &\delta\bigl(1-\cL\bigr) \hat\eta_{22} + \cP_0 \cR_2 \,=\, 0\,,
\end{split}
\end{equation}
where $\cP_n$ denotes the orthogonal projection in $\cY$ onto the subspace
$\cY_n$. The first two equations in \eqref{eta2sys} have a unique solution by
Proposition~\ref{Laminv2}, and as in Section~\ref{ssec33} we are satisfied
with the approximate solution corresponding to \eqref{deltacomp} with $N = 2$. 
Since $\cP_0 \cR_2 = \delta \chi_{21}$ by \eqref{R2exp}, the third equation reduces
to $(1-\cL)\hat\eta_{22} + \chi_{21} = 0$, which also has a unique solution due
to Proposition~\ref{LLamprop}. We conclude that we can choose $\eta_2$ in the form
\begin{equation}\label{eta2exp}
  \eta_2(R,Z,\beta_\epsilon) \,=\, \beta_\epsilon \Bigl((R^2{-}Z^2)\eta_{20} + \delta 
  RZ \eta_{21}\Bigr) + (R^2{-}Z^2)\eta_{22} + \delta 
  RZ \eta_{23} + \eta_{24}\,,
\end{equation}
where all functions $\eta_{2j}$ belong to $\cY_0 \cap \cZ$. Using \eqref{phi2def}
and the calculations at the beginning of Section~\ref{ssecA3}, we obtain 
a similar expression for the corresponding stream function
\begin{equation}\label{phi2exp}
\begin{split}
  \phi_2(R,Z,\beta_\epsilon) \,=\, \beta_\epsilon \Bigl((R^2{-}Z^2)\phi_{20} + \delta 
  RZ \phi_{21}\Bigr) + (R^2{-}Z^2)\phi_{22} + \delta RZ \phi_{23}
  + \beta_\epsilon \phi_{24} + \phi_{25}\,,
\end{split}
\end{equation}
where the functions $\phi_{2j}$ are radially symmetric and belong to $\cS_*(\R^2)$.
With these choices, the left-hand side of \eqref{approx2b} is of size
$\cO(\beta_\epsilon\delta^2 + \epsilon^2\delta)$ in the topology of $\cZ$. 

\subsection{Third order approximation}\label{ssec35}

The third order in the asymptotic expansion \eqref{etaapp} can be computed
in a similar way. According to \eqref{dotr0}, \eqref{dotr1z1} and Remark~\ref{remEuler},
we have $\bar r(t) = r_0\bigl(1-\epsilon^2 + \cO(\epsilon^{4-})\bigr)$ as 
$\epsilon \to 0$, and using \eqref{dereps} we deduce that $t\dot \epsilon 
= \epsilon/2 + \epsilon^3 + \cO(\epsilon^{5-})$. So, if we substitute 
\eqref{etaapp}, \eqref{rzapp} into \eqref{etaeq2} and keep only the 
terms involving $\epsilon^3$ or $\epsilon^3 \beta_\epsilon$, we find
\begin{equation}\label{approx3a}
\begin{split}
  \bigl\{\phi_3\,,\eta_0\bigr\} &+ \bigl\{\phi_2\,,\eta_1 - R\eta_0\bigr\} 
  + \bigl\{\phi_1\,,\eta_2 - R\eta_1 + R^2 \eta_0\bigr\} + \bigl\{\phi_0\,,\eta_3
  - R \eta_2 + R^2 \eta_1 - R^3 \eta_0\bigr\} \\
  &-\frac{r_0}{\Gamma}\Bigl(\dot{\bar r}_0\,\partial_R \eta_2 +
  \bigl(\dot{\bar r}_2{-}\dot{\bar r}_0\bigr)\,\partial_R \eta_0 
  + \dot{\bar z}_0\,\partial_Z \eta_2 + \bigl(\dot{\bar z}_2{-}\dot{\bar z}_0\bigr)
  \,\partial_Z \eta_0\Bigr) \\ 
  \,&=\, \delta \Bigl[\bigl(\cL - {\TS\frac{3}{2}}\bigr)\eta_3 + \partial_R (\eta_2 
  - R\eta_1 + R^2 \eta_0) - t\partial_t \eta_3 - \eta_1\Bigr]\,.
\end{split}
\end{equation}
On the other hand, using \eqref{BScompat} with $m = 3$ and \eqref{BSmdef},
we obtain
\begin{equation}\label{phi3def}
  \phi_3 \,=\, \sum_{m=0}^3 \BS_m[\eta_{3-m}] \,=\, \frac{1}{2\pi}\,L \eta_3
  + \frac{1}{2\pi} \sum_{m=1}^3 \Bigl((\beta_\epsilon + L)P_m + Q_m\Bigr)\eta_{3-m}\,,
\end{equation}
where the polynomials $P_m, Q_m$ are defined in \eqref{PQexp} for $m \le 2$
and in \eqref{P3Q3} for $m = 3$. We can thus write \eqref{approx3a} in the
form
\begin{equation}\label{approx3b}
  \Lambda \eta_3 + \delta \Bigl(t\partial_t \eta_3 + \bigl({\TS\frac{3}{2}}-\cL\bigr) 
  \eta_3 \Bigr) + \cR_3 \,=\, 0\,,
\end{equation}
where
\begin{equation}\label{R3def}
\begin{split}
  \cR_3 \,&=\, \frac{1}{2\pi}\Bigl\{\sum_{m=1}^3 \bigl((\beta_\epsilon + L)P_m
  + Q_m\bigr)\eta_{3-m}\,,\,\eta_0\Bigr\} + \bigl\{\phi_2\,,\eta_1- R\eta_0\bigr\} \\
  &+ \bigl\{\phi_1\,,\eta_2 - R\eta_1 + R^2 \eta_0\bigr\} - \bigl\{\phi_0\,,
  R \eta_2 - R^2 \eta_1 + R^3 \eta_0\bigr\} \\[1mm]
  &- \frac{r_0}{\Gamma}\Bigr(\bigl(\dot{\bar r}_2 - \dot{\bar r}_0\bigr)
  \partial_R \eta_0 + \bigl(\dot{\bar z}_2 - \dot{\bar z}_0\bigr)
  \partial_Z \eta_0 + \dot{\bar z}_0 \partial_Z \eta_2\Bigr) 
  + \delta \partial_R\bigl(R \eta_1 - R^2 \eta_0\bigr) + \delta \eta_1\,.
\end{split}
\end{equation}

\begin{lem}\label{R3lem}
The function $\cR_3$ defined in \eqref{R3def} belongs to $(\cY_1 + \cY_3) \cap 
\cZ$ and satisfies
\begin{equation}\label{R3exp}
  \cR_3 \,=\, \beta_\epsilon\Bigl(R^2Z \chi_{30} + Z \chi_{31}\Bigr)
  + R^2Z \chi_{32} + Z \chi_{33} + \cO(\delta)\,,
\end {equation}
for some (time-independent) radially symmetric functions $\chi_{30}, \chi_{31},
\chi_{32}, \chi_{33} \in \cY_0 \cap \cZ$.  
\end{lem}

The proof of Lemma~\ref{R3lem} is a direct calculation that is briefly outlined
in Section~\ref{ssecA4}. In particular we verify there that the quantity
$\cR_3$ does not contain any factor $\beta_\epsilon^2$, which is perhaps surprising
since $\phi_1$, $\phi_2$, and $\eta_2$ all contain at least one term multiplied 
by $\beta_\epsilon$. We do not need the expression of the $\cO(\delta)$ terms
in \eqref{R3exp}, but they can be computed too and are found to be of the form
$\delta \beta_\epsilon\bigl(R^3 \tilde\chi_{30} + R \tilde\chi_{31}\bigr) +
\delta\bigl(R^3 \tilde\chi_{32} + R\tilde\chi_{33}\bigr)$, where $\tilde\chi_{3j}$
are radially symmetric functions. Finally we mention that $\cR_3$ also
includes terms of the form \eqref{R3exp} that are multiplied by $\delta^2$.

As can be seen from the last line of \eqref{R3def}, there is a unique 
way to choose the quantities $\dot{\bar r}_2$ and $\dot{\bar z}_2$ so 
that $\cR_3 \in \cY_1' + \cY_3$, where $\cY_1'$ is the subspace defined
in \eqref{Y1def}. In view of \eqref{R3exp}, \eqref{dotr0}, \eqref{dotz0}, 
the velocities obtained in this way satisfy 
\begin{equation}\label{dotr2z2}
  \frac{r_0}{\Gamma}\,\dot{\bar r}_2 \,=\, \bigl(c_1 \beta_\epsilon
  + c_2\bigr)\delta\,, \qquad \frac{r_0}{\Gamma}\,\dot{\bar z}_2 \,=\,
  c_3 \beta_\epsilon + c_4 + \cO(\delta^2)\,,
\end{equation}
for some constants $c_1, c_2, c_3, c_4$. Now, decomposing $\cR_3 = \beta_\epsilon \cR_{31} +
\cR_{32}$ where $\cR_{31}, \cR_{32}$ are independent of $\beta_\epsilon$, we look for a
solution of \eqref{approx3b} in the form $\eta_3 = \beta_\epsilon \hat\eta_{31} + \hat\eta_{32}$
where
\begin{equation}\label{eta3aux}
 \Lambda \hat\eta_{31} + \delta\bigl({\TS\frac{3}{2}}-\cL\bigr) \hat\eta_{31}
 + \cR_{31} \,=\, 0\,, \qquad
 \Lambda \hat\eta_{32} + \delta\bigl({\TS\frac{3}{2}}-\cL\bigr) \hat\eta_{32} 
 - \frac{\delta}{2}\,\hat\eta_{31} + \cR_{32} \,=\, 0\,.
\end{equation}
Since $\cR_{31}, \cR_{32} \in \cY_1' + \cY_3$, both equations in \eqref{eta3aux}
can be solved using Proposition~\ref{Laminv2}. However, at this stage, it is
sufficient to use the crude approximation corresponding to $N = 1$ in \eqref{deltacomp}.
This means that we can determine our profiles by solving the equations
$\Lambda \hat\eta_{3j} + \cR_{3j} = 0$ for $j = 1,2$ using Proposition~\ref{Laminv}.
We thus obtain an approximate solution of \eqref{approx3b} of the form
\begin{equation}\label{eta3exp}
  \eta_3(R,Z,\beta_\epsilon) \,=\, \beta_\epsilon \Bigl(R^3\eta_{30} + R \eta_{31}\Bigr)
  + R^3\eta_{32} + R \eta_{33}\,,
\end{equation}
where all functions $\eta_{3j}$ belong to $\cY_0 \cap \cZ$. Using \eqref{phi3def} we
deduce the corresponding expression of the stream function
\begin{equation}\label{phi3exp}
  \phi_3(R,Z,\beta_\epsilon) \,=\, \beta_\epsilon \Bigl(R^3\phi_{30} + R \phi_{31}\Bigr)
  + R^3\phi_{32} + R \phi_{33}\,,
\end{equation}
where the functions $\phi_{3j}$ are radially symmetric and belong to $\cS_*(\R^2)$.
Note that \eqref{phi3exp} does not contain any factor $\beta_\epsilon^2$. With
the choices \eqref{eta3exp}, \eqref{phi3exp}, the left-hand side of \eqref{approx3b} is
of size $\cO(\beta_\epsilon\delta)$ in the topology of $\cZ$. 

\subsection{Fourth order approximation}\label{ssec36}

Finally we compute the fourth order approximation, which is the final step in
our construction. No modification of the vortex speed is needed at this stage, 
so we can take
\begin{equation}\label{dotr3z3}
  \dot{\bar r}_3 \,=\, \dot{\bar z}_3 \,=\, 0\,.
\end{equation}
The full expansion of the vortex speed is therefore
\begin{equation}\label{rzapp2}
  \dot{\bar r}(t) \,=\, \dot{\bar r}_0 + \epsilon^2 \dot{\bar r}_2
  (\beta_\epsilon)\,, \qquad \dot{\bar z}_*(t) \,=\, \dot{\bar z}_0(\beta_\epsilon) +
  \epsilon^2 \dot{\bar z}_2(\beta_\epsilon)\,,
\end{equation}
where $\dot{\bar r}_0, \dot{\bar z}_0$ are defined in \eqref{dotr0}, 
\eqref{dotz0} and $\dot{\bar r}_2, \dot{\bar z}_2$ satisfy \eqref{dotr2z2}.
As is easily verified, this implies that $\bar r(t) = r_0\bigl(1-\epsilon^2 +
\cO(\epsilon^4\beta_\epsilon)\bigr)$ and $t\dot \epsilon = \epsilon/2 + \epsilon^3 
+ \cO(\epsilon^5\beta_\epsilon)$ as $\epsilon \to 0$. 

We look for an approximate solution of \eqref{etaeq2} of the form 
\begin{equation}\label{etaapp2}
  \eta_\app(R,Z,t) \,=\, \sum_{m=0}^4 \epsilon^m \eta_m(R,Z,\beta_\epsilon)\,, \qquad
  \phi_\app(R,Z,t) \,=\, \sum_{m=0}^4 \epsilon^m \phi_m(R,Z,\beta_\epsilon)\,,
\end{equation}
where the profiles $\eta_m, \phi_m$ for $m \le 3$ have been constructed 
in the previous steps, and $\eta_0, \eta_1, \phi_0$ are actually independent
of $\beta_\epsilon$. In analogy with \eqref{phi3def}, we have
\begin{equation}\label{phi4def}
  \phi_4 \,=\, \frac{1}{2\pi}\,L\eta_4 + \frac{1}{2\pi}\sum_{m=1}^4  
  \Bigl(\bigl(\beta_\epsilon + L\bigr)P_m + Q_m\Bigr)\eta_{4-m}\,,
\end{equation}
where the polynomials $P_m$, $Q_m$ are defined in \eqref{PQexp} for 
$m \le 2$, in \eqref{P3Q3} for $m = 3$, and in \eqref{P4Q4} for $m = 4$. 
Replacing \eqref{rzapp2}, \eqref{etaapp2}, \eqref{phi4def} into \eqref{etaeq2}
and proceeding as in the previous sections, we obtain the following
equation for the profile $\eta_4$\:
\begin{equation}\label{approx4b}
  \Lambda \eta_4 + \delta \Bigl(t\partial_t \eta_4 + \bigl(2-\cL\bigr) 
  \eta_4 \Bigr) + \cR_4 \,=\, 0\,,
\end{equation}
where
\begin{align}\nonumber
  \cR_4 \,&=\, \frac{1}{2\pi}\Bigl\{\sum_{m=1}^4 \bigl((\beta_\epsilon + L)P_m
  + Q_m\bigr)\eta_{4-m}\,,\,\eta_0\Bigr\} + \bigl\{\phi_3\,,\eta_1- R\eta_0\bigr\}
  + \bigl\{\phi_2\,,\eta_2 - R\eta_1 + R^2 \eta_0\bigr\}\\ \label{R4def}
  & + \bigl\{\phi_1\,,\eta_3 - R\eta_2 + R^2 \eta_1 - R^3 \eta_0\bigr\}
  - \bigl\{\phi_0\,, R \eta_3 - R^2 \eta_2 + R^3 \eta_1 - R^4\eta_0\bigr\} \\[1mm]
  \nonumber
  &- \frac{r_0}{\Gamma}\Bigr(\bigl(\dot{\bar r}_2 - \dot{\bar r}_0\bigr)
  \partial_R \eta_1 + \bigl(\dot{\bar z}_2 - \dot{\bar z}_0\bigr)
  \partial_Z \eta_1 + \dot{\bar z}_0 \partial_Z \eta_3\Bigr) 
  + \delta \partial_R\bigl(R \eta_2 - R^2 \eta_1 + R^3\eta_0\bigr)
  + 2\delta \eta_2\,.
\end{align}

\begin{lem}\label{R4lem}
The function $\cR_4$ defined in \eqref{R4def} belongs to $(\delta \cY_0 + \cY_2
+ \cY_4) \cap \cZ$ and satisfies
\begin{equation}\label{R4exp}
  \cR_4 \,=\, \sum_{k=0}^2 \beta_\epsilon^k \Bigl(R^3Z \chi_{4k} + RZ \chi_{5k}\Bigr)
  + \cO(\delta)\,,
\end {equation}
for some (time-independent) radially symmetric functions $\chi_{4k}, \chi_{5k}
\in \cY_0 \cap \cZ$.  
\end{lem}

The proof of Lemma~\ref{R4lem} is the same as that of Lemma~\ref{R3lem},
and can therefore be omitted. The only important observation is that
the projection of $\cR_4$ onto the subspace $\cY_0$ is of size
$\cO(\delta)$. This can be seen as a consequence of Remark~\ref{remEuler}\:
when $\delta = \dot{\bar r} = 0$, all profiles $\eta_m$, $\phi_m$ are
even functions of $Z$, so that the quantities $\cR_m$ are odd in $Z$. 

We now project Eq.~\eqref{approx4b} on the subspace $\cY_m$ for $m = 0,2,4$,
and compute an (approximate) solution $\cP_m \eta_4$ invoking either
Proposition~\ref{Laminv} (for $m = 2,4$) or Proposition~\ref{LLamprop}
(for $m = 0$). In the latter case, we use the observation that $\cP_0 \cR_4
= \cO(\delta)$ to show that $\cP_0 \eta_4$ is regular in the limit
$\delta \to 0$. Altogether, we obtain an approximate solution 
of \eqref{approx4b} in the form
\begin{equation}\label{eta4exp}
  \eta_4(R,Z,\beta_\epsilon) \,=\, \sum_{k=0}^2 \beta_\epsilon^k \Bigl(R^2 Z^2\eta_{4k} + 
  \bigl(R^2 - Z^2\bigr) \eta_{5k} + \eta_{6k}\Bigr)\,,
\end{equation}
where the functions $\eta_{jk} \in \cY_0 \cap \cZ$ are radially symmetric and 
time-independent. Using \eqref{phi4def} we deduce a similar expression for the
stream function
\begin{equation}\label{phi4exp}
  \phi_4(R,Z,\beta_\epsilon) \,=\, \sum_{k=0}^2 \beta_\epsilon^k \Bigl(R^2 Z^2\phi_{4k} + 
  \bigl(R^2 - Z^2\bigr)\phi_{5k} + \phi_{6k}\Bigr)\,,
\end{equation}
and with these choices the left-hand side of \eqref{approx4b} is of size 
$\cO(\beta_\epsilon^2\delta)$ in the topology of $\cZ$.

\smallskip
Since we have now completed the construction of our approximate solution, we
explain precisely how to define the vortex radius ${\bar r}(t)$ and the
time-dependent aspect ratio $\epsilon(t) = \sqrt{\nu t}/{\bar r}(t)$.  In view
of \eqref{dotr0}, \eqref{dotr2z2}, and \eqref{rzapp2}, the function $\bar r(t)$
satisfies the differential equation
\begin{equation}\label{rODE}
  \dot{\bar r}(t) \,=\, -\frac{\Gamma\delta}{r_0}\Bigl(1 -\epsilon(t)^2
  \bigl(c_1\beta_{\epsilon(t)} + c_2\bigr)\Bigr) \,=\,
  -\frac{\Gamma\delta}{r_0}\biggl(1 - \frac{\nu t}{{\bar r}(t)^2}
  \Bigl(c_1\log\frac{\bar r(t)}{\sqrt{\nu t}} + c_2\Bigr)\biggr)\,, 
\end{equation}
with initial condition $\bar r(0) = r_0$. The right-hand side of \eqref{rODE} is
a smooth function of $\bar r > 0$, uniformly in $t \in (0,T_\dif)$, and also a
$C^{0,\alpha}$ function of time for any $\alpha < 1$. Applying the
Cauchy-Lipschitz theorem, we obtain a unique local solution of \eqref{rODE},
which can be extended as long as $\bar r(t) > 0$. Now, if we define
$\epsilon(t) = \sqrt{\nu t}/{\bar r}(t)$, it follows that $\bar r(t) = r_0
\bigl(1 - \epsilon(t)^2 + \cO(\epsilon^4\beta_\epsilon)\bigr)$, 
and it is easy to see that the solution of \eqref{rODE} is well-defined 
and has the required properties on the time intervals relevant for our
considerations.

\begin{rem}\label{Remapp}
It is useful to notice that the approximate solution $\eta_\app$ given
by \eqref{etaapp2} satisfies, for all $t > 0$, 
\begin{align}\label{appint}
  &\int_{\R^2} \eta_\app(R,Z,t)\dd R\dd Z \,=\, 1\,, \\ \label{appmom}
  &\int_{\R^2} R\,\eta_\app(R,Z,t)\dd R\dd Z \,=\,   \int_{\R^2}
  Z\,\eta_\app(R,Z,t)\dd R\dd Z  \,=\, 0\,.
\end{align}
Indeed, at each step $m \ge 1$, the vorticity profile $\eta_m$ is constructed by
solving equations of the form $\Lambda \eta_m +  \bigl(\frac{m}2{}-\cL\bigr)
\eta_m + \cR_m = 0$, where the source term $\cR_m$ has vanishing integral
(by definition) and zero first order moments (due to the choice of the
speeds $\dot{\bar r}_{m-1}$, $\dot{\bar z}_{m-1})$. These properties are
inherited by the profile $\eta_m$, due to Proposition~\ref{LLamprop}, 
and in view of \eqref{phi0def} this leads to \eqref{appint}, \eqref{appmom}. 
\end{rem}

\subsection{Estimate of the remainder}\label{ssec37}

This section is devoted to the proof of Proposition~\ref{Remprop}. Our task is
to estimate the remainder \eqref{Remdef}, where $\eta_*$, $\phi_*$ are defined
in \eqref{etaapp1}, and for this we need bounds on the derivatives of the
stream function in terms of the vorticity. If $\phi = \BS^\epsilon[\eta]$, where
the Biot-Savart operator is defined in \eqref{BSeps}, we have the formulas
\begin{align}\nonumber
  \partial_Z\phi(R,Z) \,&=\, -\frac{1}{2\pi}\int_{\Omega_\epsilon}
  \sqrt{(1{+}\epsilon R)(1{+}\epsilon R')}
  ~\tilde F(s)\,\frac{(Z{-}Z')\,\eta(R',Z')}{(R{-}R')^2 + (Z{-}Z')^2}
  \dd R' \dd Z'\,, \\ \label{derphi}
  \partial_R\phi(R,Z) \,&=\, -\frac{1}{2\pi}\int_{\Omega_\epsilon}
  \sqrt{(1{+}\epsilon R)(1{+}\epsilon R')}
  ~\tilde F(s)\,\frac{(R{-}R')\,\eta(R',Z')}{(R{-}R')^2 + (Z{-}Z')^2}
  \dd R' \dd Z' \\ \nonumber
  &\quad~ + \frac{\epsilon}{4\pi}\int_{\Omega_\epsilon}
  \frac{\sqrt{1{+}\epsilon R'}}{\sqrt{1+\epsilon R}}
  \,\bigl(F(s) + \tilde F(s)\bigr)\,\eta(R',Z')\dd R' \dd Z'\,,
\end{align}
where $\tilde F(s) = -2 s F'(s)$, see \cite[Section~4.2]{GS2}. Here, as in
\eqref{sdef1}, we use the shorthand notation
\begin{equation}\label{sdef2}
  s \,=\, \frac{\epsilon^2 D^2}{(1{+}\epsilon R)(1{+}\epsilon R')} \,\equiv\, 
  \epsilon^2\,\frac{(R{-}R')^2 + (Z{-}Z')^2}{(1{+}\epsilon R)(1{+}\epsilon R')}\,.
\end{equation}
In view of \eqref{Fexpand}, we have $\tilde F(s) \to 1$ as $s \to 0$ and $\tilde 
F(s) = \cO(s^{-3/2})$ as $s \to +\infty$. 

Throughout the proof, we fix $t > 0$ and we assume that the parameters
$\epsilon = \sqrt{\nu t}/{\bar r}(t)$ and $\delta = \nu/\Gamma$ are small
enough. By construction the vorticity $\eta_*(R,Z,t)$ defined by
\eqref{etaapp1} vanishes identically when $\rho := (R^2{+}Z^2)^{1/2} \ge
2 \epsilon^{-\sigma_0}$, so we can assume henceforth that $\rho \le 2
\epsilon^{-\sigma_0}$. In that region, we have for any $\gamma \in (0,1)$
the a priori bounds
\begin{equation}\label{etaphibd}
  \sum_{|\alpha| \le 2} |\partial^\alpha \eta_*(R,Z,t)| \,\le\,
  C\,e^{-\gamma \rho^2/4}\,, \qquad 
  \sum_{|\alpha| = 1} |\partial^\alpha \phi_*(R,Z,t)| \,\le\,
  C\,, \qquad 
\end{equation}
where $\alpha = (\alpha_1,\alpha_2) \in \N^2$ and
$\partial^\alpha = \partial_R^{\alpha_1} \partial_Z^{\alpha_2}$. Indeed, the
first estimate in \eqref{etaphibd} holds because $\eta_*$ is obtained by
truncating the asymptotic approximation $\eta_\app(R,Z,t)$ which belongs to the
space $\cZ$ defined in \eqref{Zdef}. The second estimate can then be obtained
using the expressions \eqref{derphi} with $\phi = \phi_*$ and $\eta = \eta_*$.
To see this, we first observe that $1 + \epsilon R \approx 1$ and
$1 + \epsilon R' \approx 1$ in \eqref{derphi}, because both quantities $\rho$
and $\rho' := ({R'}^2{+}{Z'}^2)^{1/2}$ are smaller than
$2 \epsilon^{-\sigma_0} \ll \epsilon^{-1}$. If we use the estimates
$|\tilde F(s)| \le C$ in the first two lines of \eqref{derphi} and
$|F(s) + \tilde F(s)| \le Cs^{-1/2}$ in the third line, we thus obtain
\[
  |\partial_R\phi_*(R,Z,t)| + |\partial_Z\phi_*(R,Z,t)| \,\le\,
  C \int_{\R^2} \frac{|\eta_*(R',Z',t)|}{\sqrt{(R{-}R')^2 + (Z{-}Z')^2}}
  \dd R'\dd Z' \,\le\, C\,,
\]
which concludes the proof of \eqref{etaphibd}. Finally, since
\[
  t\partial_t\eta_*(R,Z,t) \,=\, \chi_0\bigl(\epsilon^{\sigma_0}\rho\bigr)\,
  t\partial_t \eta_\app(R,Z,t) + \sigma_0\,\epsilon^{\sigma_0}\rho \chi_0'\bigl(
  \epsilon^{\sigma_0}\rho\bigr)\,\eta_\app(R,Z,t) \,t\partial_t\log(\epsilon)\,,
\]
it follows from the expressions given in Sections~\ref{ssec33}--\ref{ssec36}
that $t\partial_t \eta_*$ satisfies the same bound as $\eta_*$ in \eqref{etaphibd}. 
Summarizing, in view of \eqref{etaphibd}, the remainder $\Rem(R,Z,t)$ satisfies
\begin{equation}\label{Remfirst}
  e^{\gamma_0 \rho^2/4}\,|\Rem(R,Z,t)| \,\le\, C\,\delta^{-1}(1+\rho)
  \,e^{-(\gamma-\gamma_0)\rho^2/4}\,, \qquad \hbox{when }\,\rho \le 2
  \epsilon^{-\sigma_0}\,,
\end{equation}
for any $\gamma_0 \in (0,1)$. If we assume that $\gamma \in (\gamma_0,1)$, we conclude
that the right-hand side of \eqref{Remfirst} is $\cO(\delta^{-1}\epsilon^\infty)$ if
$\rho \ge \epsilon^{-\sigma_0}$.  So from now on we may concentrate on the inner region
$\rho \le \epsilon^{-\sigma_0}$, where $\eta_* = \eta_\app$ is given by \eqref{etaapp2}. 

In that region we decompose the stream function as $\phi_* = \BS^\epsilon[\chi_0\,\eta_\app] =
\phi_*^0 - \phi_*^1 + \phi_*^2$, where
\[
  \phi_*^0 \,=\, \sum_{m=0}^4 \epsilon^m\,\BS_m[\eta_\app]\,, \quad
  \phi_*^1 \,=\, \sum_{m=0}^4 \epsilon^m\,\BS_m[(1{-}\chi_0)\,\eta_\app]\,, \quad
  \phi_*^2 \,=\, \sum_{m=5}^\infty \epsilon^m \,\BS_m[\chi_0\,\eta_\app]\,.
\]
Here $\chi_0$ is a shorthand notation for $\chi_0(\epsilon^{\sigma_0}\rho)$. 
The convergence of the series defining $\phi_*^2$ is easily justified
using Lemmas~\ref{Flem} and \ref{Kexpansion}, if we observe that both
inequalities in \eqref{sdef1} are satisfied since $\rho,\rho' \ll \epsilon^{-1}$. 
The principal term $\BS_5[\chi_0\,\eta_\app]$ can be estimated using the
explicit representation \eqref{BSmdef}, where $P_5, Q_5$ are homogeneous
polynomials of degree $5$, and this leads to a bound of the form 
\[
  |\partial_R  \phi_*^2(R,Z,t)| + |\partial_Z  \phi_*^2(R,Z,t)|  \,\le\,
  C \epsilon^5\beta_\epsilon \,(1+\rho)^5\,, \qquad
  \rho \le  \epsilon^{-\sigma_0}\,,
\]
where $\beta_\epsilon = \log(1/\epsilon)$. Moreover we have $ |\partial_R \phi_*^1|
+ |\partial_Z \phi_*^1| = \cO(\epsilon^\infty)$ because $(1-\chi_0)\eta_\app =
\cO(\epsilon^\infty)$. Finally, in view of \eqref{BScompat} and \eqref{etaapp2}, 
we have the identity
\[
  \phi_*^0 \,=\, \phi_\app \,+\, \sum_{m=5}^8 \,\epsilon^m \!\!\sum_{k=m-4}^4 \BS_k[\eta_{m-k}]\,.
\]
from which we easily deduce
\[
  |\partial_R\bigl(\phi_*^0-\phi_\app\bigr)| + |\partial_Z\bigl(\phi_*^0-\phi_\app
  \bigr)| \,\le\, C \epsilon^5\beta_\epsilon^3 \,(1+\rho)^5\,. 
\]
Collecting the estimates above, it is straightforward to verify that the remainder
\eqref{Remdef} satisfies
\begin{equation}\label{Remdiff}
  \big|\Rem(R,Z,t) - \widehat\Rem(R,Z,t)\big| \,\le\,  C \delta^{-1} \epsilon^5
  \beta_\epsilon^3 \,(1+\rho)^5\,e^{-\gamma \rho^2/4}\,, \qquad
  \rho \le  \epsilon^{-\sigma_0}\,, 
\end{equation}
where $\widehat\Rem(R,Z,t)$ is the quantity defined for all $(R,Z) \in \R^2$
by the formula
\begin{equation}\label{hatRemdef}
  \cL \eta_\app + \epsilon \partial_R\bigl(S_4\eta_\app\bigr) - t\partial_t
  \eta_\app - \frac{1}{\delta}\,\bigl\{\phi_\app\,,S_4\eta_\app\bigr\} +
  \frac{\epsilon \bar r}{\delta \Gamma}\Bigl(\dot{\bar r}\,\partial_R\eta_\app
  + \dot{\bar z}_* \,\partial_Z\eta_\app\Bigr)\,,
\end{equation}
with $S_4 = 1 - \epsilon R + (\epsilon R)^2 - (\epsilon R)^3 + (\epsilon R)^4$. 

Now the crucial observation is that the asymptotic approximation \eqref{etaapp},
\eqref{rzapp} was constructed precisely so as to make the quantity \eqref{hatRemdef}
small in the topology of $\cZ$. More precisely, the results of
Sections~\ref{ssec33}--\ref{ssec36} can be rephrased as follows:
\begin{equation}\label{hatRemest}
  \delta\,\widehat\Rem(R,Z,t) \,=\, \cO_\cZ\Bigl(\epsilon \delta^2
  + \epsilon^2\beta_\epsilon \delta^2 + \epsilon^3 \beta_\epsilon \delta +
  \epsilon^4\beta_\epsilon^2 \delta + \epsilon^5\beta_\epsilon^3\Bigr)\,.
\end{equation}
Inside the parenthesis in the right-hand side, the first four terms
represent what remains from the quantities $\epsilon^m \bigl(\Lambda
\eta_m + \delta\bigl[t\partial_t + \frac{m}{2}-\cL\bigr]\eta_m +
\cR_m)$ for $m = 1,2,3,4$ after the profiles $\eta_m$ have been
determined, and the last one corresponds to those terms in
\eqref{hatRemdef} which are of order $\cO(\epsilon^5)$ or higher and
were therefore not considered in the construction of $\eta_\app$. Combining
\eqref{Remdiff}, \eqref{hatRemest} and using Young's inequality, we obtain
\[
  \sup_{\rho \le \epsilon^{-\sigma_0}} e^{\gamma_0\rho^2/4}\,|\Rem(R,Z,t)|
  \,\le\, \frac{C}{\delta}\Bigl(\epsilon \delta^2 + \epsilon^3 \beta_\epsilon
  \delta + \epsilon^5 \beta_\epsilon^3\Bigr) \,\le\, C\bigl(\epsilon\delta
  + \epsilon^{\gamma_5}\delta^{-1}\bigr)\,,
\]
for any $\gamma_5 < 5$. This concludes the proof of \eqref{Remest}. \QED
  
\subsection{The Eulerian approximation}\label{ssec38}

As was already observed in Remark~\ref{remEuler}, if we set
$\delta = \dot{\bar r} = 0$ in the expansion \eqref{etaapp}, we obtain an
approximate solution $\eta_\app^E, \phi_\app^E, \dot{\bar z}_E$ of equation
\eqref{etaeqEuler}, which is nothing but the stationary Euler equation in a
frame moving with (constant) velocity $\dot{\bar z}_E\,e_z$.  As is well known
\cite{Arn}, steady states of the Euler system are often characterized by a
global functional relation between the vorticity and the stream function. In our
case, in view of \eqref{etaeqEuler}, we expect finding a function
$\Phi_\epsilon : \R_+ \to \R$ such that
\begin{equation}\label{Phirel}
  \phi_\app^E(R,Z) - \frac{r_0 \dot{\bar z}_E}{2\Gamma}\,(1+\epsilon R)^2
  \,=\, \Phi_\epsilon\biggl(\frac{\eta_\app^E(R,Z)}{1+\epsilon R}\biggr)
  + \cO\bigl(\epsilon^{M+1-}\bigr)\,,
\end{equation}
for all $(R,Z) \in \R^2$ such that $\rho := \sqrt{R^2+Z^2} \ll \epsilon^{-1}$. 

In this section, we first verify that a relation of the form \eqref{Phirel}
holds to second order, namely with $M = 2$. Using the expressions
\eqref{eta1exp}, \eqref{phi1exp}, \eqref{eta2exp}, \eqref{phi2exp} with
$\delta = 0$ and simplifying somehow the notation, we can write our
approximate solution in the form
\begin{equation}\label{etaEulerapp}
\begin{split}
  \eta_\app^E(R,Z) \,&=\, \eta_0 + \epsilon R \eta_1 \,+ \epsilon^2 (R^2{-}Z^2)
   \eta_2 \,+ \epsilon^2\eta_3\,, \\[1mm]
  \phi_\app^E(R,Z) \,&=\, \phi_0 + \epsilon R \phi_1 + \epsilon^2 (R^2{-}Z^2)
  \phi_2 + \epsilon^2 \phi_3\,, 
\end{split}
\end{equation}
where $\eta_0,\phi_0$ are given by \eqref{phi0def}, and the profiles
$\eta_1,\eta_2,\eta_3 \in \cZ$ and $\phi_1,\phi_2,\phi_3 \in \cS_*(\R^2)$ are
all radially symmetric. Note that $\eta_m, \phi_m$ may include factors of
$\beta_\epsilon = \log(1/\epsilon)$ when $m \ge 1$, but this dependence is not explicitly
indicated. We also expand the unknown function $\Phi_\epsilon$ in \eqref{Phirel}
in powers of $\epsilon$\:
\begin{equation}\label{Phiexpand}
  \Phi_\epsilon(s) \,=\, \Phi_0(s) + \epsilon \Phi_1(s) + \epsilon^2 \Phi_2(s)\,.
\end{equation}
Finally, to simplify the writing, we denote
\begin{equation}\label{omegadef}
  \omega \,=\, \frac{1}{4\pi}\bigl(\beta_\epsilon - 1 + 2v\bigr) \,=\,
  \frac{r_0 \dot{\bar z}_E}{\Gamma}  + \cO(\epsilon^2\beta_\epsilon)\,,
\end{equation}
where the last equality follows from \eqref{dotz0}, \eqref{dotr1z1},
\eqref{dotr2z2}.

If we consider equality \eqref{Phirel} to leading order in $\epsilon$, thus
neglecting all terms that are $\cO(\epsilon)$ or $\cO(\epsilon\beta_\epsilon)$,  
we obtain the relation $\phi_0 - \omega/2 = \Phi_0(\eta_0)$, which determines the 
principal term $\Phi_0$ in the expansion \eqref{Phiexpand}. In view of 
\eqref{phi0def}, \eqref{phi0exp} we thus have
\begin{equation}\label{Phi0def}
  \Phi_0(s) \,=\, \phi_0(0) - \frac{\omega}{2} - \frac{1}{4\pi}\,\Ein
  \Bigl(\log\frac{1}{4\pi s}\Bigr)\,, \qquad s > 0\,.
\end{equation}
The constant in \eqref{Phi0def} has no relevance, but it is important to
note that $\Phi_0(s) \sim -\frac{1}{4\pi}\log\log \frac{1}{s}$ as $s \to 0$.
For later use we define
\begin{equation}\label{Adef}
  A(\rho) \,=\, \Phi_0'\bigl(\eta_0(\rho)\bigr) \,=\, \frac{\partial_R \phi_0}{\partial_R \eta_0}
  \,=\, \frac{\partial_Z \phi_0}{\partial_Z \eta_0} \,=\, \frac{4}{\rho^2}\Bigl(
  e^{\rho^2/4} - 1\Bigr)\,, \qquad \rho > 0\,.
\end{equation}
Incidentally we observe that $A(\rho) = 1/h(\rho)$ where $h$ is defined in 
\eqref{phihdef}. 

To the next order in $\epsilon$, we deduce from \eqref{Phirel} the relation
\begin{equation}\label{Phi1def}
  (\phi_1 - \omega)R \,=\, \Phi_0'(\eta_0)(\eta_1-\eta_0)R + \Phi_1(\eta_0)\,,
\end{equation}
which can be satisfied only if $\Phi_1= 0$, because $\Phi_1(\eta_0)$ is
the only radially symmetric term in \eqref{Phi1def}. Dividing by $R$, 
we obtain the equality $\phi_1 - \omega = A(\eta_1 - \eta_0)$, which happens to 
be satisfied in view of our definitions of the profiles $\eta_1, \phi_1$. This 
fact can be verified by following carefully the calculations in Section~\ref{ssec33}. 

Finally we exploit \eqref{Phirel} to order $\epsilon^2$, keeping in mind that
$\Phi_1 = 0$. In this calculation, we neglect the $\cO(\epsilon^2\beta_\epsilon)$
correction in \eqref{omegadef}, because this term would only add an irrelevant
constant to the function $\Phi_2$. We thus obtain the relation
\begin{align*}
  (R^2{-}Z^2)\phi_2 + \phi_3 - \frac{\omega}{2}\,R^2 \,=\, &\,\Phi_0'(\eta_0)\Bigl(
  (R^2{-}Z^2)\eta_2 + \eta_3 + (\eta_0-\eta_1)R^2\Bigr) \\ 
  &+ \frac12 \Phi_0''(\eta_0)(\eta_0 - \eta_1)^2 R^2 + \Phi_2(\eta_0)\,,
\end{align*}
where it is useful to substitute $R^2 = \frac12(R^2{+}Z^2) + \frac12(R^2{-}Z^2)$.
The terms containing $R^2{-}Z^2$ cancel exactly due to the identity
\[
  \phi_2 - \frac12 \Psi - A \eta_2 \,=\, 0\,, \qquad \hbox{where}\quad
  \Psi \,=\, \frac{\omega}{2} + \Phi_0'(\eta_0)(\eta_0-\eta_1) + \frac12
  \Phi_0''(\eta_0)(\eta_0-\eta_1)^2\,,
\]
which is satisfied by definition of the profiles $\phi_2, \eta_2$, as can
be verified by following the calculations in Section~\ref{ssec34}. We
are thus left with a relation involving only radially symmetric terms
\begin{equation}\label{Phi2def}
  \phi_3 - \frac12 (R^2{+}Z^2)\Psi - A \eta_3 \,=\, \Phi_2(\eta_0)\,,
\end{equation}
which provides the definition of the second order correction $\Phi_2$
in \eqref{Phiexpand}. Summarizing, if $\Phi_\epsilon$ is defined by 
\eqref{Phiexpand} with $\Phi_1 = 0$, $\Phi_0$ given by \eqref{Phi0def}
and $\Phi_2$ by \eqref{Phi2def}, we have shown that \eqref{Phirel} holds 
with $M = 2$. 

We now come back to the approximate solution $\eta_*, \phi_*$ of \eqref{etaeq2}
constructed in Sections~\ref{ssec33}--\ref{ssec36}, and we show that it
also satisfies a relation of the form \eqref{Phirel}, in a sufficiently 
small region near the origin. To formulate that result, we denote
\begin{equation}\label{Thetadef}
  \Theta(R,Z,t) \,=\, \phi_*(R,Z,t) - \frac{{\bar r}
  \dot{\bar z}_*}{2\Gamma}\,(1+\epsilon R)^2 - \Phi_\epsilon
  \biggl(\frac{\eta_*(R,Z,t)}{1+\epsilon R}\biggr)\,, \qquad
  (R,Z) \in \Omega_\epsilon\,.
\end{equation}
  
\begin{prop}\label{Phiprop}
There exist $\sigma_1 \in (0,\sigma_0)$ and $N \in \N$ such that, 
for any $\gamma_3 < 3$, the quantity $\Theta$ defined by \eqref{Thetadef} 
satisfies, for some $C > 0$,
\begin{equation}\label{Thetaest}
  |\partial_R \Theta(R,Z,t)| +  |\partial_Z \Theta(R,Z,t)| \,\le\,
  C(\epsilon\delta + \epsilon^{\gamma_3})(1+\rho)^N\,, \qquad \rho \le
  \epsilon^{-\sigma_1}\,,
\end{equation}
whenever $\epsilon$ and $\delta$ are small enough. 
\end{prop}

\begin{proof}
The idea is to compare $\Theta$ with the second order Eulerian approximation
\begin{equation}\label{ThetaE}
  \Theta_\app^E(R,Z,t) \,=\, \phi_\app^E(R,Z,t) - \frac{r_0\dot{\bar z}_E}{2\Gamma}\,
  (1+\epsilon R)^2 - \Phi_\epsilon \biggl(\frac{\eta_\app^E(R,Z,t)}{1+\epsilon R}
  \biggr)\,,
\end{equation}
which is of size $\cO(\epsilon^{3-})$ in view of \eqref{Phirel}. Here we consider
both quantities $\eta_\app^E,\phi_\app^E$ as time-dependent, because we deal with
the viscous case where $\epsilon = \sqrt{\nu t}/{\bar r}(t)$. We already
estimated the difference $\phi_* - \phi_\app$ in the proof of
Proposition~\ref{Remprop}, and by construction we know that $\phi_\app = \phi_\app^E
+ \cO(\epsilon\delta + \epsilon^3\beta_\epsilon)$. These arguments lead to
the bound
\begin{equation}\label{phicomp}
  |\partial_R\bigl(\phi_*-\phi_\app^E\bigr)| + |\partial_Z\bigl(\phi_*-\phi_\app^E
  \bigr)| \,\le\, C \bigl(\epsilon\delta + \epsilon^3\beta_\epsilon\bigr)\,(1+\rho)^5\,, 
  \qquad \rho \le \epsilon^{-\sigma_0}\,.
\end{equation}
On the other hand, we know that that ${\bar r}(t) = r_0(1+\cO(\epsilon^2))$,
and in view of \eqref{dotr1z1}, \eqref{dotr2z2} the difference between the vertical
speed $\dot{\bar z}_*$ and its second order Eulerian approximation $\dot{\bar z}_E$
is of size $(\Gamma/r_0)\cdot\cO(\epsilon^2\beta_\epsilon)$. We thus find
\begin{equation}\label{speedcomp}
  \Bigl|\frac{{\bar r}\dot{\bar z}_*}{2\Gamma} - \frac{r_0\dot{\bar z}_E}{2\Gamma}\Bigr|
  \,\bigl|\partial_R (1+\epsilon R)^2\bigr| \,\le\, C\epsilon^3\beta_\epsilon\,,
  \qquad \rho \le \epsilon^{-\sigma_0}\,.
\end{equation}
Finally $\eta_*$ is just a truncation of $\eta_\app$ and by definition $\eta_\app -
\eta_\app^E = \cO(\epsilon\delta + \epsilon^3\beta_\epsilon)$ in the topology of $\cZ$.
This gives the following bound
\begin{equation}\label{etacomp}
  \sum_{|\alpha| \le 1}\bigl|\partial^\alpha\bigl(\eta_* - \eta_\app^E\bigr)(R,Z,t)\bigr|
  \,\le\, C \bigl(\epsilon\delta + \epsilon^3\beta_\epsilon\bigr)\,(1+\rho)^N
  e^{-\rho^2/4}\,, \qquad \rho \le \epsilon^{-\sigma_0}\,,
\end{equation}
for some $N \in \N$. 

At this point we observe that $\eta_* - \eta_0 = \cO(\epsilon)$ in the topology of
$\cZ$ when $\rho \le \epsilon^{-\sigma_0}$. In particular, there exists $N \in \N$
such that $|\eta_* - \eta_0| \le C\epsilon (1+\rho)^N \eta_0$ in that
region, and one can verify that $N = 3$ is in fact sufficient. If we choose
$\sigma_1 > 0$ small enough so that $N \sigma_1 < 1$, it follows that 
\begin{equation}\label{etastarest}
  \frac{1}{2}\,\eta_0(\rho) \,\le\, \frac{\eta_*(R,Z,t)}{1+\epsilon R}
  \,\le\, 2\,\eta_0(\rho)\,, \qquad \rho \le \epsilon^{-\sigma_1}\,,
\end{equation}
whenever $\epsilon > 0$ is small enough. The same estimate holds for the
Eulerian approximation $\eta_\app^E$.

To conclude the proof of Proposition~\ref{Phiprop}, we need bounds on the
derivatives of the function $\Phi_\epsilon$ defined in \eqref{Phiexpand}. We begin
with the leading order term $\Phi_0$ which is given by the explicit formula
\eqref{Phi0def}. We have
\[
  \Phi_0'\Bigl(\frac{s}{4\pi}\Bigr) \,=\, \frac{1-s}{s\log(1/s)}\,, \qquad  
  \frac{1}{4\pi}\,\Phi_0''\Bigl(\frac{s}{4\pi}\Bigr) \,=\, -\frac{s-1+\log(1/s)}{
    s^2\bigl(\log(1/s)\bigr)^2}\,, \qquad s > 0\,. 
\]
Thanks to \eqref{etastarest} we only need to evaluate these expressions 
when the argument $s/(4\pi)$ takes its values in the interval  
$\bigl[\frac12 \eta_0(\rho),2 \eta_0(\rho)\bigr]$. In view of 
Lemma~\ref{fglem} below, there exists $C > 1$ such that, for all 
$\lambda \in [1/2,2]$ and all $\rho > 0$,
\begin{equation}\label{Phi0derest}
  \frac{A(\rho)}{C} \,\le\, \Phi_0'\bigl(\lambda \eta_0(\rho)\bigr)
  \,\le\, C A(\rho)\,, \qquad 
  \bigl|\Phi_0''\bigl(\lambda \eta_0(\rho)\bigr)\bigr| \,\le\,
  C B(\rho)\,, \qquad 
\end{equation}
where $A(\rho)$ is defined in \eqref{Adef} and
\begin{equation}\label{Bdef}
  B(\rho) \,=\, -\Phi_0''(\eta_0(\rho)) \,=\, \frac{16\pi}{\rho^4}
  \,\Bigl((\rho^2-4)e^{\rho^2/2} + 4 e^{\rho^2/4}\Bigr)\,, \qquad \rho > 0\,.
\end{equation}
The second order contribution $\Phi_2$ is not known explicitly, but
from the definition \eqref{Phi2def}, where the left-hand side belongs
to $\cS_*(\R^2)$, we deduce that there exist $C > 0$ and $N \in \N$
such that
\begin{equation}\label{Phi2derest}
  \bigl|\Phi_2'\bigl(\lambda\eta_0(\rho)\bigr)\bigr| \,\le\, C A(\rho)(1+\rho)^N\,, \qquad
  \bigl|\Phi_2''\bigl(\lambda\eta_0(\rho)\bigr)\bigr| \,\le\, C B(\rho)(1+\rho)^N\,, \qquad
\end{equation}
for all $\rho > 0$ and all $\lambda \in [1/2,2]$. 

Now, if $\partial^\alpha = \partial_R$ or $\partial_Z$, we decompose
\begin{align*}
  \partial^\alpha \Phi_\epsilon\Bigl(\frac{\eta_*}{1{+}\epsilon R}\Bigr) -
  \partial^\alpha \Phi_\epsilon\Bigl(\frac{\eta_\app^E}{1{+}\epsilon R}\Bigr)
  \,=\, &\,\Phi_\epsilon'\Bigl(\frac{\eta_*}{1{+}\epsilon R}\Bigr)
  \Bigl(\partial^\alpha \Bigl(\frac{\eta_*}{1{+}\epsilon R}\Bigr) - \partial^\alpha
   \Bigl(\frac{\eta_\app^E}{1{+}\epsilon R}\Bigr)\Bigr) \\
   &+ \Bigl(\Phi_\epsilon'\Bigl(\frac{\eta_*}{1{+}\epsilon R}\Bigr) -
   \Phi_\epsilon'\Bigl(\frac{\eta_\app^E}{1{+}\epsilon R}\Bigr)\Bigr)
   \,\partial^\alpha \Bigl(\frac{\eta_\app^E}{1{+}\epsilon R}\Bigr)\,,
\end{align*}
and we estimate the right-hand side using \eqref{etacomp}, \eqref{Phi0derest}, 
and \eqref{Phi2derest}. Taking into account the preliminary bounds
\eqref{phicomp}, \eqref{speedcomp}, we arrive at an estimate of the
form
\[
  \sum_{|\alpha| = 1} \bigl|\partial^\alpha \bigl(\Theta(R,Z,t) -
  \Theta_\app^E(R,Z,t)\bigr)\bigr| \,\le\, C(\epsilon\delta +
  \epsilon^3\beta_\epsilon)(1+\rho)^N\,, \qquad \rho \le \epsilon^{-\sigma_1}\,.
\]
As was already mentioned, the approximation $\Theta_\app^E(R,Z,t)$ is 
$\cO(\epsilon^{3-})$ in the topology of $\cS_*(\R^2)$, so altogether we arrive 
at \eqref{Thetaest}. 
\end{proof}

In the argument above we used the following elementary result, whose 
proof can be omitted. 

\begin{lem}\label{fglem}
Let $f,g : (0,+\infty) \to (0,+\infty)$ be defined by
\[
   f(s) \,=\, \frac{1-s}{s\log(1/s)}\,, \qquad g(s) \,=\, 
   \frac{s-1+\log(1/s)}{s^2\bigl(\log(1/s)\bigr)^2} \,=\, -f'(s)\,, \qquad
   s > 0\,.
\]
Then given any $\Lambda > 1$ there exists $C > 1$ such that, for any
$\lambda \in [\Lambda^{-1},\Lambda]$ and any $s > 0$,
\[
  \frac{1}{C} \,\le\, \frac{f(\lambda s)}{f(s)} \,\le\, C\,, \qquad
  \frac{1}{C} \,\le\, \frac{g(\lambda s)}{g(s)} \,\le\, C\,.
\] 
\end{lem}

\section{Energy estimates and stability proof}\label{sec4}

In the previous section we constructed an approximate solution $\eta_*$ of the
rescaled vorticity equation \eqref{etaeq} which corresponds, in the original
variables, to a sharply concentrated vortex ring of radius $\bar r(t)$ located
at the approximate vertical position $\bar z_*(t)$.  Our goal is now to control
the difference between this approximation and the actual solution $\eta$ of
\eqref{etaeq} with initial data $\eta_0$, which is located at the
modified vertical position $\bar z(t) = {\bar z}_*(t) + \delta\tilde z(t)$
given by \eqref{barzpert}. This will conclude the proof of our main results,
Theorems~\ref{main1} and \ref{main2}.

Our starting point is the evolution equation \eqref{tildeq} for the perturbation
$\tilde \eta$ defined in \eqref{etapert}, which can be written in the form
\begin{equation}\label{tildeq2}
\begin{split}
  t\partial_t \tilde\eta &+ \frac{1}{\delta}\bigl\{\phi_*\,,\tilde\zeta\bigr\} +
  \frac{1}{\delta}\bigl\{\tilde\phi\,,\zeta_*\bigr\} + \bigl\{\tilde\phi\,,
  \tilde\zeta\bigr\} -\frac{\epsilon \bar r}{\delta\Gamma}\Bigl(\dot{\bar r}\,
  \partial_R \tilde\eta + \dot{\bar z}_* \,\partial_Z \tilde\eta\Bigr)\\ \,&=\,
  \cL \tilde\eta + \epsilon\partial_R \tilde\zeta + \frac{1}{\delta}
  \,\Rem(R,Z,t) + \frac{\epsilon \bar r \dot{\tilde z}}{\delta\Gamma}
  \,\Bigl(\partial_Z\eta_* + \delta\partial_Z\tilde\eta\Bigr)\,,
\end{split}
\end{equation}
where to simplify the notation we use the letter $\zeta$ to denote the
potential vorticity: 
\begin{equation}\label{zetadef}
  \tilde \zeta(R,Z,t) \,=\, \frac{\tilde\eta(R,Z,t)}{1+\epsilon R}\,, \qquad
  \zeta_*(R,Z,t) \,=\, \frac{\eta_*(R,Z,t)}{1+\epsilon R}\,.
\end{equation}
From our previous work \cite{GS2} we know that Eq.~\eqref{tildeq2} has a unique
solution $\tilde\eta$, in an appropriate weighted $L^2$ space, with zero initial
data. Our goal is to control the evolution of that solution on a large time
interval, uniformly with respect to the viscosity in the limit $\nu \to 0$. This
is not an easy task, because several terms in \eqref{tildeq2} are multiplied by
the Reynolds number $\delta^{-1} = \Gamma/\nu$, which becomes arbitrarily large
in the regime we consider. As was explained in the introduction, we shall use
energy estimates to control the solution of \eqref{tildeq2}, but a few
preliminary steps are necessary before starting the actual calculations.

\subsection{Control of the lowest order moments}\label{ssec41}

To implement our strategy based on energy estimates, we need a precise information
on the lowest order moments of the solution of \eqref{tildeq2}. We first
define, for all $t > 0$, 
\begin{equation}\label{mudef}
  \mu_0(t) \,=\, \int_{\Omega_\epsilon} \tilde \eta(R,Z,t)\dd X\,, \qquad
  \mu_1(t) \,=\, \int_{\Omega_\epsilon} \bigl(R + \epsilon R^2/2\bigr)
  \,\tilde \eta(R,Z,t)\dd X\,,
\end{equation}
where $\D X = \D R\dd Z$ denotes the Lebesgue measure in $\R^2$. 

\begin{lem}\label{lem:mu0mu1}
The moments defined in \eqref{mudef} satisfy $\mu_0(t) = \cO(\epsilon^\infty\delta^{-1})$ 
and $\mu_1(t) = \cO(\epsilon + \epsilon^{\gamma_5}\delta^{-2})$ for any $\gamma_5 < 1$,
whenever $\epsilon$ and $\delta$ are small enough.
\end{lem}

\begin{proof}
The conclusion can be obtained by direct calculations, but we find it more
illuminating to use the conserved quantities of the original equation \eqref{omeq}.
The first one is the total circulation
\begin{equation}\label{Mdef}
  M(t) \,=\, \int_\Omega \omega_\theta(r,z,t)\dd r\dd z \,=\, \Gamma
  \int_{\Omega_\epsilon} \bigl(\eta_* + \delta \tilde \eta\bigr)(R,Z,t)\dd X
  \,=\, \Gamma \int_{\Omega_\epsilon} \eta_* \dd X + \Gamma \delta \mu_0(t)\,,
\end{equation}
which satisfies $M(0) = \Gamma$ and is almost constant in time. In fact it is
proved in \cite[Section~4.4]{GS2} that $0 \le 1 - M(t)/\Gamma \le C\exp(-c/\epsilon^2)$
for some positive constants $C$ and $c$. Moreover, since the approximate solution
$\eta_\app$ lies in the space $\cZ$ defined by \eqref{Zdef}, it follows from
\eqref{etaapp1} and \eqref{appint} that $\int_{\Omega_\epsilon}\eta_* \dd X = 1 + 
\cO(\exp(-c/\epsilon^{2\sigma_0}))$. Therefore $\mu_0(t) = \cO(\exp(-c/\epsilon^{2\sigma_0})
\,\delta^{-1})$ by \eqref{Mdef}.

We next consider the total impulse in the vertical direction
\begin{equation}\label{Idef}
  I \,=\, \int_\Omega r^2 \omega_\theta(r,z,t)\dd r\dd z \,=\, \Gamma \bar r(t)^2
  \int_{\Omega_\epsilon} (1+\epsilon R)^2 \bigl(\eta_* + \delta \tilde \eta\bigr)(R,Z,t)
  \dd X\,,
\end{equation}
which is known to be exactly conserved \cite{MB,GS1}, so that $I = \Gamma r_0^2$ for 
all times. Equality \eqref{Idef} can be rephrased as $I/\Gamma = I_*(t) + \delta\bar r(t)^2
\mu(t)$, where 
\begin{equation}\label{Istardef}
  I_*(t) \,=\, \bar r(t)^2 \int_{\Omega_\epsilon} (1+\epsilon R)^2 \eta_*(R,Z,t)\dd X\,,
  \qquad \mu(t) \,=\, \mu_0(t) + 2\epsilon \mu_1(t)\,.
\end{equation}
It is not difficult to show that
\begin{equation}\label{Istarder}
  t I_*'(t) \,=\, -\bar r(t)^2 \int_{\Omega_\epsilon} (1+\epsilon R)^2\,\Rem(R,Z,t)
  \dd X\,.
\end{equation}
The easiest way to establish \eqref{Istarder} is to observe that the impulse
$I_*(t)$ would be conserved if $\eta_*$ was an exact solution of \eqref{etaeq},
so that the remainder $\Rem(R,Z,t)$ defined in \eqref{Remdef} is the only term
that contributes to the evolution of $I_*(t)$. However equality \eqref{Istarder}
can also be verified by a direct calculation. In any case, since $\Rem(R,Z,t)$
satisfies estimate \eqref{Remest} and $\int_{\Omega_\epsilon} \Rem(R,Z,t)\dd x =
\cO(\epsilon^\infty)$, we deduce from \eqref{Istarder} that $|tI_*'(t)| \le C
r_0^2\bigl(\epsilon^2\delta + \epsilon^{\gamma_5+1}\delta^{-1}\bigr)$, hence
\[
  |I_*(t)-r_0^2| \,\le\, \int_0^t |I_*'(s)|\dd s \,\le\, C r_0^2 \int_0^t
  \frac{\epsilon(s)^2\delta + \epsilon(s)^{\gamma_5+1}\delta^{-1}}{s}\dd s
  \,\le\, C r_0^2\bigl(\epsilon^2\delta + \epsilon^{\gamma_5+1}\delta^{-1}\bigr)\,.
\]
As $r_0^2 - I_*(t) = \delta\bar r(t)^2\mu(t)$, we conclude that $\mu(t) =
\cO\bigl(\epsilon^2 + \epsilon^{\gamma_5+1}\delta^{-2}\bigr)$, which gives
the desired estimate for $\mu_1(t)$. 
\end{proof}

It is not clear if the strategy above can be applied to control the first order
moment of the perturbation $\tilde\eta$ with respect to the vertical variable
$Z$. In particular, we are not aware of any (approximately) conserved quantity
that we could use for this purpose. Instead we choose the modulation parameter
$\tilde z(t)$ in \eqref{barzpert} so that the vertical moment vanishes identically\: 
\begin{equation}\label{mu2def}
  \mu_2(t) \,:=\, \int_{\Omega_\epsilon} Z \,\tilde\eta(R,Z,t)\dd X \,=\, 0\,.
\end{equation}
Differentiating \eqref{mu2def} with respect to time and using \eqref{tildeq2}, we
obtain the relation
\begin{equation}\label{tildezeq}
  \dot{\tilde z}(t) \int_{\Omega_\epsilon}Z \bigl(\partial_Z\eta_*
  + \delta \partial_Z \tilde\eta\bigr)\dd X \,=\, \frac{\delta \Gamma}{\epsilon \bar r}
  \int_{\Omega_\epsilon} Z\,\cR(R,Z,t) \dd X\,,
\end{equation}
where
\begin{equation}\label{cRdef}
\begin{split}
  \cR \,=\, &\frac{1}{\delta}\bigl\{\phi_*\,,\tilde\zeta\bigr\} + \frac{1}{\delta}
  \bigl\{\tilde\phi\,,\zeta_*\bigr\} + \bigl\{\tilde\phi\,,\tilde\zeta\bigr\}
  -\frac{\epsilon \bar r}{\delta\Gamma}\Bigl(\dot{\bar r}\,\partial_R \tilde\eta + 
  \dot{\bar z}_* \,\partial_Z \tilde\eta\Bigr)\\  &- \cL \tilde\eta - \epsilon\partial_R
  \tilde\zeta - \frac{1}{\delta}\,\Rem(R,Z,t)\,.
\end{split}
\end{equation}
In view of Lemma~\ref{lem:mu0mu1} the integral in the left-hand side of \eqref{tildezeq}
is equal to $-1 + \cO(\epsilon^\infty)$, and is therefore bounded away from zero
if $\epsilon$ is small enough. The integral in the right-hand side is a priori
of size $\cO(\delta^{-1})$, but we observe that $\cR = \delta^{-1}\Lambda\tilde\eta
+ \cO(\epsilon \delta^{-1})$, where $\Lambda$ is the linear operator defined in
\eqref{Lamdef}. Using the properties established in Proposition~\ref{LLamprop},
we see that the leading term gives no contribution\:
\[
  \frac{1}{4\pi}\int_{\R^2} Z\,\Lambda \tilde \eta \dd X \,=\, \bigl(Z\eta_0\,,
  \Lambda\tilde\eta\bigr)_{\cY} \,=\, -\bigl(\Lambda(Z\eta_0)\,,
  \tilde\eta\bigr)_{\cY} \,=\, 0\,,
\]
since $Z\eta_0 = -2\partial_Z\eta_0$ is in the kernel of $\Lambda$. These
considerations, which will be made rigorous in Section~\ref{ssec48} below,
show that the modulation speed $\dot{\tilde z}$ is uniquely determined by
\eqref{tildezeq}, and suggest that $\dot{\tilde z}(t) =
\cO(\|\tilde \eta\|_{\cX_\epsilon})$ as long as $\|\tilde \eta\|_{\cX_\epsilon}$
remains of size $\cO(1)$. In particular $\delta\tilde z(t)$ is indeed
a small correction to the vertical position of the vortex ring.

\subsection{Definition and properties of the weight function}\label{ssec42}

We now provide the precise definition of the weight function $W_\epsilon : \Omega_\epsilon
\to (0,+\infty)$ which appears in the energy functional \eqref{Edef}. We give
ourselves three positive numbers $\sigma_1, \sigma_2, \gamma$ such that
\begin{equation}\label{sigmarange}
  0 \,<\, \sigma_1 \,<\, \sigma_0 \,<\, 1 \,<\, \sigma_2\,, \qquad
  \gamma \,=\, \sigma_1/\sigma_2\,,
\end{equation}
where $\sigma_0 \in (0,1)$ is the cut-off exponent already introduced in \eqref{etaapp1}.
As we shall see $\sigma_2 > 1$ can be chosen arbitrarily, but $\sigma_1 > 0$ has
to be taken sufficiently small. In particular $\sigma_1$ should be small enough
so that Proposition~\ref{Phiprop} holds. 

As in \eqref{zetadef}, if $\epsilon > 0$ and $\delta > 0$ are sufficiently
small, we denote $\zeta_* = \eta_*/(1+\epsilon R)$, where $\eta_*$ is the
approximate solution of \eqref{etaeq} given by \eqref{etaapp1}. We recall that
$\zeta_*$ and $\phi_* := \BS^\epsilon[\eta_*]$ satisfy the relation
\eqref{Phieps}, where $\Phi_\epsilon : \R_+ \to \R$ is the function constructed
in Section~\ref{ssec38}.  We decompose the domain
$\Omega_\epsilon = \bigl\{(R,Z) \,;\, 1 + \epsilon R > 0 \bigr\}$ into a
disjoint union
$\Omega_\epsilon' \cup \Omega_\epsilon'' \cup \Omega_\epsilon'''$, where
\begin{equation}
\begin{split}
  \Omega_\epsilon' \,&=\, \Bigl\{(R,Z) \in \Omega_\epsilon\,;\, \Phi_\epsilon'(\zeta_*(R,Z))
  < \exp\bigl(\epsilon^{-2\sigma_1}/4\bigr)\Bigr\}\,, \\ \label{Ompdef}
  \Omega_\epsilon'' \,&=\, \Bigl\{(R,Z) \in \Omega_\epsilon\setminus \Omega_\epsilon'\,;\,
  \rho \le \epsilon^{-\sigma_2}\Bigr\}\,, \\
  \Omega_\epsilon''' \,&=\, \Bigl\{(R,Z) \in \Omega_\epsilon\,;\, \rho >
  \epsilon^{-\sigma_2}\Bigr\}\,.
\end{split}
\end{equation}
Here and in what follows, if $(R,Z) \in \R^2$, we denote $\rho = (R^2{+}Z^2)^{1/2}$. 
The domains $\Omega_\epsilon', \Omega_\epsilon''$ also depend (mildly) on $\delta$, 
but for simplicity this dependence is not indicated explicitly. 

\begin{lem}\label{innerlem}
If $\epsilon > 0$ is small enough, the inner region $\Omega_\epsilon'$ defined in
\eqref{Ompdef} is diffeomorphic to a open disk, and there exists $\kappa > 0$
such that 
\begin{equation}\label{Ompinc}
  \bigl\{(R,Z)\,;\, \rho \le \epsilon^{-\sigma_1}\bigl\} \,\subset\, \Omega_\epsilon'
  \,\subset\, \Bigl\{(R,Z)\,;\, \rho^2 \le \epsilon^{-2\sigma_1} + \kappa\log
  \frac{1}{\epsilon}\Bigl\}\,.
\end{equation}
\end{lem}

\begin{proof}
The main properties of the function $\Phi_\epsilon$ are established in the proof of
Proposition~\ref{Phiprop}. In particular, using estimates \eqref{etastarest},
\eqref{Phi0derest}, \eqref{Phi2derest}, it is easy to verify that
\begin{equation}\label{Phipest}
  \frac12\,A(\rho) \,\le\, \Phi_\epsilon'\bigl(\zeta_*(R,Z)\bigr) \,\le\,
  2 A(\rho)\,, \qquad \hbox{when}~\,\rho \le 2\epsilon^{-\sigma_1}\,.
\end{equation}
Here $A(\rho) = (4/\rho^2)\bigl(e^{\rho^2/4}-1\bigr)$, see \eqref{Adef}. Since
$2A(\epsilon^{-\sigma_1}) < \exp(\epsilon^{-2\sigma_1}/4)$ as soon as 
$\epsilon^{-\sigma_1} \ge 3$, we deduce that $(R,Z) \in \Omega_\epsilon'$ if $\rho
\le \epsilon^{-\sigma_1}$. Similarly, using the lower bound in \eqref{Phipest}, it is
easy to verify that the inner region $\Omega_\epsilon'$ is contained in the disk
$\rho^2 \le \epsilon^{-2\sigma_1} + \kappa\log\frac{1}{\epsilon}$ if $\kappa > 4\sigma_1$
and $\epsilon > 0$ is small enough. Finally $\Omega_\epsilon'$ is diffeomorphic
to a disk because $\Phi_\epsilon'(\zeta_*)$ is $C^2$-close to a strictly increasing
radially symmetric function when $\epsilon > 0$ is small, see \eqref{Phiexpand}. 
\end{proof}

We next choose a smooth cut-off function $\chi_1 : \R \to [\frac12,3]$ such that
\begin{equation}\label{chi1def}
  \chi_1(x) \,=\, \frac{1}{1+x} \quad \hbox{for }\, |x| \le \frac12\,, \qquad
  \chi_1'(x) \,=\, 0 \quad \hbox{for }\, |x| \ge \frac34\,.
\end{equation}
The weight $W_\epsilon :  \Omega_\epsilon \to (0,+\infty)$ is defined by
\begin{equation}\label{Wdef}
  W_\epsilon(R,Z) \,=\, \chi_1(\epsilon R) \times \begin{cases}
  \Phi_\epsilon'\bigl(\zeta_*(R,Z)\bigr) & \hbox{in}~\,\Omega_\epsilon'\,, \\[1mm]
  \exp\bigl(\epsilon^{-2\sigma_1}/4\bigr) & \hbox{in}~\,\Omega_\epsilon''\,, \\[1mm]
  \exp\bigl(\rho^{2\gamma}/4\bigr) & \hbox{in}~\,\Omega_\epsilon'''\,,
  \end{cases}
\end{equation}
where $\gamma = \sigma_1/\sigma_2 < 1$ and $\Omega_\epsilon', \Omega_\epsilon'',
\Omega_\epsilon'''$ are the regions defined in \eqref{Ompdef}. In other words, 
we assume that $W_\epsilon = \Phi_\epsilon'(\zeta_*)/(1{+}\epsilon R)$ as long as
the numerator remains smaller than the threshold value $\exp(\epsilon^{-2\sigma_1}/4)$. 
Outside this inner region, the weight is radially symmetric except for the
geometric factor $\chi_1(\epsilon R)$, and the radial profile remains constant
as long as $\rho \le \epsilon^{-\sigma_2}$ before increasing again like
$\exp(\rho^{2\gamma}/4)$ when $\rho > \epsilon^{-\sigma_2}$. By construction the function
$W_\epsilon$ is locally Lipschitz continuous in $\Omega_\epsilon$, and smooth in the
interior of all three regions \eqref{Ompdef}. The (mild) dependence of $W_\epsilon$
upon the parameter $\delta > 0$ is not indicated explicitly. A schematic representation 
of the graph of $W_\epsilon$ is given in Figure~\ref{fig2}.

Further properties of the weight $W_\epsilon$ are collected in the following lemma.

\begin{lem}\label{Wlem}
There exist positive constants $C_1, C_2$ such that, if $\epsilon$, $\delta$, and $\sigma_1$
are small enough, the weight $W_\epsilon$ satisfies the uniform bounds 
\begin{equation}\label{Wunifbd}
  C_1\exp\bigl(\rho^{2\gamma}/4\bigr) \,\le\, W_\epsilon(R,Z) \,\le\, C_2 A(\rho)\,,
  \qquad (R,Z) \in \Omega_\epsilon\,,
\end{equation}
where $\rho = (R^2{+}Z^2)^{1/2}$ and $A(\rho)$ is defined in \eqref{Adef}.
Moreover, given any $\gamma_1 < 1$ there exists $C_3 > 0$ such that the
following estimates hold in the inner region
\begin{equation}\label{Wapprox}
  \bigl|W_\epsilon(R,Z) - A(\rho)\bigr| \,+\, \bigl|\nabla W_\epsilon(R,Z) -
  \nabla A(\rho)\bigr| \,\le\, C_3\,\epsilon^{\gamma_1} A(\rho)\,,
  \qquad (R,Z) \in \Omega_\epsilon'\,.
\end{equation}
\end{lem}

\begin{proof}
Since $\frac12 \le \chi_1(\epsilon R) \le 3$ and $\exp(\rho^{2\gamma}/4) \le C
A(\rho)$, we deduce from \eqref{Phipest} that the bounds \eqref{Wunifbd} hold 
in the inner region $\Omega_\epsilon'$, as well as in the far field region
$\Omega_\epsilon'''$. In the intermediate region $\Omega_\epsilon''$
we know that $\rho \le \epsilon^{-\sigma_2}$, so that $\exp(\rho^{2\gamma}/4) \le
\exp(\epsilon^{-2\sigma_1}/4)$ since $\gamma = \sigma_1/\sigma_2$, and this 
gives the lower bound in \eqref{Wunifbd}. If $\rho \ge 2\epsilon^{-\sigma_1}$, it
is clear that $\exp(\epsilon^{-2\sigma_1}/4) \le A(\rho)$, which is the desired
upper bound. Finally if $(R,Z) \in \Omega_\epsilon''$ and $\rho \le 2\epsilon^{-\sigma_1}$,
we deduce from \eqref{Phipest} that $\exp(\epsilon^{-2\sigma_1}/4) \le
\Phi_\epsilon'\bigl(\zeta_*(R,Z)\bigr) \le 2 A(\rho)$, which concludes the proof
of the upper bound in \eqref{Wunifbd}. 

To prove \eqref{Wapprox}, we start from the expression \eqref{Wdef} of the
weight $W_\epsilon$ in the inner region $\Omega_\epsilon'$. We know from
\eqref{Adef} that $A(\rho) = \Phi_0'(\eta_0)$, where $\eta_0$ is defined in
\eqref{phi0def}. We thus find
\begin{equation}\label{Wapp2}
  |W_\epsilon(R,Z) - A(\rho)\bigr| \,\le\, \bigl|\chi_1(\epsilon R) - 1\bigr|
  \Phi_\epsilon'(\zeta_*) + \bigl|\Phi_\epsilon'(\zeta_*) - \Phi_\epsilon'(\eta_0)\bigr|
  + \bigl|\Phi_\epsilon'(\eta_0) - \Phi_0'(\eta_0)\bigr|\,.
\end{equation}
Since $\chi_1(\epsilon R) = (1+\epsilon R)^{-1}$ when $(R,Z) \in \Omega_\epsilon'$, 
the first term in the right-hand of \eqref{Wapp2} is smaller than $C \epsilon |R|
\,\Phi_\epsilon'(\zeta_*) \le C \epsilon^{1-\sigma_1}A(\rho)$. For the second term, we use
the bounds \eqref{etastarest}, \eqref{Phi0derest}, and \eqref{Phi2derest} to obtain
\[
  \bigl|\Phi_\epsilon'(\zeta_*) - \Phi_\epsilon'(\eta_0)\bigr| \,\le\,
  \sup_{\frac12\le\lambda\le2}\bigl|\Phi_\epsilon''(\lambda\eta_0)\bigr|\,|\zeta_* -
  \eta_0| \,\le\, CB(\rho)(1+\rho)^N \epsilon \eta_0 \,\le\, C \epsilon^{\gamma_1}
  A(\rho)\,,
\]
where in the last inequality we assumed that $\sigma_1 > 0$ is small enough
so that $N\sigma_1 \le 1 - \gamma_1$. The last term in \eqref{Wapp2} is bounded
by $\epsilon^2 |\Phi_2'(\eta_0)| \le C \epsilon^{\gamma_1} A(\rho)$ in view of
\eqref{Phi2derest}. Altogether we arrive at the estimate $|W_\epsilon(R,Z) - A(\rho)\bigr|
\le C \epsilon^{\gamma_1} A(\rho)$. The corresponding inequality for the first order
derivatives can be obtained in a similar way, and we omit the details
\end{proof}

\subsection{Coercivity of the energy functional}\label{ssec43}

For $\epsilon \ge 0$ small enough, we introduce the weighted $L^2$ space
$\cX_\epsilon = \bigl\{\eta \in L^2(\Omega_\epsilon)\,;\,\|\eta\|_{\cX_\epsilon}
< \infty\bigr\}$ defined by the norm \eqref{Xepsdef}, namely 
\begin{equation}\label{cXdef}
  \|\eta\|_{\cX_\epsilon}^2 \,=\, \int_{\Omega_\epsilon} W_\epsilon(R,Z)
  \,|\eta(R,Z)|^2 \dd R\dd Z\,. 
\end{equation}
In the limiting case $\epsilon = 0$, it is understood that $\Omega_0 =
\R^2$ and $W_0(R,Z) = A(\rho)$, in agreement with \eqref{Wapprox}.
Assuming that $\epsilon > 0$, we consider the energy functional
\eqref{Edef}, namely
\begin{equation}\label{Edef2}
  E_\epsilon[\eta] \,=\, \frac12\,\|\eta\|_{\cX_\epsilon}^2 - E_\epsilon^\kin[\eta]\,,
  \qquad \eta \in \cX_\epsilon\,,
\end{equation}
where $E_\epsilon^\kin$ is the kinetic energy defined by
\begin{equation}\label{Ekindef}
  E_\epsilon^\kin[\eta] \,=\, \frac12 \int_{\Omega_\epsilon} \phi\,
  \eta\dd R\dd Z \,=\, \frac12 \int_{\Omega_\epsilon} \frac{|\nabla
  \phi|^2}{1+\epsilon R}\dd R\dd Z\,,\qquad \phi \,=\, \BS^\epsilon[\eta]\,.
\end{equation}

Since we are interested in the regime where $\epsilon$ is small, it is
important to observe that $E_\epsilon^\kin[\eta]$ becomes {\em singular} in the limit
$\epsilon \to 0$, if the vorticity $\eta$ has nonzero mean. This divergence is
related to the well-known fact that any (nontrivial) nonnegative vorticity distribution
in $\R^2$ has infinite kinetic energy. The regular part of $E_\epsilon^\kin[\eta]$ is
given, to leading order, by the two-dimensional energy
\begin{equation}\label{Ekin0def}
  E_0^\kin[\eta] \,=\, \frac{1}{4\pi} \int_{\R^2} \bigl(L\eta)\eta \dd X \,=\,
  \frac{1}{4\pi} \int_{\R^2}\int_{\R^2} \log\Bigl(\frac{8}{D}\Bigr)\,\eta(R,Z)
  \eta(R',Z')\dd X \dd X'\,,
\end{equation}
where $L$ is the integral operator \eqref{Lopdef} and $D^2 = (R{-}R')^2 + (Z{-}Z')^2$.
More precisely, we have the following statement, whose proof is postponed to
Section~\ref{ssecB1}. 

\begin{lem}\label{Ekinlem}
If $\epsilon > 0$ is small and $\eta \in \cX_\epsilon$ satisfies $\supp(\eta) \subset
B_\epsilon := \{(R,Z) \in \Omega_\epsilon\,;\, \rho \le \epsilon^{-\sigma_1}\}$, we 
have the expansion 
\begin{equation}\label{Ekinexp}
  E_\epsilon^\kin[\eta]  \,=\, \frac{\beta_\epsilon-2}{4\pi}\,\mu_0^2 +
  E_0^\kin[\eta] + \cO\bigl(\epsilon\beta_\epsilon\|\eta\|_{\cX_\epsilon}^2\bigr)\,,
 \qquad \hbox{as}~\epsilon \to 0\,,
\end{equation}
where $\beta_\epsilon = \log(1/\epsilon)$ and $\mu_0 = \int_{\Omega_\epsilon} \eta\dd R\dd Z$. 
\end{lem}

We now consider the (formal) limit of the functional $E_\epsilon[\eta]$ as
$\epsilon \to 0$, assuming that $\eta$ has zero mean to avoid the logarithmic
divergence in the right-hand side of \eqref{Ekinexp}. In view of \eqref{Wapprox}
and Lemma~\ref{Ekinlem}, we obtain the limiting functional
\begin{equation}\label{E0def}
  E_0[\eta] \,=\, \frac12 \int_{\R^2} A(\rho)\,\eta(R,Z)^2\dd R\dd Z -
  E^\kin_0[\eta] \,=\, \frac12\,\|\eta\|_{\cX_0}^2 - E^\kin_0[\eta]\,,
\end{equation}
which is studied in detail in our previous work \cite{GS3}. In particular,
we have the following property\:

\begin{prop}\label{E0prop}
There exists constants $C_4 > 2$ and $C_5 > 0$ such that, for all $\eta\in \cX_0$,
\begin{equation}\label{E0ineq}
  \|\eta\|_{\cX_0}^2 \,\le\, C_4 E_0[\eta] + C_5\bigl(\mu_0^2 + \mu_1^2 + \mu_2^2\bigr)\,,
\end{equation}
where $\mu_0 = \int_{\R^2} \eta\dd X$, $\mu_1 = \int_{\R^2} R\eta\dd X$, $\mu_2 =
\int_{\R^2} Z\eta\dd X$.  
\end{prop}

\begin{proof}
The results of \cite[Section~2]{GS3} show that \eqref{E0ineq} holds when
$\mu_0 = \mu_1 = \mu_2 = 0$, and the general case is easily deduced by the
following argument. Given $\eta \in \cX_0$ we define
\[
  \hat \eta \,=\, \eta - \mu_0\eta_0 + \mu_1 \partial_R\eta_0 + \mu_2
  \partial_Z\eta_0\,, \qquad
  \hat \phi \,=\, \phi - \mu_0\phi_0 + \mu_1 \partial_R\phi_0 + \mu_2
  \partial_Z\phi_0\,, 
\]
where $\phi = (2\pi)^{-1}L\eta$ and $\eta_0,\phi_0$ are as in \eqref{phi0def}.
By construction the integral and the first order moments of the new function
$\hat\eta \in \cX_0$ vanish, so that we can apply the results of \cite{GS3}
which give the bound $\|\hat\eta\|_{\cX_0}^2 \,\le\, C_4 E_0[\hat\eta]$. On
the other hand, expanding the quadratic expressions $\|\hat\eta\|_{\cX_0}^2$
and $E_0[\hat\eta]$ and using H\"older's inequality, it is straightforward
to verify that
\[
  \|\hat\eta\|_{\cX_0}^2 \,\ge\, \frac12\,\|\eta\|_{\cX_0}^2 - C\bigl(\mu_0^2
  + \mu_1^2 + \mu_2^2)\,, \quad  E_0[\hat\eta] \,\le\, E_0[\eta] +
  \frac{1}{4C_4}\,\|\eta\|_{\cX_0}^2 + C\bigl(\mu_0^2+ \mu_1^2 + \mu_2^2)\,,
\]
for some $C > 0$. If we combine these estimates, we arrive at the
bound \eqref{E0ineq} with a deteriorated constant $C_4$.   
\end{proof}

Using Proposition~\ref{E0prop}, we now establish a similar coercivity property for
the functional $E_\epsilon$ when $\epsilon > 0$ is small. The proof of the
following proposition is again postponed to Section~\ref{ssecB1}. 

\begin{prop}\label{Eepsprop}
If the weight $W_\epsilon$ satisfies \eqref{Wunifbd} and \eqref{Wapprox}, there
exist constants $C_6 > 0$ and $C_7 > 0$ such that, for all sufficiently small
$\epsilon > 0$ and all $\eta\in \cX_\epsilon$, we have the estimate
\begin{equation}\label{Eepsineq}
  \|\eta\|_{\cX_\epsilon}^2 \,\le\, C_6 E_\epsilon[\eta] + C_7\bigl(\beta_\epsilon
  \mu_0^2 + \mu_1^2 + \mu_2^2\bigr)\,,
\end{equation}
where $\beta_\epsilon = \log(1/\epsilon)$ and $\mu_0 = \int_{\Omega_\epsilon}\eta\dd X$,
$\mu_1 = \int_{\Omega_\epsilon} R\eta\dd X$, $\mu_2 = \int_{\Omega_\epsilon}Z\eta\dd X$.
\end{prop}

In what follows we use the bound \eqref{Eepsineq} to estimate the vorticity
perturbation $\tilde\eta$ introduced in \eqref{etapert}. The corresponding
moments $\mu_0, \mu_1$ are under control thanks to Lemma~\ref{lem:mu0mu1}, and
$\mu_2 = 0$ according to \eqref{mu2def}. So it remains to bound the energy
functional $E_\epsilon[\tilde\eta]$, which is the purpose of the remaining
sections.

\subsection{Time evolution of the energy}\label{ssec44}

Let $\tilde\eta$ be the solution of \eqref{tildeq2} with zero initial data.
Assuming that $\delta > 0$ and $\sigma > 0$ are sufficiently small, we consider
for $t \in (0,T_\adv \delta^{-\sigma})$ the energy function
\begin{equation}\label{Edef3}
  E_\epsilon(t) \,=\, \frac12 \int_{\Omega_\epsilon} W_\epsilon(R,Z)\,\tilde
  \eta(R,Z,t)^2 \dd X \,-\, \frac12 \int_{\Omega_\epsilon} \tilde \phi(R,Z,t)
  \,\tilde \eta(R,Z,t)\dd X\,,
\end{equation}
where $\epsilon = \sqrt{\nu t}/{\bar r}(t)$ and $W_\epsilon$ is the weight
function defined by \eqref{Wdef}. The first term in the right-hand side
of \eqref{Edef3} is equal to $\frac12 \|\tilde\eta\|_{\cX_\epsilon}^2$, and
the second one is the kinetic energy $E^\kin_\epsilon[\tilde\eta]$, which satisfies
\eqref{Ekindef} and involves the stream function $\tilde\phi = \BS^\epsilon[\tilde\eta]$
defined by the Biot-Savart formula \eqref{BSeps}. Differentiating \eqref{Edef3} with
respect to time and using the relations \eqref{dereps}, \eqref{Ekindef} together
with the evolution equation \eqref{tildeq2}, we obtain by a direct calculation
\begin{align*}
  t\partial_t E_\epsilon \,&=\,  \int_{\Omega_\epsilon} \Bigl(W_\epsilon \tilde\eta
  \,t\partial_t\tilde\eta +  \frac12 t(\partial_t W_\epsilon) \tilde\eta^2\Bigr)\dd X
  \,-\, \int_{\Omega_\epsilon} \Bigl(\tilde\phi\,t\partial_t \tilde\eta
  + \frac{t\dot \epsilon}{2}\frac{R|\nabla \tilde \phi|^2}{(
  1+\epsilon R)^2}\Bigr)\dd X \\ \,&=\, I_1 + I_2 + I_3 + I_4 + I_5 + I_6\,,
\end{align*}
where the quantities $I_1,\dots,I_6$ collect the following terms. 

\noindent {\bf 1.} {\em  Local advection terms\:}
\begin{equation}\label{I1def}
\begin{split}
  I_1 \,&=\, -\frac{1}{\delta}\int_{\Omega_\epsilon} W_\epsilon \tilde\eta \bigl\{\phi_*\,,
  \,\tilde \zeta\bigr\}\dd X + \frac{\epsilon \bar r \dot{\bar z}_*}{\delta\Gamma}
  \int_{\Omega_\epsilon} W_\epsilon \tilde\eta\,\partial_Z \tilde\eta \dd X \\ \,&=\,
  -\frac{1}{\delta}\int_{\Omega_\epsilon} W_\epsilon\tilde\eta \Bigl\{\phi_* - 
  \frac{\bar r \dot{\bar z}_*}{2\Gamma}\,(1+\epsilon R)^2\,,\,\tilde \zeta\Bigr\}
  \dd X\\ \,&=\, -\frac{1}{2\delta}\int_{\Omega_\epsilon} \Bigl\{W_\epsilon(1+\epsilon R)\,,\,
  \phi_* - \frac{\bar r \dot{\bar z}_*}{2\Gamma}\,(1+\epsilon R)^2\Bigr\}
  \tilde \zeta^2\dd X\,.
\end{split}
\end{equation}

\noindent {\bf 2.} {\em  Nonlocal advection terms\:}
\begin{equation}\label{I2def}
\begin{split}
  I_2 \,&=\, \frac{1}{\delta}\int_{\Omega_\epsilon}\tilde\phi\bigl\{\phi_*\,,
  \,\tilde \zeta\bigr\}\dd X - \frac{\epsilon \bar r \dot{\bar z}_*}{\delta\Gamma}
  \int_{\Omega_\epsilon} \tilde\phi\,\partial_Z \tilde \eta \dd X
  - \frac{1}{\delta}\int_{\Omega_\epsilon} \bigl(W_\epsilon \tilde\eta -\tilde\phi\bigr)
  \bigl\{\tilde\phi\,,\,\zeta_*\bigr\}\dd X \\ \,&=\,
  \frac{1}{\delta}\int_{\Omega_\epsilon}\tilde\phi\Bigl\{\phi_* - 
  \frac{\bar r \dot{\bar z}_*}{2\Gamma}\,(1+\epsilon R)^2\,,\,\tilde \zeta\Bigr\}
  \dd X -\frac{1}{\delta}\int_{\Omega_\epsilon} W_\epsilon \tilde\eta \bigl\{\tilde\phi\,,
  \,\zeta_*\bigr\}\dd X \\ \,&=\, 
  \frac{1}{\delta}\int_{\Omega_\epsilon} \Bigl\{\tilde\phi\,,\,
  \phi_* - \frac{\bar r \dot{\bar z}_*}{2\Gamma}\,(1+\epsilon R)^2\Bigr\}
  \tilde \zeta\dd X -\frac{1}{\delta}\int_{\Omega_\epsilon} W_\epsilon(1+\epsilon R)
  \bigl\{\tilde\phi\,,\,\zeta_*\bigr\}\tilde\zeta\dd X\,.
\end{split}
\end{equation}

\noindent {\bf 3.} {\em  Nonlinear terms\:}
\begin{equation}\label{I3def}
  I_3 \,=\, -\int_{\Omega_\epsilon} \bigl(W_\epsilon \tilde \eta -\tilde\phi\bigr)
  \bigl\{\tilde\phi\,,\tilde\zeta\bigr\}\dd X \,=\, -\int_{\Omega_\epsilon}
  \bigl\{W_\epsilon\tilde\eta\,,\tilde\phi\bigr\}\tilde\zeta \dd X\,.
\end{equation}

\noindent {\bf 4.} {\em  Diffusive terms\:}
\[
  I_4 \,=\, \int_{\Omega_\epsilon} \bigl(W_\epsilon \tilde \eta -\tilde\phi\bigr)
  \Bigl(\cL \tilde\eta + \epsilon \partial_R \tilde\zeta\Bigr)\dd X\,.
\]
Integrating by parts as explained in Section~\ref{ssecB2}, we obtain the
equivalent expression
\begin{equation}\label{I4def}
\begin{split}
  I_4 \,&=\, -\int_{\Omega_\epsilon} W_\epsilon |\nabla\tilde\eta|^2\dd X -
  \int_{\Omega_\epsilon} (\nabla W_\epsilon\cdot\nabla\tilde\eta)\tilde\eta \dd X
   - \int_{\Omega_\epsilon} V_\epsilon \tilde\eta^2\dd X\\
  &\quad\, -\frac{\epsilon}{2}\int_{\Omega_\epsilon} \partial_R\bigl(W_\epsilon(1+\epsilon R)
  \bigr) \tilde\zeta^2\dd X  + \frac{\epsilon}{4}
  \int_{\Omega_\epsilon}\frac{R|\nabla\tilde\phi|^2}{(1+\epsilon R)^2}\dd X\,,   
\end{split}
\end{equation}
\quad where 
\begin{equation}\label{Vdef}
  V_\epsilon \,=\, \frac14(R\partial_R + Z\partial_Z)W_\epsilon - \frac12 W_\epsilon
  -(1+\epsilon R)\,.
\end{equation}

\smallskip\noindent {\bf 5.} {\em  Remainder term\:}
\begin{equation}\label{I5def}
  I_5 \,=\, \frac{1}{\delta}\int_{\Omega_\epsilon} \bigl(W_\epsilon \tilde \eta - 
  \tilde\phi\bigr)\,\Rem(R,Z,t)\dd X\,.
\end{equation}

\noindent {\bf 6.} {\em  Additional terms\:}
\begin{equation}\label{I6def}
\begin{split}
  I_6 \,&=\, \frac12 \int_{\Omega_\epsilon} t(\partial_t W_\epsilon)\tilde \eta^2
  \dd X + \frac{\epsilon \bar r \dot{\bar r}}{\delta\Gamma}
  \int_{\Omega_\epsilon} \bigl(W_\epsilon \tilde \eta - \tilde\phi\bigr)\,\partial_R
  \tilde \eta\dd X \\
  &\quad\, - \frac{t\dot \epsilon}{2}\int_{\Omega_\epsilon}  \frac{R|\nabla
  \tilde \phi|^2}{(1+\epsilon R)^2}\dd X + \frac{\epsilon \bar r \dot{\tilde z}}
  {\delta\Gamma}\int_{\Omega_\epsilon}\bigl(W_\epsilon \tilde \eta -\tilde\phi\bigr)
  \,\bigl(\partial_Z\eta_* + \delta\partial_Z\tilde\eta\bigr)\dd X\,.
\end{split}
\end{equation}

For the purposes of our analysis, it is useful to reorganize some terms appearing
in the quantities $I_4$ and $I_6$. First, using \eqref{phidef} and integrating by parts,
it is easy to verify that
\begin{equation}\label{Iaux}
  -\int_{\Omega_\epsilon}\tilde\phi\,\partial_R\tilde\eta \dd X \,=\,
  \int_{\Omega_\epsilon}\tilde\eta \,\partial_R \tilde\phi\dd X \,=\,
  \frac{\epsilon}{2}\int_{\Omega_\epsilon}  \frac{|\nabla \tilde \phi|^2}{
  (1+\epsilon R)^2}\dd X\,.
\end{equation}
So, if we collect all terms involving $|\nabla\tilde\phi|^2$ in \eqref{I4def},
\eqref{I6def}, and \eqref{Iaux}, we obtain the quantity
\[
  \Bigl(\frac{\epsilon}{4} - \frac{t\dot \epsilon}{2}\Bigr) \int_{\Omega_\epsilon}
  \frac{R|\nabla \tilde \phi|^2}{(1+\epsilon R)^2}\dd X + \frac{\epsilon^2 \bar r
  \dot{\bar r}}{2\delta\Gamma} \int_{\Omega_\epsilon}\frac{|\nabla \tilde \phi|^2}{
  (1+\epsilon R)^2}\dd X \,=\, \frac{t\dot{\bar r}}{2\bar r}\int_{\Omega_\epsilon}
  \frac{|\nabla \tilde \phi|^2}{1+\epsilon R}\dd X\,,
\]
where we used the expression \eqref{dereps} of $t\dot\epsilon$. Next, we prefer
including the term involving $t\partial_t W_\epsilon$ in $I_4$ rather than $I_6$,
because it will be combined with the diffusive terms in $I_4$ to obtain negative quantities
that will allow us to control the evolution of the energy. Summarizing, if we define
\begin{equation}\label{hatI4def}
\begin{split}
  \hat I_4 \,&=\, -\int_{\Omega_\epsilon} W_\epsilon |\nabla\tilde\eta|^2\dd X -
  \int_{\Omega_\epsilon} (\nabla W_\epsilon\cdot\nabla\tilde\eta)\tilde\eta \dd X
   - \int_{\Omega_\epsilon} V_\epsilon \tilde\eta^2\dd X\\
  &\quad\, -\frac{\epsilon}{2}\int_{\Omega_\epsilon} \partial_R\bigl(W_\epsilon(1+\epsilon R)
  \bigr) \tilde\zeta^2\dd X  + \frac12 \int_{\Omega_\epsilon} t(\partial_t W_\epsilon)
  \tilde \eta^2\dd X\,,
\end{split}
\end{equation}
and
\begin{equation}\label{hatI6def}
  \hat I_6 \,=\, \frac{\epsilon \bar r \dot{\bar r}}{\delta\Gamma}
  \int_{\Omega_\epsilon} W_\epsilon \tilde \eta\partial_R \tilde \eta\dd X  +
  \frac{t\dot{\bar r}}{\bar r}\,E^\kin_\epsilon[\tilde\eta] + 
  \frac{\epsilon \bar r \dot{\tilde z}} {\delta\Gamma}\int_{\Omega_\epsilon}
  \bigl(W_\epsilon \tilde \eta -\tilde\phi\bigr)\,\bigl(\partial_Z\eta_* +
  \delta\partial_Z\tilde\eta\bigr)\dd X\,,
\end{equation}
we obtain the identity $t\partial_t E_\epsilon = I_1 + I_2 + I_3 + \hat I_4
+ I_5 + \hat I_6$, which we exploit in Sections~\ref{ssec46}--\ref{ssec49}. 

\subsection{Bounds on the stream function}\label{ssec45}

In this section we collect a few estimates on the stream function
$\phi = \BS^\epsilon[\eta]$, where $\BS^\epsilon$ is the $\epsilon$-dependent
Biot-Savart operator \eqref{BSeps}. We are especially interested in
bounds on the velocity field $U \,=\, (U_R,U_Z)$ defined by \eqref{Udef}.

\begin{lem}\label{lemphi1}
There exists a constant $C > 0$ such that, for all $\epsilon \in (0,1)$,
\begin{equation}\label{phibd1}
  \Bigl|\frac{\partial_R\phi}{1+\epsilon R}\Bigr| + \Bigl|\frac{\partial_Z\phi}{1+
  \epsilon R}\Bigr| \,\le\, \int_{\Omega_\epsilon} \frac{C}{\sqrt{(R{-}R')^2 + (Z{-}Z')^2}}
  \,|\eta(R',Z')|\dd X'\,.
\end{equation}
In particular, for any $q > 2$, we have $\|U\|_{L^q} \le C_q \|\eta\|_{\cX_\epsilon}$
where $U$ is the velocity field \eqref{Udef}. 
\end{lem}

\begin{proof}
Estimate \eqref{phibd1} is established in the proof of \cite[Lemma~4.1]{GS2}, which in
turn relies on \cite[Proposition~2.3]{GS1}. Using the Hardy-Littlewood-Sobolev inequality,
we deduce from \eqref{phibd1} that $\|U\|_{L^q} \le C_q \|\eta\|_{L^p}$ if $q > 2$ and
$p \in (1,2)$ satisfy the relation $1/p = 1/q + 1/2$. Finally, the lower bound on
$W_\epsilon$ in \eqref{Wunifbd} implies that $\|\eta\|_{L^p} \le C \|\eta\|_{\cX_\epsilon}$
for any $p \in [1,2]$. 
\end{proof}

The particular case where $\eta = \eta_*$ is the approximate solution \eqref{etaapp1} plays
an important role. 

\begin{lem}\label{lemphi2}
The following estimates hold for the stream function $\phi_* = \BS^\epsilon[\eta_*]$\:
\begin{equation}\label{phibd2}
  \Bigl|\frac{\partial_R\phi_*}{1+\epsilon R}\Bigr| + \Bigl|\frac{\partial_Z\phi_*}{1+
  \epsilon R}\Bigr| \,\le\, \frac{C}{1+\rho+\epsilon^2\rho^3}\,, \qquad
  \Bigl|\frac{\partial_Z\phi_*}{(1+\epsilon R)^2}\Bigr| \,\le\, \frac{C}{1+\rho
  +\epsilon^3\rho^4}\,,
\end{equation}
where $\rho = (R^2{+}Z^2)^{1/2}$. 
\end{lem}

\begin{proof}
In the region where $\rho \le 1/(2\epsilon)$, we can use estimate \eqref{phibd1}
with $\eta = \eta_*$. Since $\eta_*$ satisfies the Gaussian bound
\eqref{etaphibd}, we easily deduce that $|U| \le C(1+\rho)^{-1}$, which gives
estimate \eqref{phibd2} in that case. We now concentrate on the region
$\rho \ge 1/(2\epsilon)$, where a more careful analysis is needed. We start from
the formulas \eqref{derphi} with $\eta = \eta_*$, and we first estimate the
vertical derivative $\partial_Z\phi_*$. Since $|\tilde F(s)| \le Cs^{-3/2}$ for
all $s > 0$, we see that
\begin{equation}\label{phibdaux1}
  \Bigl|\frac{\partial_Z\phi_*}{(1+\epsilon R)^2}\Bigr| \,\le\, \frac{C}{\epsilon^3}
  \int_{\Omega_\epsilon} \frac{(1{+}\epsilon R')^2\,|\eta_*(R',Z')|}{\bigl((R{-}R')^2 +
  (Z{-}Z')^2\bigr)^2}\dd R'\dd Z'\,.
\end{equation}
Note that the integral is, in fact, taken over the support of $\eta_*$, which is
included in the ball $\rho' := ({R'}^2{+}{Z'}^2)^{1/2} \le 2 \epsilon^{-\sigma_0}$
where $\sigma_0 < 1$. In particular we can disregard the factor $(1{+}\epsilon R')^2$
in the numerator, and the denominator is always larger that $\rho^4/2$ if $\epsilon$
is sufficiently small. So the right-hand side of \eqref{phibdaux1} is bounded by
$C \epsilon^{-3}\rho^{-4}$ when $\rho \ge 1/(2\epsilon)$, which concludes the
proof of the second inequality in \eqref{phibd2}. Since $1 + \epsilon R \le 1 +
\epsilon \rho$, the estimate on $\partial_Z \phi_*/(1{+}\epsilon R)$ in \eqref{phibd2}
follows immediately. 

To conclude the proof of the first inequality in \eqref{phibd2}, we must estimate the
quantity $\partial_R \phi_*$ which contains an additional term given by the last line
in \eqref{derphi}. In the region where $\rho \ge 1/(2\epsilon)$, using the fact that
$|F(s)| + |\tilde F(s)| \le Cs^{-3/2}$, we see that the contribution of that term to
the vertical speed $U_Z = \partial_R \phi_*/(1{+}\epsilon R)$ is bounded by
\[
  \frac{C}{\epsilon^2}\int_{\Omega_\epsilon} \frac{(1{+}\epsilon R')^2\,|\eta_*(R',Z')|}{\bigl(
  (R{-}R')^2 + (Z{-}Z')^2\bigr)^{3/2}}\dd R'\dd Z' \,\le\, \frac{C}{\epsilon^2\rho^3}\,.
\]
The proof of \eqref{phibd2} is thus complete. 
\end{proof}

\subsection{Control of the advection terms}\label{ssec46}

In what follows we always assume that $\delta > 0$ is sufficiently small and
that $\epsilon^2 \lesssim \delta^{1-\sigma}$ for some small $\sigma > 0$, see
Remark~\ref{crossrem}. As in Lemma~\ref{Wlem}, we also suppose that the exponent
$\sigma_1 > 0$ is small enough. We first estimate the advection terms $I_1$,
$I_2$ defined in \eqref{I1def}, \eqref{I2def}. These terms are potentially
dangerous because they include a factor $1/\delta$ which is very large in the
vanishing viscosity limit, but the energy functional \eqref{Edef} was designed
precisely so that these contributions can be controlled.

\begin{lem}\label{I1lem}
There exist $\gamma_1 > 0$ and $C > 0$ such that
\begin{equation}\label{I1est}
  |I_1| \,\le\, C\epsilon^{\gamma_1}\,\|\tilde\eta\|_{\cX_\epsilon}^2 +
  \frac{C\epsilon^2}{\delta}\int_{\Omega_\epsilon''} W_\epsilon \tilde\eta^2\dd X\,.
\end{equation}
\end{lem}

\begin{proof}
To exploit the properties of the weight $W_\epsilon$, we decompose the integral
\eqref{I1def} defining $I_1$ in three pieces, which correspond to the 
subdomains \eqref{Ompdef}. If $(R,Z) \in \Omega_\epsilon'$, we know from \eqref{Wdef},
\eqref{Thetadef} that 
\begin{equation}\label{I1aux0}
  W_\epsilon \,=\, \frac{\Phi_\epsilon'(\zeta_*)}{1+\epsilon R}\,, \qquad
  \phi_* - \frac{\bar r \dot{\bar z}_*}{2\Gamma}\,(1+\epsilon R)^2 \,=\,
  \Phi_\epsilon(\zeta_*) + \Theta\,, 
\end{equation}
where $\Theta$ is a remainder term that is studied in Proposition~\ref{Phiprop}.
It follows that
\[
  \Bigl\{W_\epsilon(1+\epsilon R)\,,\,\phi_* - \frac{\bar r \dot{\bar z}_*}{2\Gamma}\,
  (1+\epsilon R)^2\Bigr\} \,=\, \bigl\{\Phi_\epsilon'(\zeta_*)\,,\,\Phi_\epsilon(\zeta_*)
  + \Theta\bigr\} \,=\, \bigl\{\Phi_\epsilon'(\zeta_*)\,,\,\Theta\bigr\}\,,
\]
where the right-hand side can be controlled using the bounds \eqref{Thetaest} on
$\Theta$ and the estimates \eqref{Phipest}, \eqref{Wapprox} on the weight $W_\epsilon$
in $\Omega_\epsilon'$. This gives, for some integer $N$ and any $\gamma_3 \in (2,3)$,  
\begin{equation}\label{I1aux1}
  \bigl|\bigl\{\Phi_\epsilon'(\zeta_*)\,,\,\Theta\bigr\}\bigr| \,\le\,
  C\bigl(\epsilon\delta + \epsilon^{\gamma_3}\bigr)(1+\rho)^N \,W_\epsilon
  \,\le\, C\bigl(\epsilon\delta + \epsilon^{\gamma_3}\bigr)\,\epsilon^{-N\sigma_1}
  \,W_\epsilon\,,
\end{equation}
where we used the fact that $1+\rho \le 2 \epsilon^{-\sigma_1}$ when $(R,Z) \in
\Omega_\epsilon'$. Since $\tilde\zeta \approx \tilde \eta$ in $\Omega_\epsilon'$ 
and since $\delta^{-1} \lesssim \epsilon^{-2/(1-\sigma)}$ in the parameter regime
we consider, it follows from \eqref{I1aux1} that
\begin{equation}\label{I1aux2}
  \frac{1}{\delta}\int_{\Omega_\epsilon'}\bigl|\bigl\{\Phi_\epsilon'(\zeta_*)\,,\,\Theta
  \bigr\}\bigr|\,\tilde\zeta^2 \dd X \,\le\, C\Bigl(\epsilon + \frac{\epsilon^{\gamma_3}}{
  \delta}\Bigr)\,\epsilon^{-N\sigma_1}\int_{\Omega_\epsilon'}W_\epsilon \tilde\eta^2 \dd X
  \,\le\, C\epsilon^{\gamma_1}\,\|\tilde\eta\|_{\cX_\epsilon}^2\,,
\end{equation}
where $\gamma_1$ is taken so that $0 < \gamma_1 < \gamma_3 - 2/(1{-}\sigma) - N\sigma_1$.
As $\gamma_3 < 3$ is arbitrary, such a choice is always possible if we assume that
$\sigma > 0$ and $\sigma_1 > 0$ are small enough. 

We next consider the intermediate region $\Omega_\epsilon''$ in which $W_\epsilon(1+\epsilon R) =
\chi_2(\epsilon R)\,\exp\bigl(\epsilon^{-2\sigma_1}/4\bigr)$, where $\chi_2(x) =
(1+x)\chi_1(x)$. In that region, we thus have
\[
  J_\epsilon \,:=\,\Bigl\{W_\epsilon(1+\epsilon R)\,,\,\phi_* - \frac{\bar r \dot{\bar z}_*}{
  2\Gamma}\,(1+\epsilon R)^2\Bigr\} \,=\, \epsilon \chi_2'(\epsilon R) \,\exp\bigl(
  \epsilon^{-2\sigma_1}/4\bigr)\partial_Z \phi_*\,.
\]
Since $\chi_2(x) = 1$ when $|x| \le \frac12$, the quantity $J_\epsilon$ vanishes when
$\rho := (R^2{+}Z^2)^{1/2} \le 1/(2\epsilon)$. In the region where $1/(2\epsilon) \le \rho \le
\epsilon^{-\sigma_2}$, we know from \eqref{phibd2} that $|\partial_Z\phi_*/(1{+}\epsilon R)^2|
\le C\epsilon^{-3}\rho^{-4} \le C \epsilon$, and that $W_\epsilon \approx \exp\bigl(
\epsilon^{-2\sigma_1}/4\bigr)$. Since $\chi_2'$ is a bounded function, we deduce
\begin{equation}\label{I1aux3}
  \frac{1}{\delta}\int_{\Omega_\epsilon''} |J_\epsilon|\,\tilde \zeta^2\dd X \,=\,
  \frac{1}{\delta}\int_{\Omega_\epsilon''} \frac{|J_\epsilon|\,\tilde\eta^2}{
  (1{+}\epsilon R)^2}\dd X \,\le\, \frac{C\epsilon^2}{\delta}\int_{\Omega_\epsilon''}
  W_\epsilon \tilde\eta^2\dd X\,.
\end{equation}

Finally, in $\Omega_\epsilon'''$ we have $W_\epsilon(1+\epsilon R) \,=\,
\chi_2(\epsilon R)\hat W_\epsilon$ where $\hat W_\epsilon = \exp(\rho^{2\gamma}/4)$, so 
that 
\[
  J_\epsilon \,=\, \epsilon \chi_2'(\epsilon R) \hat W_\epsilon\, \partial_Z \phi_*
  + \frac{\epsilon \bar r \dot{\bar z}_*}{\Gamma}\, \chi_1(\epsilon R)(1+\epsilon R)^2
  \partial_Z \hat W_\epsilon + \chi_2(\epsilon R) \bigl\{\hat W_\epsilon\,,\,
  \phi_*\bigr\}\,.
\]
The first term in the right-hand side is estimated as above, with the difference
that we now have the improved bound $|\partial_Z\phi_*/(1{+}\epsilon R)^2| \le
C\epsilon^{-3}\rho^{-4} \le C \epsilon^{4\sigma_2-3}$. For the second one we observe that
\begin{equation}\label{I1aux4}
  \bigl|\partial_R \hat W_\epsilon\bigr| + \bigl|\partial_Z \hat W_\epsilon\bigr| \,\le\, 
  \gamma \,\rho^{2\gamma-1}\,\hat W_\epsilon \,\le\, \gamma \,\epsilon^{\sigma_2-2\sigma_1}
  \,\hat W_\epsilon\,, \quad \hbox{since}~\, \rho \ge \epsilon^{-\sigma_2}\,,
\end{equation}
and the last term is estimated using \eqref{I1aux4} and the first bound in \eqref{phibd2}.
Altogether we find
\begin{equation}\label{I1aux5}
  \frac{1}{\delta}\int_{\Omega_\epsilon'''} |J_\epsilon|\,\tilde \zeta^2\dd X \,\le\, 
  \frac{C}{\delta}\int_{\Omega_\epsilon'''}\Bigl(\frac{1}{\epsilon^2\rho^4} +
  \frac{\bar r |\dot{\bar z}_*|}{\Gamma} \frac{\epsilon}{\rho^{1-2\gamma}} + 
  \frac{1}{\epsilon^2\rho^{4-2\gamma}}\Bigr)W_\epsilon  \tilde\eta^2 \dd X \,\le\, 
  C\epsilon^{\gamma_1}\,\|\tilde\eta\|_{\cX_\epsilon}^2\,,
\end{equation}
provided $0 < \gamma_1 < \sigma_2 + 1 - 2 \sigma_1 -  2/(1{-}\sigma)$. Since $\sigma_2 > 1$,
such a choice is again possible if $\sigma > 0$ and $\sigma_1 > 0$ are small enough. 
Combining \eqref{I1aux2}, \eqref{I1aux3}, \eqref{I1aux5}, we arrive at \eqref{I1est}. 
\end{proof}

\begin{lem}\label{I2lem}
There exist $\gamma_1 > 0$ and $C > 0$ such that
\begin{equation}\label{I2est}
  |I_2| \,\le\, C\epsilon^{\gamma_1}\,\|\tilde\eta\|_{\cX_\epsilon}^2\,.
\end{equation}
\end{lem}

\begin{proof}
In $\Omega_\epsilon'$ we have $W_\epsilon(1+\epsilon R) = \Phi_\epsilon'(\zeta_*)$
by \eqref{I1aux0}, hence $W_\epsilon(1+\epsilon R)\bigl\{\tilde\phi\,,\,\zeta_*\bigr\}
= \bigl\{\tilde\phi\,,\,\Phi_\epsilon(\zeta_*)\bigr\}$. Using the second relation
in \eqref{I1aux0}, we deduce that
\begin{equation}\label{I2aux0}
  \Bigl\{\tilde\phi\,,\,\phi_* - \frac{\bar r \dot{\bar z}_*}{2\Gamma}\,(1+\epsilon R)^2
  \Bigr\} - W_\epsilon(1+\epsilon R) \bigl\{\tilde\phi\,,\,\zeta_*\bigr\} \,=\,
  \bigl\{\tilde\phi\,,\,\Theta\bigr\}\,.
\end{equation}
The first-order derivatives of $\Theta$ are estimated in Proposition~\ref{Phiprop}.
Proceeding as in the previous lemma, we thus obtain
\begin{equation}\label{I2aux1}
  \frac{1}{\delta}\int_{\Omega_\epsilon'}\bigl|\bigl\{\tilde\phi\,,\,\Theta\bigr\}
  \bigr|\,|\tilde\zeta| \dd X \,\le\, C\Bigl(\epsilon + \frac{\epsilon^{\gamma_3}}{
  \delta}\Bigr)\,\epsilon^{-N\sigma_1}\int_{\Omega_\epsilon'}
  \frac{|\nabla\tilde\phi|}{1{+}\epsilon R}\,|\tilde\eta| \dd X \,\le\, 
  C\epsilon^{\gamma_1}\,\|\tilde\eta\|_{\cX_\epsilon}^2\,,
\end{equation}
where $0 < \gamma_1 < \gamma_3 - 2/(1{-}\sigma) - N\sigma_1$. In the last step, we
used H\"older's inequality with exponents $3$ and $3/2$, and we invoked
Lemma~\ref{lemphi1} to control the $L^3$ norm of $\nabla\tilde\phi/(1+\epsilon R)$. 

In $\cD_\epsilon := \Omega_\epsilon \setminus \Omega_\epsilon'$, we consider both terms in
the left-hand side of \eqref{I2aux0} separately. The contribution of the first one
to $I_2$ is estimated by
\begin{equation}\label{I2aux2}
  \frac{1}{\delta}\int_{\cD_\epsilon} \frac{|\nabla\tilde \phi|\,|\nabla\phi_*|}{1+\epsilon R}|
  \,\tilde\eta|\dd X + \frac{\epsilon \bar r|\dot{\bar z}_*|}{\delta\Gamma}
  \int_{\cD_\epsilon} |\partial_Z \tilde\phi|\,|\tilde\eta|\dd X \,=\, \cO\bigl(\epsilon^\infty
  \|\tilde\eta\|_{\cX_\epsilon}^2\bigr)\,,
\end{equation}
because $|\nabla\phi_*| \le C$ by \eqref{phibd2}, $\|\nabla\tilde\phi/(1{+}\epsilon R)
\|_{L^3} \le C\|\tilde\eta\|_{\cX_\epsilon}$ by Lemma~\ref{lemphi1}, and 
\[
  \|\tilde \eta\|_{L^{3/2}(\cD_\epsilon)} \,\le\, \biggl(\int_{\cD_\epsilon} W_\epsilon \tilde
  \eta^2\dd X\biggr)^{1/2} \biggl(\int_{\cD_\epsilon} W_\epsilon^{-3}\dd X\biggr)^{1/6}
  \,=\, \cO\bigl(\epsilon^\infty\|\tilde\eta\|_{\cX_\epsilon}\bigr)\,.
\]
The second term in the left-hand side of \eqref{I2aux0} is nonzero only
if $\rho \le 2\epsilon^{-\sigma_0}$, in view of \eqref{etaapp1}. In that region,
we know that $W_\epsilon|\nabla \zeta_*| \le C(1+\rho)^N$ for some integer $N$,
because $W_\epsilon$ satisfies the upper bound in \eqref{Wunifbd} and $\eta_*$
belongs to the space $\cZ$ defined in \eqref{Zdef}. The contribution of that term
to $I_2$ can therefore be estimated in the same way as above: 
\begin{equation}\label{I2aux3}
  \frac{1}{\delta}\int_{\cD_\epsilon} W_\epsilon\,|\bigl\{\tilde\phi,\zeta_*\bigr\}|
  \,|\tilde\zeta|\dd X \,\le\, \frac{C}{\delta} \int_{\cD_\epsilon}
  \frac{|\nabla\tilde\phi|\,|\tilde\eta|}{1+\epsilon R}
  \,(1+\rho)^N\dd X \,=\, \cO\bigl(\epsilon^\infty\|\tilde\eta\|_{\cX_\epsilon}^2\bigr)\,.
\end{equation}
Combining \eqref{I2aux1}, \eqref{I2aux2}, \eqref{I2aux3}, we obtain \eqref{I2est}. 
\end{proof}

\subsection{Control of the diffusive terms}\label{ssec47}

Our next task is to estimate the diffusive terms collected in \eqref{hatI4def}. 
To formulate the result, we introduce the continuous function $\rho_\gamma : \R^2\times
\R_+ \to \R_+$ defined by
\begin{equation}\label{rhogamdef}
  \rho_\gamma(R,Z,\epsilon) \,=\, \begin{cases} \,\rho & \hbox{if} \quad\rho \le
  \epsilon^{-\sigma_1}\,,\\ \,\epsilon^{-\sigma_1} & \hbox{if}\quad \epsilon^{-\sigma_1} <
  \rho < \epsilon^{-\sigma_2}\,,\\ \,\rho^\gamma & \hbox{if} \quad \rho \ge
  \epsilon^{-\sigma_2}\,,\end{cases}
\end{equation}
where as usual $\rho = (R^2+Z^2)^{1/2}$. Our goal in this section is: 

\begin{prop}\label{I4prop}
There exist $\kappa > 0$ and $C > 0$ such that
\begin{equation}\label{I4est}
  \hat I_4 \,\le\, -\kappa \int_{\Omega_\epsilon} W_\epsilon\Bigl(|\nabla\tilde\eta|^2 +
  \rho_\gamma^2 \tilde\eta^2 + \tilde \eta^2\Bigr)\dd X + C\bigl(\mu_0^2 + \mu_1^2 +
  \mu_2^2\bigr)\,,
\end{equation}
where $\mu_0,\mu_1,\mu_2$ are defined in \eqref{mudef}, \eqref{mu2def}. 
\end{prop}

The proof of Proposition~\ref{I4prop} requires several steps. We first control
the term in $\hat I_4$ that involves the time derivative of the weight function
$W_\epsilon$.

\begin{lem}\label{timederlem}
There exist $C > 0$ and $\gamma_1 > 0$ such that
\begin{equation}\label{timederest}
  \int_{\Omega_\epsilon} t(\partial_t W_\epsilon)\tilde \eta^2\dd X \,\le\,
  -\frac{\sigma_1}{5}\int_{\Omega_\epsilon''} W_\epsilon \rho_\gamma^2
  \tilde \eta^2\dd X + C \int_{\Omega_\epsilon'''}W_\epsilon \tilde\eta^2\dd X
  + C\epsilon^{\gamma_1} \|\tilde\eta\|_{\cX_\epsilon}^2\,.
\end{equation}
\end{lem}

\begin{proof}
Following \eqref{Wdef} we decompose $W_\epsilon(R,Z) = \chi_1(\epsilon R)\,\hat W_\epsilon(R,Z)$,
so that
\begin{equation}\label{timeder1}
  t\partial_t W_\epsilon \,=\, \chi_1(\epsilon R) \,t\partial_t\hat W_\epsilon(R,Z) +
  t\dot \epsilon R \chi_1'(\epsilon R)\,\hat W_\epsilon(R,Z)\,.
\end{equation}
We first estimate the right-hand side in the region $\Omega_\epsilon'$ defined by \eqref{Ompdef},
where $\hat W_\epsilon = \Phi_\epsilon'(\zeta_*)$. As $\Phi_\epsilon = \Phi_0 + \epsilon^2 \Phi_2$
according to \eqref{Phiexpand}, we have $t\partial_t \hat W_\epsilon = \Phi_\epsilon''(\zeta_*)\,
t\partial_t \zeta_* + 2t\epsilon \dot \epsilon\,\Phi_2'(\zeta_*)$ in that region. We recall 
that $2t\dot\epsilon = \epsilon(1 + \cO(\epsilon^2))$ by \eqref{dereps}, and that the functions
$\Phi_0, \Phi_2$ satisfy the estimates \eqref{Phi0derest}, \eqref{Phi2derest}. It follows
immediately that $|t\epsilon \dot \epsilon\,\Phi_2'(\zeta_*)| \le C \epsilon^{2-N\sigma_1}
\hat W_\epsilon\le C \epsilon W_\epsilon$. Moreover, since $\zeta_* = \eta_*/(1{+}\epsilon R)$
with $\eta_* = \eta_\app$ in $\Omega_\epsilon'$, we also have $|\Phi_\epsilon''(\zeta_*)
t\partial_t \zeta_*| \le C\epsilon(1+\rho)^N\hat W_\epsilon \le C\epsilon^{\gamma_1}W_\epsilon$,
provided $0 < \gamma_1 < 1 -N\sigma_1$. Finally, the last term in \eqref{timeder1} is bounded
by $C\epsilon\rho W_\epsilon \le C\epsilon^{1-\sigma_1}W_\epsilon$. Altogether
we have shown that $|t\partial_t W_\epsilon| \le C\epsilon^{\gamma_1}W_\epsilon$
in $\Omega_\epsilon'$.

In the intermediate region $\Omega_\epsilon''$ we have $\hat W_\epsilon = \exp\bigl(
\epsilon^{-2\sigma_1}/4\bigr)$ and $\rho_\gamma = \epsilon^{-\sigma_1}$, so that
\[
  \,t\partial_t\hat W_\epsilon \,=\, -\frac{\sigma_1}{2}\,\exp\bigl(\epsilon^{-2\sigma_1}/4\bigr)
  \,\frac{t\dot\epsilon}{\epsilon^{2\sigma_1 + 1}} \,=\, -\frac{\sigma_1}{2}\,\hat W_\epsilon\, 
  \rho_\gamma^2\,\frac{t\dot\epsilon}{\epsilon} \,\approx\, -\frac{\sigma_1}{4}\,\hat W_\epsilon\, 
  \rho_\gamma^2\,.
\]
Since $|t\dot \epsilon R \chi_1'(\epsilon R)| \le |\epsilon R \chi_1'(\epsilon R)| \le C$, it
follows that $t\partial_t W_\epsilon\le -(\sigma_1/5)W_\epsilon\rho_\gamma^2$ in $\Omega_\epsilon''$.
Finally, in the exterior region $\Omega_\epsilon'''$, the function $\hat W_\epsilon = 
\exp(\rho^{2\gamma}/4)$ does not depend on time, and we deduce from \eqref{timeder1} that 
$|t\partial_t W_\epsilon| \le C W_\epsilon$. Collecting all these estimates, we arrive at 
\eqref{timederest}. 
\end{proof}

We next consider the term involving $\tilde\zeta$ in \eqref{hatI4def}. 

\begin{lem}\label{tildezetalem}
There exist $C > 0$ and $\gamma_1 > 0$ such that
\begin{equation}\label{tildezetaest}
 -\frac{\epsilon}{2}\int_{\Omega_\epsilon} \partial_R\bigl(W_\epsilon(1+\epsilon R)
  \bigr)\tilde\zeta^2\dd X \,\le\, -\frac{\epsilon^2}{4}\int_{\Omega_\epsilon} 
  W_\epsilon\tilde\zeta^2\dd X + C\epsilon^{\gamma_1} \|\tilde\eta\|_{\cX_\epsilon}^2\,.
\end{equation}
\end{lem}

\begin{proof}
If $\cD_\epsilon$ denotes any of the three regions defined in \eqref{Ompdef}, we have
\begin{align}\label{tildezeta1}
  -\frac{\epsilon}{2}\int_{\cD_\epsilon} \partial_R\bigl(W_\epsilon(1+\epsilon R)
  \bigr)\tilde\zeta^2\dd X \,&=\, -\frac{\epsilon^2}{2}\int_{\cD_\epsilon} W_\epsilon
  \tilde\zeta^2\dd X - \frac{\epsilon}{2}\int_{\cD_\epsilon} \bigl(\partial_R W_\epsilon
  \bigr)\tilde \zeta \tilde \eta\dd X \\ \label{tildezeta2}
  \,&\le\, -\frac{\epsilon^2}{4}\int_{\cD_\epsilon} W_\epsilon\tilde\zeta^2\dd X
  + \frac{1}{4} \int_{\cD_\epsilon}\frac{(\partial_R W_\epsilon)^2}{W_\epsilon}\,
  \tilde\eta^2 \dd X\,,    
\end{align}
where in the second line we used Young's inequality. In the inner region $\Omega_\epsilon'$
we observe that $\tilde\zeta \approx \tilde\eta$, because $|\epsilon R| \le 2\epsilon^{1-\sigma_1}
\ll 1$. Moreover we have $\epsilon|\partial_R W_\epsilon| \le C\epsilon^{\gamma_1}W_\epsilon$ for
some $\gamma_1 > 0$, so taking $\cD_\epsilon = \Omega_\epsilon'$ and using \eqref{tildezeta1} we obtain
the analogue of \eqref{tildezetaest} in that region. Outside $\Omega_\epsilon'$, we cannot
directly compare $\tilde\zeta$ and $\tilde\eta$, so we prefer using inequality \eqref{tildezeta2}. 
In the intermediate region $\Omega_\epsilon''$, we have $|\partial_R W_\epsilon| \le C\epsilon
W_\epsilon$ by \eqref{Wdef}, and \eqref{tildezetaest} easily follows. Finally, in the exterior 
region $\Omega_\epsilon'''$, we observe that
\[
  \partial_R W_\epsilon \,=\, \biggl(\frac{\epsilon \chi_1'(\epsilon R)}{\chi_1(\epsilon R)}
  + \frac{\gamma R}{2}\,\rho^{2\gamma - 2}\biggr)W_\epsilon\,.
\]
Taking $\sigma_1$ small enough so that $\gamma \equiv \sigma_1/\sigma_2 < 1/2$, and using the
fact that $\rho \ge \epsilon^{-\sigma_2}$ in $\Omega_\epsilon'''$, we deduce that
$|\partial_R W_\epsilon| \le C\epsilon^{\gamma_1} W_\epsilon$ for some $\gamma_1 > 0$,
and this leads to \eqref{tildezetaest}. The proof is thus complete. 
\end{proof}

To conclude the proof of Proposition~\ref{I4prop}, we consider the quadratic form
given by the first line of \eqref{hatI4def}, namely
\begin{equation}\label{Qepsdef}
  Q_\epsilon[\eta] \,=\, \int_{\Omega_\epsilon} W_\epsilon |\nabla\eta|^2\dd X +
  \int_{\Omega_\epsilon} (\nabla W_\epsilon\cdot\nabla\eta)\eta \dd X
   + \int_{\Omega_\epsilon} V_\epsilon \eta^2\dd X\,,
\end{equation}
where $V_\epsilon$ is defined in \eqref{Vdef}. Taking formally the limit $\epsilon \to 0$
in \eqref{Qepsdef}, we obtain using \eqref{Wapprox}
\begin{equation}\label{Q0def}
  Q_0[\eta] \,=\, \int_{\R^2} A |\nabla\eta|^2\dd X + \int_{\R^2} (\nabla A \cdot\nabla\eta)
  \eta \dd X + \int_{\R^2} V \eta^2\dd X\,,
\end{equation}
where $A$ is defined by \eqref{Adef} and $V = \frac14(R\partial_R + Z\partial_Z)A - \frac12 A -1$. 
The limiting quadratic form \eqref{Q0def} is carefully studied in our previous work \cite{GS3},
and we have the following result\:

\begin{prop}\label{Q0prop}
There exists constants $C_8 > 2$ and $C_9 > 0$ such that, for all $\eta\in \cX_0$ with
$\rho \eta \in \cX_0$ and $\nabla\eta \in \cX_0^2$, we have
\begin{equation}\label{Q0ineq}
  \|\nabla\eta\|_{\cX_0}^2 + \|\rho\eta\|_{\cX_0}^2 + \|\eta\|_{\cX_0}^2
  \,\le\, C_8 Q_0[\eta] + C_9\bigl(\mu_0^2 + \mu_1^2 + \mu_2^2\bigr)\,,
\end{equation}
where $\mu_0 = \int_{\R^2} \eta\dd X$, $\mu_1 = \int_{\R^2} R\eta\dd X$, $\mu_2 =
\int_{\R^2} Z\eta\dd X$.  
\end{prop}

\begin{proof}
In \cite[Theorem~4.2]{GS3} we prove that there exists $\delta_0 > 0$ such that
$Q_0[\eta] \ge \delta_0 \|\eta\|_{\cX_0}^2$ for any $\eta \in \cX_0$ such that
$\mu_0 = \mu_1 = \mu_2 = 0$. On the other hand, if we apply Young's inequality
to the middle term in the right-hand side of \eqref{Q0def}, we obtain the
lower bound
\[
  Q_0[\eta] \,\ge\, \frac14\int_{\R^2} A |\nabla\eta|^2\dd X + \int_{\R^2}
  \Bigl(V - \frac{|\nabla A|^2}{3A}\Bigr)\eta^2\dd X \,\ge\, \frac14\|\nabla\eta\|_{\cX_0}^2
  + \frac{1}{24}\|\rho\eta\|_{\cX_0}^2 - C\|\eta\|_{\cX_0}^2\,, 
\]
because a direct calculation reveals that $V/A - |\nabla A|^2/(3A^2) \ge \rho^2/(24) - C$ for
some constant $C > 0$. Taking a convex combination of both estimates, we see
that there exists $C_8 > 0$ such that 
\begin{equation}\label{Q0ineqbis}
  \|\nabla\eta\|_{\cX_0}^2 + \|\rho\eta\|_{\cX_0}^2 + \|\eta\|_{\cX_0}^2 \,\le\,
  C_8 Q_0[\eta]\,,
\end{equation}
whenever $\eta \in \cX_0$ satisfies $\mu_0 = \mu_1 = \mu_2 = 0$. It remains to
deduce \eqref{Q0ineq} from \eqref{Q0ineqbis}, which is easily done using
exactly the same arguments as in the proof of Proposition~\ref{E0prop}.
\end{proof}

The analogue of Proposition~\ref{Q0prop} for the full quadratic form
\eqref{Qepsdef} is the following statement, whose proof is postponed
to Section~\ref{ssecB3}. 

\begin{prop}\label{Qepsprop}
There exists constants $C_{10} > 2$ and $C_{11} > 0$ such that, for all
sufficiently small $\epsilon > 0$ and all $\eta\in\cX_\epsilon$ with
$\rho_\gamma \eta \in \cX_\epsilon$ and $\nabla\eta \in \cX_\epsilon^2$, we
have
\begin{equation}\label{Qepsineq}
  \|\nabla\eta\|_{\cX_\epsilon}^2 + \|\eta\|_{\cX_\epsilon}^2 + \int_{\Omega_\epsilon' \cup
  \Omega_\epsilon'''}W_\epsilon \rho_\gamma^2 \eta^2\dd X \,\le\, C_{10} Q_\epsilon[\eta] + C_{11}
  \Bigl(\mu^2 + \int_{\Omega_\epsilon''} W_\epsilon \eta^2\dd X\Bigr)\,,
\end{equation}
where $\mu^2 = \mu_0^2 + \mu_1^2 + \mu_2^2$ and $\mu_0 = \int_{\Omega_\epsilon} \eta\dd X$,
$\mu_1 = \int_{\Omega_\epsilon} R\eta\dd X$, $\mu_2 = \int_{\Omega_\epsilon} Z\eta\dd X$.  
\end{prop}

\begin{proof}[\bf End of the proof of Proposition~\ref{I4prop}] In view of
\eqref{hatI4def} and \eqref{Qepsdef} we have
\[
  \hat I_4 \,=\, -Q_\epsilon[\tilde\eta] -\frac{\epsilon}{2}\int_{\Omega_\epsilon} \partial_R
  \bigl(W_\epsilon(1+\epsilon R)\bigr) \tilde\zeta^2\dd X  + \frac12 \int_{\Omega_\epsilon}
  t(\partial_t W_\epsilon)\tilde \eta^2\dd X\,. 
\]
The three terms in the right-hand side are estimated using \eqref{Qepsineq},
\eqref{tildezetaest}, and \eqref{timederest}, respectively. Taking $\epsilon > 0$
sufficiently small and recalling that $\rho_\gamma \ge \epsilon^{-\sigma_1} \gg 1$
outside the inner region $\Omega_\epsilon'$, we arrive at \eqref{I4est}. The slight
discrepancy between the definitions of $\mu_1$ in \eqref{mudef} and in 
Proposition~\ref{Qepsprop} is completely harmless. 
\end{proof}

\subsection{Control of the remaining terms}\label{ssec48}

In this section, we estimate the remaining terms $I_3$, $I_5$, and $\hat I_6$ defined
in \eqref{I3def}, \eqref{I5def}, and \eqref{hatI6def}, respectively. 

\medskip\noindent{\bf Control of $I_3$.} We deduce from \eqref{I3def} that
\begin{equation}\label{I3bd1}
  |I_3| \,\le\, \int_{\Omega_\epsilon} \frac{|\nabla\tilde\phi|}{1{+}\epsilon R}
  \,|\tilde\eta| \bigl|\nabla(W_\epsilon \tilde\eta)\bigr|\dd X \,\le\,
  \int_{\Omega_\epsilon} \frac{|\nabla\tilde\phi|}{1{+}\epsilon R}
  \,|\tilde\eta|\Bigl(|\tilde\eta||\nabla W_\epsilon| + W_\epsilon|\nabla\tilde\eta|\Bigr)\dd X\,.
\end{equation}
To estimate the right-hand side, we use \eqref{phibd1} and \cite[Lemma~2.1]{GW1} to obtain
the uniform bound
\[
  \Bigl\|\frac{|\nabla\tilde\phi|}{1{+}\epsilon R}\Bigr\|_{L^\infty} \,\le\,
  C \|\tilde \eta\|_{L^{4/3}}^{1/2}\|\tilde \eta\|_{L^4}^{1/2} \,\le\,
  C \|\tilde \eta\|_{\cX_\epsilon}^{1/2}\bigl(\|\tilde \eta\|_{\cX_\epsilon}^{1/2} +
  \|\nabla\tilde \eta\|_{\cX_\epsilon}^{1/2}\bigr)\,.
\]
On the other hand it is easy to verify that $|\nabla W_\epsilon| \le C(1+\rho_\gamma)W_\epsilon$
where $\rho_\gamma$ is defined in \eqref{rhogamdef}. It follows that
\begin{equation}\label{I3bd2}
  |I_3| \,\le\, C \|\tilde \eta\|_{\cX_\epsilon}^{3/2}\bigl(\|\tilde \eta\|_{\cX_\epsilon}^{1/2} +
  \|\nabla\tilde \eta\|_{\cX_\epsilon}^{1/2}\bigr)\bigl(\|\tilde \eta\|_{\cX_\epsilon} +
  \|\rho_\gamma\tilde \eta\|_{\cX_\epsilon} + \|\nabla\tilde \eta\|_{\cX_\epsilon}\bigr)
  \,\le\, C \|\tilde \eta\|_{\cX_\epsilon} D_\epsilon[\tilde\eta]\,,
\end{equation}
where for convenience we denote
\begin{equation}\label{Depsdef}
  D_\epsilon[\tilde\eta] \,=\, \|\nabla\tilde \eta\|_{\cX_\epsilon}^2 + \|\rho_\gamma\tilde
  \eta\|_{\cX_\epsilon}^2 + \|\tilde \eta\|_{\cX_\epsilon}^2\,.
\end{equation}

\medskip\noindent{\bf Control of $I_5$.} Proposition~\ref{Remprop} asserts that the 
remainder $\Rem(R,Z,t)$ satisfies the pointwise estimate \eqref{Remest}, which implies
in particular that $\Rem \in \cX_\epsilon$. In view of \eqref{I5def}, we thus find
\[
  |I_5| \,\le\, \frac{1}{\delta}\,\|\Rem\|_{\cX_\epsilon}\Bigl(\|\tilde \eta\|_{\cX_\epsilon}
  + \|W_\epsilon^{-1} \tilde\phi\|_{\cX_\epsilon}\Bigr) \,\le\, C\Bigl(\epsilon +
  \frac{\epsilon^{\gamma_5}}{\delta^2}\Bigr)\Bigl(\|\tilde \eta\|_{\cX_\epsilon} + \|W_\epsilon^{-1}
  \tilde\phi\|_{\cX_\epsilon}\Bigr)\,.
\]
It remains to estimate the norm of $W_\epsilon^{-1} \tilde\phi$ in the space $\cX_\epsilon$.
This can be done by decomposing the Biot-Savart kernel as in the proof of
Lemma~\ref{Ekinlem}, see in particular Eq.~\eqref{Gdecomp} below. Neglecting contributions
of order $\cO(\epsilon^\infty)$, we can restrict the integrals to the region where
$R^2 + Z^2 \le \epsilon^{-2\sigma_1}$ and ${R'}^2 + {Z'}^2 \le \epsilon^{-2\sigma_1}$.
Invoking \eqref{Gest} and recalling that $\mu_0(t) = \cO(\epsilon^\infty)$ by
Lemma~\ref{lem:mu0mu1}, we find that $\|W_\epsilon^{-1}\tilde\phi\|_{\cX_\epsilon} =
\|W_\epsilon^{-1/2}\tilde\phi\|_{L^2(\Omega_\epsilon)} \le C \|\tilde \eta\|_{\cX_\epsilon}$. 
We conclude that
\begin{equation}\label{I5est}
  |I_5| \,\le\, C\Bigl(\epsilon + \frac{\epsilon^{\gamma_5}}{\delta^2}\Bigr)\,\|\tilde
  \eta\|_{\cX_\epsilon}\,.
\end{equation}

\medskip\noindent{\bf Control of $\hat I_6$.} The first two terms in \eqref{hatI6def}
are easily estimated, because $\dot{\bar r} = \cO(\delta)$ by \eqref{rODE}.
Proceeding as in Lemma~\ref{Ekinlem} to control the kinetic energy, and recalling
that $\mu_0(t) = \cO(\epsilon^\infty)$, we find
\[
  \cI_0 \,:=\,
  \biggl|\frac{\epsilon \bar r \dot{\bar r}}{\delta\Gamma} \int_{\Omega_\epsilon} W_\epsilon
  \tilde \eta \partial_R \tilde \eta\dd X  + \frac{t\dot{\bar r}}{\bar r}\,E^\kin_\epsilon
  [\tilde\eta]\biggr| \,\le\, C\epsilon \|\tilde\eta\|_{\cX_\epsilon} \|\nabla\tilde\eta\|_{\cX_\epsilon} +
  C \epsilon^2 \|\tilde\eta\|_{\cX_\epsilon}^2\,.
\]
So it remains to estimate the last term in \eqref{hatI6def}, which involves the
correction $\dot{\tilde z}(t)$ to the vertical speed introduced in \eqref{tildezeq}.
Using \eqref{phidef} and integrating by parts we first observe that
\begin{equation}\label{cI1def}
\begin{split}
  \cI_1 \,:&=\, \int_{\Omega_\epsilon} \bigl(W_\epsilon \tilde \eta -\tilde\phi\bigr)\partial_Z\eta_*\dd X
  \,=\, \int_{\Omega_\epsilon}\bigl(W_\epsilon\partial_Z\eta_* - \partial_Z \phi_*\bigr)\tilde\eta \dd X \\  
  \,&=\, -\int_{\Omega_\epsilon'}(\partial_Z\Theta)\tilde\eta\dd X + \int_{\Omega_\epsilon''\cup
   \Omega_\epsilon'''}\bigl(W_\epsilon\partial_Z\eta_* - \partial_Z \phi_*\bigr)\tilde\eta \dd X\,,
\end{split}
\end{equation}
where $\Theta$ is defined in \eqref{Thetadef}. In the second line, we used the expression
\eqref{Wdef} of $W_\epsilon$ in the inner region $\Omega_\epsilon'$ to obtain the identity
$W_\epsilon\partial_Z\eta_* - \partial_Z \phi_* = \Phi_\epsilon'(\zeta_*)\partial_Z\zeta_* -
\partial_Z\phi_* = -\partial_Z\Theta$. The last integral in \eqref{cI1def} is of order
$\cO(\epsilon^\infty\|\tilde\eta\|_{\cX_\epsilon})$, and the integral over $\Omega_\epsilon'$ can
be controlled using Proposition~\ref{Phiprop}. We thus obtain $|\cI_1| \le C(\epsilon\delta +
\epsilon^{\gamma_3})\|\tilde\eta\|_{\cX_\epsilon}$. Moreover, we obviously have
\[
  \cI_2 \,:=\, \biggl|\int_{\Omega_\epsilon} \bigl(W_\epsilon \tilde \eta -\tilde\phi\bigr)
  \partial_Z\tilde\eta\dd X\biggr| \,\le\, C\|\tilde\eta\|_{\cX_\epsilon} \|\nabla\tilde
  \eta\|_{\cX_\epsilon}\,.
\]

Finally, to control the velocity $\dot{\tilde z}(t)$, we need the following lemma: 

\begin{lem}\label{Jlem}
Let $J(t) = \int_{\Omega_\epsilon}Z \cR(R,Z,t)\dd X$ where $\cR$ is defined in \eqref{cRdef}.
Then there exists a constant $C > 0$ such that
\begin{equation}\label{Jest}
  |J| \,\le\, \frac{C\epsilon\beta_\epsilon}{\delta}\,\Bigl(\|\tilde\eta\|_{\cX_\epsilon} +
  \delta \|\tilde\eta\|_{\cX_\epsilon}^2\Bigr) + C\Bigl(\epsilon + \frac{\epsilon^{\gamma_5}}{\delta^2}
  \Bigr)\,.
\end{equation}
\end{lem}

\begin{proof}
We consider separately the various terms in the right-hand side of \eqref{cRdef}.
Integrating by parts, we find
\[
  J_1 \,:=\, \frac{1}{\delta}\int_{\Omega_\epsilon}Z \Bigl(\bigl\{\phi_*\,,\tilde\zeta\bigr\}
  + \bigl\{\tilde\phi\,,\zeta_*\bigr\}\Bigr) \dd X \,=\, -\frac{1}{\delta}\int_{\Omega_\epsilon}
  \biggl(\frac{\tilde\eta \partial_R\phi_*}{1{+}\epsilon R} +
  \frac{\eta_* \partial_R\tilde\phi}{1{+}\epsilon R}\biggr)\dd X\,.  
\]
In the right-hand side, we can restrict the integration to the region where $\rho \le
\epsilon^{-\sigma_1}$, because the integral on the complement is of order
$\cO\bigl(\epsilon^\infty \|\tilde\eta\|_{\cX_\epsilon}\bigr)$. Thus, expanding the
Biot-Savart formula as in Section~\ref{ssec31}, we obtain
\begin{equation}\label{Jest1}
  -\delta J_1 \,=\,\frac{1}{2\pi}\int_{\R^2} \Bigl(\tilde\eta \partial_R(L\eta_*) +
  \eta_* \partial_R (L\tilde\eta)\Bigr)\dd X + \cO\bigl(\epsilon \beta_\epsilon
  \|\tilde\eta\|_{\cX_\epsilon}\bigr)\,,
\end{equation}
where $L$ is the convolution operator \eqref{Lopdef}. Since $L$ is symmetric
in $L^2(\R^2)$ and commutes with $\partial_R$, the integral in \eqref{Jest1}
vanishes and we conclude that $|J_1| \le \delta^{-1}\epsilon\beta_\epsilon\|\tilde
\eta\|_{\cX_\epsilon}$. 

Similarly, we have
\[
  J_2 \,:=\, \int_{\Omega_\epsilon}Z \bigl\{\tilde\phi\,,\tilde\zeta\bigr\}\dd X \,=\,
  \int_{\Omega_\epsilon}\bigl\{Z\,,\tilde\phi\bigr\}\tilde\zeta\dd X \,=\,
  -\int_{\Omega_\epsilon}\frac{\tilde\eta \partial_R\tilde\phi}{1{+}\epsilon R}\dd X\,.  
\]
Here again, up to a negligible error, we can assume that $\tilde \eta$ is supported in the
ball $\rho \le \epsilon^{-\sigma_1}$. Proceeding as before, we thus find
\begin{equation}\label{Jest2}
  J_2 \,=\, -\frac{1}{2\pi}\int_{\R^2}\tilde\eta \partial_R(L\tilde\eta)\dd X + 
  \cO\bigl(\epsilon \beta_\epsilon\|\tilde\eta\|_{\cX_\epsilon}^2\bigr) \,=\,
  \cO\bigl(\epsilon \beta_\epsilon\|\tilde\eta\|_{\cX_\epsilon}^2\bigr)\,.
\end{equation}

The remaining terms in \eqref{cRdef} are easier to treat. In view of \eqref{mu2def}
we have
\[
  \int_{\Omega_\epsilon} Z\Bigl(\cL \tilde\eta + \epsilon\partial_R\tilde\zeta
  \Bigr)\dd x \,=\, 0\,, \qquad \hbox{and} \quad
  \int_{\Omega_\epsilon} Z\Bigl(\dot{\bar r}\,\partial_R \tilde\eta + 
  \dot{\bar z}_* \,\partial_Z \tilde\eta\Bigr)\dd X \,=\, -\dot{\bar z}_* \mu_0\,,
\]
where $\mu_0(t) = \cO(\epsilon^\infty)$ by Lemma~\ref{lem:mu0mu1}. Finally, using estimate
\eqref{Remest}, we obtain
\begin{equation}\label{Jest3}
  \frac{1}{\delta}\int_{\Omega_\epsilon}|Z| \,|\Rem(R,Z,t)|\dd X  \,\le\,
  C\Bigl(\epsilon + \frac{\epsilon^{\gamma_5}}{\delta^2}\Bigr)\,.  
\end{equation}
Combining \eqref{Jest1}, \eqref{Jest2}, and \eqref{Jest3}, we arrive at \eqref{Jest}. 
\end{proof}

\begin{cor}\label{zcor}
There exists a constant $C > 0$ such that the velocity $\dot{\tilde z}$ defined by
\eqref{tildezeq} satisfies
\begin{equation}\label{tildezest}
  \frac{\bar r |\dot{\tilde z}|}{\Gamma} \,\le\, C\beta_\epsilon\Bigl(\|\tilde\eta\|_{\cX_\epsilon} +
  \delta\|\tilde\eta\|_{\cX_\epsilon}^2\Bigr) + C\Bigl(\delta +
   \frac{\epsilon^{\gamma_5-1}}{\delta}\Bigr)\,.
\end{equation}
\end{cor}

We now conclude the estimate of the term $\hat I_6$. To simplify the writing, we assume that
$\|\tilde\eta\|_{\cX_\epsilon} \le 1$ and we use the shorthand notation \eqref{Depsdef}.
Also, since $\epsilon^2 \lesssim \delta^{1-\sigma}$ we observe that
\begin{equation}\label{Restdef}
  \epsilon + \frac{\epsilon^{\gamma_5}}{\delta^2} \,\lesssim\, \Rest_\epsilon(t)\,, \qquad
  \hbox{where} \quad \Rest_\epsilon(t) \,:=\, \epsilon + \frac{\epsilon^{\gamma_3}}{\delta}\,.
\end{equation}
Here $\gamma_3 = \gamma_5 - 2/(1{-}\sigma) < 3$, so that $\gamma_3$ can be chosen arbitrary
close to $\gamma_5 - 2$ if $\sigma > 0$ is small enough. In view of \eqref{tildezeq}
and \eqref{hatI6def} we have $|\hat I_6| \le \cI_0 + |J| \bigl(|\cI_1| + \delta \cI_2\bigr)$,
so that
\begin{equation}\label{hatI6est}
\begin{split}
  |\hat I_6| \,&\le\, C\epsilon\|\tilde\eta\|_{\cX_\epsilon} D_\epsilon^{1/2} + 
  C\Bigl(\frac{\epsilon\beta_\epsilon}{\delta}\,\|\tilde\eta\|_{\cX_\epsilon} + \Rest_\epsilon\Bigr)
  \Bigl(\delta \Rest_\epsilon \|\tilde\eta\|_{\cX_\epsilon} + \delta \|\tilde\eta\|_{\cX_\epsilon}
  D_\epsilon^{1/2}\Bigr)\\
  \,&\le\, C\|\tilde\eta\|_{\cX_\epsilon}\bigl(D_\epsilon^{1/2} + \Rest_\epsilon\bigr)
  \bigl(\epsilon\beta_\epsilon + \delta \Rest_\epsilon\bigr) \,\le\, C\epsilon\beta_\epsilon
  \|\tilde\eta\|_{\cX_\epsilon}\bigl(D_\epsilon^{1/2} + \Rest_\epsilon\bigr)\,.
\end{split}
\end{equation}

\subsection{Conclusion of the proof}\label{ssec49}

We are now in position to conclude the proof of Theorem~\ref{main2}, hence also of
Theorem~\ref{main1}. Let $\tilde\eta$ be the unique solution of \eqref{tildeq} with zero
initial data. The associated energy \eqref{Edef} satisfies the evolution equation
\begin{equation}\label{Edot}
  t\partial_t E_\epsilon(t) \,=\,  I_1 + I_2 + I_3 + \hat I_4 + I_5 + \hat I_6\,,
\end{equation}
where the various terms in the right-hand side are defined in Section~\ref{ssec44}
and estimated in Sections~\ref{ssec46}--\ref{ssec48}. Using \eqref{I1est}, \eqref{I2est},
\eqref{I3bd2}, \eqref{I4est}, \eqref{I5est}, and \eqref{hatI6est}, we find that, as
long as $t \le T_\adv\delta^{-\sigma}$ and $\|\tilde\eta\|_{\cX_\epsilon} \le 1$, there
exist positive constants $C,C_*,\kappa$ such that
\[
  t\partial_t E_\epsilon(t) \,\le\, -\kappa D_\epsilon + C_*\|\tilde\eta\|_{\cX_\epsilon}
  D_\epsilon + C \|\tilde\eta\|_{\cX_\epsilon}\bigl(\Rest_\epsilon + \epsilon\beta_\epsilon
  D_\epsilon^{1/2}\bigr) + \frac{C\epsilon^2}{\delta}\int_{\Omega_\epsilon''} W_\epsilon
  \tilde\eta^2\dd X + C \mu^2\,,
\]
where $D_\epsilon$ is defined in \eqref{Depsdef}, $\Rest_\epsilon$ in \eqref{Restdef}, and
$\mu^2 := \mu_0^2 + \mu_1^2 + \mu_2^2 \le C\,\Rest_\epsilon^2$ by Lemma~\ref{lem:mu0mu1}.
Since $\rho_\gamma \ge \epsilon^{-\sigma_1}$ in the region $\Omega_\epsilon''$, the integral
term can be estimated as follows
\[
  \frac{\epsilon^2}{\delta }\int_{\Omega_\epsilon''}W_\epsilon \tilde\eta^2\dd X \,\le\,
  \frac{\epsilon^{2+2\sigma_1}}{\delta}\int_{\Omega_\epsilon''}W_\epsilon \rho_\gamma^2 \tilde\eta^2
  \dd X \,\lesssim\, \epsilon^{\gamma_*} D_\epsilon\,,
\]
where $\gamma_* = 2 + 2\sigma_1 - 2/(1{-}\sigma) > 0$ if $\sigma > 0$ is small enough. 
So, if we assume that $C_*\|\tilde\eta\|_{\cX_\epsilon} \le \kappa/4$ and that $\epsilon$
is sufficiently small, we obtain by Young's inequality
\[
  t\partial_t E_\epsilon(t) \,\le\, -\frac{\kappa}{2}\,D_\epsilon + C\,\Rest_\epsilon
  \|\tilde\eta\|_{\cX_\epsilon} + C\mu^2 \,\le\, -\frac{\kappa}{4}\,D_\epsilon + 
  C\,\Rest_\epsilon^2\,.
\]
Integrating that differential inequality over the time interval $(0,t)$ and recalling
that $E_\epsilon(0) = 0$, we arrive at
\[
  E_\epsilon(t) + \frac{\kappa}{4} \int_0^t \frac{D_\epsilon(s)}{s}\dd s
  \,\le\, C \int_0^t \frac{\Rest_\epsilon(s)^2}{s}\dd s \,\le\, C\,
  \Rest_\epsilon(t)^2\,.
\]
Finally, in view of \eqref{Eepsineq}, \eqref{mu2def}, and Lemma~\ref{lem:mu0mu1},
we infer that
\begin{equation}\label{etabound}
  \|\tilde\eta(t)\|_{\cX_\epsilon}^2 \,\le\, C_6 E_\epsilon(t) + C_7
  \bigl(\beta_\epsilon\mu_0(t)^2 + \mu_1(t)^2\bigr) \,\le\,
  C\,\Rest_\epsilon(t)^2\,.
\end{equation}
Inequality \eqref{etabound} holds as long as $\|\tilde\eta(t) \|_{\cX_\epsilon} \le
\min\bigl(1,\kappa/(4C_*)\bigr)$ and $t < T_\adv\delta^{-\sigma}$. But on that time interval we
know that $\Rest_\epsilon \lesssim \epsilon^{\gamma_3 - 2/(1-\sigma)} \ll 1$, so \eqref{etabound}
is actually valid for all $t \in (0,T_\adv\delta^{-\sigma})$ if $\epsilon > 0$ is
small enough. Returning to the solution of \eqref{etaeq} with initial data \eqref{etazero},
we obtain in view of \eqref{etapert}, \eqref{etabound}
\[
  \|\eta(t) - \eta_*(t)\|_{\cX_\epsilon} \,=\, \delta \|\tilde\eta(t)\|_{\cX_\epsilon} \,\le\,
  C \delta\,\Rest_\epsilon(t) \,=\, C\bigl(\epsilon \delta + \epsilon^{\gamma_3}\bigr)\,,
  \qquad t \in (0,T_\adv\delta^{-\sigma})\,,
\]
which gives \eqref{main2est}. This concludes the proof of Theorem~\ref{main2}. \QED

\begin{rem}\label{velrem}
The correction $\tilde z(t)$ to the vertical position of the vortex is small, and
produces negligible effects in our calculations. Indeed, it follows from \eqref{tildezest}
and \eqref{etabound} that
\begin{equation}\label{tzbdd2}
  \frac{\bar r |\dot{\tilde z}(t)|}{\Gamma} \,\lesssim\, \bigl(\beta_\epsilon
  \,\Rest_\epsilon + \delta\bigr)\,, \qquad \hbox{hence} \qquad
  \delta |\tilde z(t)|  \,\lesssim\, \epsilon^2 \bar r(t)\bigl(\delta  + \beta_\epsilon
  \,\Rest_\epsilon\bigr)\,.
\end{equation}
This gives in particular \eqref{tzbdd}.
\end{rem}

\begin{proof}[\bf Proof of Theorem~\ref{main1}.]
Let $\omega_\lin(r,z,t)$ be the solution of the (axisymmetric) heat equation in
$\Omega$ with initial data $\Gamma\,\delta_{(r_0,z_0)}$. Using the same self-similar
variables as in the proof of Theorem~\ref{main2}, we define the rescaled vorticity
$\eta_\lin$ by the relation
\begin{equation}\label{omlin}
  \omega_\lin\bigl(r,z-a_3(t),t\bigr) \,=\, \frac{\Gamma}{\nu t}\,\eta_\lin
  \Bigl(\frac{r-\bar r(t)}{\sqrt{\nu t}}\,,\,\frac{z-\bar z_*(t)-\delta\tilde z(t)}{\sqrt{\nu t}}
  \,,\,t\Bigr)\,,
\end{equation}
where $a_3(t) = \int_0^t V(s)\dd s$ and $V$ is given by \eqref{i4}. A direct
calculation then shows that $\eta_\lin$ satisfies the linear equation
\begin{equation}\label{etalineq}
  t\partial_t \eta_\lin - \frac{\epsilon \bar r}{\delta\Gamma}\Bigl(\dot{\bar r}
  \,\partial_R \eta_\lin + \dot s\,\partial_Z \eta_\lin\Bigr) \,=\, \cL \eta_\lin
  + \partial_R \Bigl(\frac{\epsilon\eta_\lin}{1+\epsilon R}\Bigr)\,,
\end{equation}
with initial data $\eta_0$, where the shift $s(t) = \bar z_*(t) -a_3(t) +\delta\tilde z(t)$ measures
the difference between the vertical position of the vortex as computed in Theorem~\ref{main2}
and the approximation given by the Kelvin-Saffman formula \eqref{i4} without correction terms.
Since $\dot a_3 = \dot{\bar z}_0$, it follows from \eqref{rzapp2} that $\dot s =
\epsilon^2 \dot{\bar z}_2 + \delta \dot{\tilde z}$. Using \eqref{dotr2z2} and \eqref{tzbdd2}, 
we thus obtain
\begin{equation}\label{shift}
  \frac{\epsilon\bar r |\dot s(t)|}{\delta\Gamma} \,\le\, C\Bigl(
  \frac{\beta_\epsilon\epsilon^3}{\delta} + \beta_\epsilon \epsilon^2 + \epsilon\delta\Bigr)
  \,\le\, C \epsilon^{1-3\sigma}\,,
\end{equation}
because $\epsilon^2 \lesssim \delta^{1-\sigma}$ so that $\beta_\epsilon \epsilon^3 \delta^{-1}
\le \epsilon^{1-3\sigma}$ if $0 < \sigma < 1/3$ and $\epsilon > 0$ is small enough.

The solution of \eqref{etalineq} with initial data $\eta_0$ can be estimated as in
\cite[Section~4.4]{GS2}, with substantial simplifications. We use the approximate
solution $\hat\eta_0(R,Z,t) := \chi_0(4\epsilon\rho)\eta_0(R,Z)$, where $\chi_0$
is the cut-off function in \eqref{etaapp1}. Decomposing $\eta_\lin = \hat\eta_0
+ \hat \eta$, we see that the correction $\hat\eta$ satisfies
\begin{equation}\label{etahateq}
  t\partial_t \hat\eta - \frac{\epsilon \bar r}{\delta\Gamma}\Bigl(\dot{\bar r}
  \,\partial_R \hat\eta + \dot s\,\partial_Z \hat\eta\Bigr) \,=\, \cL \hat\eta
  + \partial_R \Bigl(\frac{\epsilon\hat\eta}{1+\epsilon R}\Bigr) + \cR_0\,,
\end{equation}
where
\[
  \cR_0 \,=\, \cL \hat\eta_0 + \partial_R \Bigl(\frac{\epsilon\hat\eta_0}{1+\epsilon R}
  \Bigr) + \frac{\epsilon \bar r}{\delta\Gamma}\Bigl(\dot{\bar r}
  \,\partial_R \hat\eta_0 + \dot s\,\partial_Z \hat\eta_0\Bigr) - t\partial_t \hat\eta_0\,.
\]
To control the solution of \eqref{etahateq}, we introduce the space $\hat\cX_\epsilon$
defined by the norm
\[
  \|\hat\eta\|_{\hat\cX_\epsilon}^2 \,=\, \int_{\Omega_\epsilon} e^{(R^2+Z^2)/4}
  \,\hat\eta(R,Z)^2\dd R\dd Z\,.
\]
In view of \eqref{shift} we have $\|\cR_0\|_{\hat\cX_\epsilon} \le C \epsilon^{1-3\sigma}$,
and using energy estimates as in \cite{GS2} we deduce that the solution of \eqref{etahateq}
with zero initial data satisfies $\|\hat\eta\|_{\hat\cX_\epsilon} \le C \epsilon^{1-3\sigma}$
for $t \in (0,T_\adv\delta^{-\sigma})$. Since $\hat\cX_\epsilon \hookrightarrow \cX_\epsilon$
by \eqref{Wunifbd}, \eqref{cXdef}, we conclude that $\|\eta_\lin - \eta_0\|_{\cX_\epsilon} =
\cO\bigl(\epsilon^{1-3\sigma})$ as $\epsilon \to 0$.

Now the solution of \eqref{omeq} with initial data $\Gamma\,\delta_{(r_0,z_0)}$ satisfies,
instead of \eqref{etadef}, 
\begin{equation}\label{omnonlin}
  \omega_\theta(r,z,t) \,=\, \frac{\Gamma}{\nu t}\,\eta
  \Bigl(\frac{r-\bar r(t)}{\sqrt{\nu t}}\,,\,\frac{z-\bar z_*(t)-\delta\tilde z(t)}{\sqrt{\nu t}}
  \,,\,t\Bigr)\,,
\end{equation}
so combining \eqref{omlin}, \eqref{omnonlin} we obtain
\begin{equation}\label{nseapp2}
\begin{split}
  \frac{1}{\Gamma}\int_{\Omega} \Bigl|\omega_\theta\bigl(r,z,t\bigr) -
  \omega_\lin\bigl(r,z-a_3(t),t\bigr)\Bigr|\dd r\dd z \,&=\,
  \|\eta(t) - \eta_\lin(t)\|_{L^1(\Omega_\epsilon)} \\
  \,&\le\, C \|\eta(t) - \eta_\lin(t)\|_{\cX_\epsilon} \,\le\, C \epsilon^{1-3\sigma}\,,
\end{split}
\end{equation}
because $\|\eta(t) - \eta_0\|_{\cX_\epsilon} \le C\epsilon$ and $\|\eta_0 - \eta_\lin
\|_{\cX_\epsilon} \le C\epsilon^{1-3\sigma}$. Using the notations of \eqref{nseapp},
inequality \eqref{nseapp2} exactly means that $\|\omega_\corr(\cdot\,,t)\| \le
C\Gamma \epsilon^{1-3\sigma}$. This concludes the proof of Theorem~\ref{main1}.
\end{proof}

\appendix
\section{Appendix to Section~\ref{sec3}}\label{secA}

\subsection{Inverting the operator $\Lambda$}\label{ssecA1}

Following \cite{Ga}, we give here a short proof of Proposition~\ref{Laminv}. 
Assume that $n \ge 2$ and $f \in \cY_n \cap \cZ$, or that $n = 1$ and $f \in \cY'_1 
\cap \cZ$. In both cases, we have $f \in \Ker(\Lambda)^\perp$. We want to show that 
there exists a unique $\eta \in \cY_n \cap \cZ$ (respectively, $\eta \in \cY_1'
\cap \cZ$ if $n = 1$) such that $\Lambda \eta = f$. 

To make things concrete, we suppose without loss of generality that $f =
a(\rho)\sin(n\vt)$, for some function $a : \R_+ \to \R$. Our hypotheses
imply that $a$ is smooth, that $a(\rho) = \cO(\rho^n)$ as $\rho \to 0$, and
that $e^{\rho^2/4}a(\rho)$ grows at most polynomially as $\rho \to \infty$. We
look for a solution of the form $\eta = \omega(\rho) \cos(n\vt)$, where
$\omega : \R_+ \to \R$ has to be determined. By \eqref{Lamdef}, we have
\begin{equation}\label{Lameta}
  \Lambda \eta \,=\, \bigl\{\phi_0\,,\eta\bigr\} + \{\Psi\,,\eta_0\bigr\}\,,
  \qquad \hbox{where}\quad \phi_0 \,=\, \frac{1}{2\pi}\,L\eta_0\,, \quad
  \Psi \,=\, \frac{1}{2\pi}\,L\eta\,.
\end{equation}
The function $\phi_0$ is radially symmetric and satisfies $\partial_\rho \phi_0 = 
-\rho\vf(\rho)$, see \eqref{phihdef} and \eqref{psi0rep} below. It follows 
that
\begin{equation}\label{aux1}
  \bigl\{\phi_0\,,\eta\bigr\} \,=\, \partial_\rho \phi_0\, \frac{1}{\rho}
  \,\partial_\vt \eta \,=\, n \vf(\rho)\omega(\rho)\sin(n\vt)\,.
\end{equation}
On the other hand, as $-\Delta \Psi = \eta$, we have $\Psi = \Omega(\rho)
\cos(n\vt)$, where $\Omega$ is the unique regular solution of the differential 
equation
\begin{equation}\label{Omeq}
  -\Omega''(\rho) -\frac{1}{\rho}\,\Omega'(\rho) + \frac{n^2}{\rho^2}\,
  \Omega(\rho) \,=\, \omega(\rho)~, \qquad \rho > 0~.  
\end{equation}
Since $\eta_0$ is radially symmetric and $\partial_\rho \eta_0 = -(\rho/2)\eta_0
= -\rho \vf(\rho) h(\rho)$, see \eqref{phihdef}, we deduce 
\begin{equation}\label{aux2}
  \{\Psi\,,\eta_0\bigr\} \,=\, -\partial_\rho \eta_0\, \frac{1}{\rho}
  \,\partial_\vt \Psi \,=\, -n\vf(\rho)h(\rho)\Omega(\rho)\sin(n\vt)\,.
\end{equation}
In view of \eqref{Lameta}, \eqref{aux1}, \eqref{aux2}, the equation 
$\Lambda \eta = f$ is equivalent to the relation \eqref{omexp}, 
and using in addition \eqref{Omeq} we obtain the differential 
equation \eqref{Omexp} for the stream function $\Omega$. 

The main step in the proof is to show that \eqref{Omexp} has a unique
solution that is regular at the origin and decays to zero at
infinity. Here we distinguish two cases according to the value of
the angular Fourier mode $n$.

\smallskip\noindent 
{\bf 1.} If $n \ge 2$, the homogeneous equation \eqref{Omexp} with 
$a \equiv 0$ has two linearly independent solutions $\psi_+$, $\psi_-$ 
which satisfy
\begin{equation}\label{psipm}
  \psi_-(\rho) \,\sim\, \begin{cases} \rho^n \hspace{6mm}
  \hbox{as }\rho \to 0\,, \\
  \kappa \rho^n \quad \hbox{as }\rho \to \infty\,,\end{cases}
  \qquad
  \psi_+(\rho) \,\sim\, \begin{cases} \kappa \rho^{-n} \quad
  \hbox{as }\rho \to 0\,, \\
  \rho^{-n} \hspace{6mm} \hbox{as }\rho \to \infty\,,\end{cases}
\end{equation}
for some $\kappa > 0$, see \cite{Ga}. Here we use the crucial observation 
that $(n^2/\rho^2) - h(\rho) > 0$ when $n \ge 2$, so that the differential
operator in the left-hand side of \eqref{Omexp} satisfies the Maximum Principle. 
We deduce the following representation formula for the solution of the 
inhomogeneous equation\:
\begin{equation}\label{Omsol}
  \Omega(\rho) \,=\, \psi_+(\rho)\int_0^\rho \frac{r}{w_0}\,\psi_-(r) 
  \frac{a(r)}{n\vf(r)}\dd r + \psi_-(\rho)\int_\rho^\infty \frac{r}{w_0}
  \,\psi_+(r)\frac{a(r)}{n\vf(r)}\dd r\,,
\end{equation}
where $w_0 = 2n\kappa$. It is then straightforward to verify that 
$\Omega(\rho) = \cO(\rho^n)$ as $\rho \to 0$ and $\Omega(\rho) = 
\cO(\rho^{-n})$ as $\rho \to \infty$. Moreover, if $\omega$ is 
defined by \eqref{omexp}, the function $\eta = \omega(\rho)\cos(n\vt)$ 
lies in $\cY_n \cap \cZ$ and satisfies $\Lambda \eta = f$ by construction. 
The details can be found in \cite[Lemma 4]{Ga}. 

\smallskip\noindent 
{\bf 2.} The situation is quite different when $n = 1$, because the lower order 
term $1/\rho^2 - h(\rho)$ in \eqref{Omexp} is no longer positive. In that case, 
it happens that the homogeneous equation \eqref{Omexp} with $a \equiv 0$ has 
a solution $\psi(\rho) = \rho \vf(\rho)$ which satisfies $\psi(\rho) \sim 
\rho/(8\pi)$ as $\rho \to 0$ and $\psi(\rho) \sim 1/(2\pi\rho)$ as $\rho \to 
\infty$. In other words, the linear operator in the left-hand side of 
\eqref{Omexp} has a one-dimensional kernel, and for that reason 
we have to impose the solvability condition 
\begin{equation}\label{solvcond}
  f \,\in\, \cY_1' \,\subset\, \Ker(\Lambda)^\perp\,, \quad \hbox{or equivalently}
  \quad \int_0^\infty a(\rho)\rho^2\dd \rho \,=\, 0\,. 
\end{equation}
To solve \eqref{Omexp} for $n = 1$, we look for a solution of the form $\Omega(\rho) 
= b(\rho)\psi(\rho)$, which leads to a first-order differential equation for $b(\rho)$.
In view of \eqref{solvcond}, we thus find
\begin{equation}\label{bfirst}
  b'(\rho) \,=\, -\frac{1}{\rho\psi(\rho)^2}\int_0^\rho a(r)r^2\dd r \,=\, 
  \frac{1}{\rho\psi(\rho)^2}\int_\rho^\infty a(r)r^2\dd r\,. 
\end{equation}
Integrating \eqref{bfirst} gives the representation formula
\[
  b(\rho) \,=\, b_0  - \int_0^\rho a(r) r^2 \Bigl(\cF(\rho) - \cF(r)\Bigr)\dd r\,,
  \qquad \hbox{for some } b_0 \in \R\,,
\]
where
\[
  \cF(\rho) \,=\, 8\pi^2 \left(\log\bigl(e^{\rho^2/4}-1\bigr) - \frac{1}{
  e^{\rho^2/4}-1}\right)\,, \qquad \cF'(\rho) \,=\, \frac{1}{\rho\psi(\rho)^2}\,.
\]
We now substitute $\Omega(\rho) = b(\rho)\psi(\rho)$ into \eqref{omexp} 
with $n = 1$, and we choose the constant $b_0$ so that $\int_0^\infty \omega(\rho) 
\rho^2 \dd\rho = 0$. This is always possible in a unique way, since
\[
  \int_0^\infty h(\rho) \psi(\rho) \rho^2 \dd\rho \,=\, 
  \int_0^\infty h(\rho) \vf(\rho) \rho^3 \dd\rho \,=\,
  \frac{1}{8\pi}\int_0^\infty e^{-\rho^2/4}\rho^3\dd \rho \,=\, \frac{1}{\pi}
  \,\neq\, 0\,.
\]
To conclude the proof, it remains to verify that the function $\eta = \omega(\rho)
\cos(\vt)$ constructed above belongs to $\cY_1' \cap \cZ$ and satisfies 
$\Lambda \eta = f$. These are straightforward calculations, which can be 
omitted. \QED

\subsection{First order calculations}\label{ssecA2}

We first establish the relations \eqref{LP1}. As $\eta_0 \in \cY_0$ has unit mass we 
find, using \eqref{PQexp},
\begin{equation}\label{P1eta0}
  \bigl(P_1 \eta_0\bigr)(R,Z) \,=\, \int_{\R^2}\frac{R{+}R'}{2}\,\eta_0(R',Z')
  \dd R' \dd Z' \,=\, \frac{R}{2}\,,  
\end{equation}
hence $\{P_1 \eta_0\,,\eta_0\} = \frac12\, \partial_Z \eta_0$. On the other hand, 
since $\partial_R \eta_0 = -(R/2)\eta_0$ and $L$ is a convolution operator, which 
therefore commutes with derivatives, we have 
\[
  \bigl(L P_1 \eta_0\bigr)(R,Z) \,=\, \frac{R}{2}\,(L\eta_0)(R,Z) +   
  L\Bigl(\frac{R}{2}\eta_0\Bigr)(R,Z) \,=\, \frac{R}{2}\,(L\eta_0)(R,Z) -  
  \partial_R\bigl(L\eta_0\bigr)(R,Z)\,. 
\]
Recalling that $L\eta_0 = 2\pi \phi_0$, and that $\{\phi_0,\eta_0\} = 0$ 
because both $\phi_0$, $\eta_0$ are radially symmetric, we thus obtain
\begin{align*}
  \frac{1}{2\pi}\bigl\{L P_1 \eta_0\,,\eta_0\bigr\} \,&=\, 
  \Bigl\{\frac{R}{2}\,\phi_0-  \partial_R\phi_0\,,\eta_0\Bigr\}
  \,=\, \frac{1}{2}\,\phi_0\,\partial_Z \eta_0 + \bigl\{\phi_0\,,
  \partial_R \eta_0 \bigr\} \\
  \,&=\,  \frac{1}{2}\,\phi_0\,\partial_Z \eta_0 - \Bigl\{\phi_0\,,
  \frac{R}{2}\,\eta_0 \Bigr\} \,=\, \frac{1}{2}\,\phi_0\,\partial_Z \eta_0 + 
  \frac{1}{2}\,(\partial_Z\phi_0)\eta_0\,,
\end{align*}
which concludes the proof of \eqref{LP1}.

We next prove formula \eqref{dotz0} for the vertical velocity. Assuming that
$\dot{\bar z}_0$  is given by \eqref{dotz0} for some $v \in \R$, we see that
the right-hand side of \eqref{eta1eq} belongs to $\cY_1' = \cY \cap \Ker(\Lambda)^\perp$
if and only if
\begin{equation}\label{vdef1}
  \int_{\R^2} \Bigl(\frac{v}{2\pi}\,\partial_Z \eta_0 -\frac{3}{2}
  (\partial_Z \phi_0) \eta_0 -\frac{1}{2}\phi_0 \partial_Z \eta_0\Bigr)
  Z\, \dd R \dd Z \,=\, 0\,.
\end{equation}
Since  $\partial_Z \eta_0 = -(Z/2)\eta_0$ and $\int_{\R^2} Z^2 \eta_0 \dd R\dd Z 
= 2$, it is straightforward to verify that \eqref{vdef1} is equivalent to 
\begin{equation}\label{vdef2}
  v \,=\, \pi\int_{\R^2} \phi_0 \eta_0 \bigl(3 - Z^2\bigr)\dd R \dd Z
  \,=\, \frac{\pi}{2}\int_{\R^2} \phi_0\eta_0 \bigl(6 - |X|^2\bigr)\dd X\,,
\end{equation}
where $X = (R,Z)$ and $|X|^2 = R^2 + Z^2$. 

To evaluate the right-hand side of \eqref{vdef2}, we temporarily denote 
$\psi_0 = 2\pi \phi_0 = L\eta_0$, namely
\[
  \psi_0(X) \,=\, \frac{1}{4\pi} \int_{\R^2} \log\Bigl(\frac{8}{|X-Y|}\Bigr)
  \,e^{-|Y|^2/4}\dd Y\,, \qquad X \in \R^2\,.
\]
This function satisfies $-\Delta \psi_0 = 2\pi\eta_0 = \frac12\,e^{-|X|^2/4}$, so
that
\begin{equation}\label{psi0rep}
  \psi_0(X) \,=\, \psi_0(0) \,-\, \int_0^{|X|} \frac{1 - e^{-\rho^2/4}}{\rho}
  \dd \rho \,=:\, \tilde\psi_0(|X|)\,, \qquad X \in \R^2\,,
\end{equation}
where
\begin{equation}\label{psi00}
  \psi_0(0) \,=\, \log(8) - \frac{1}{4\pi} \int_{\R^2} \log(|Y|) 
  \,e^{-|Y|^2/4}\dd Y \,=\, 2\log(2) + \frac{\gamma_E}{2}\,.
\end{equation}
Using \eqref{psi0rep}, \eqref{psi00} and integrating by parts, we easily find
\[
  \int_{\R^2} \psi_0 \eta_0 \dd X \,=\, \frac{1}{2} \int_0^\infty
  \tilde \psi_0(\rho) e^{-\rho^2/4} \rho \dd \rho \,=\, 
  \psi_0(0) + \int_0^\infty \tilde \psi_0'(\rho) e^{-\rho^2/4}\dd \rho \,=\, 
  \frac{3}{2}\log(2) + \frac{\gamma_E}{2}\,,
\]
and similarly
\[
  \int_{\R^2} \psi_0 \eta_0 |X|^2 \dd X \,=\, 4\psi_0(0) + \int_0^\infty
  \tilde \psi_0'(\rho) e^{-\rho^2/4}(\rho^2 + 4)\dd \rho \,=\, 
  6\log(2) + 2\gamma_E -1\,.
\]
Returning to \eqref{vdef2}, we conclude that
\begin{equation}\label{vdef3}
  v \,=\, \frac{1}{4}\int_{\R^2} \psi_0\eta_0 \bigl(6 - |X|^2\bigr)\dd X
  \,=\, \frac{3}{4}\log(2) + \frac{1}{4}\gamma_E + \frac{1}{4}\,. 
\end{equation}

\subsection{Second order calculations}\label{ssecA3}

Our goal here is to prove Lemma~\ref{R2lem}. To establish \eqref{R2exp}, we consider 
separately the various terms in \eqref{R2def}. As $\eta_1 \in \cY_1$ has zero mean, 
we find as in \eqref{P1eta0} that $P_1 \eta_1$ is a constant, 
which can be disregarded. Moreover $L P_1 \eta_1 = \frac{R}{2}L\eta_1 +    
L\bigl(\frac{R}{2}\eta_1\bigr)$, hence using the expression \eqref{eta1exp}
of $\eta_1$ we find that
\[
  L P_1 \eta_1 \,=\, (R^2-Z^2) \chi_1(\rho) + \delta RZ \chi_2(\rho) + \chi_3(\rho)\,,
\]
where $\chi_1, \chi_2, \dots$ are functions of the radial variable $\rho =
(R^2 + Z^2)^{1/2}$. As $\eta_0$ itself is radially symmetric, we deduce that
\begin{equation}\label{2nd1}
  \bigl\{(\beta_\epsilon - 1)P_1 \eta_1 + LP_1 \eta_1\,,\eta_0\bigr\} \,=\, 
   RZ \chi_4(\rho) + \delta (R^2-Z^2)\chi_5(\rho)\,.
\end{equation}

Next, using the expression \eqref{PQexp} of $P_2$, we see that 
\[
  \bigl(P_2 \eta_0\bigr)(R,Z) \,=\, \frac{1}{16}\int_{\R^2}\Bigl(
  (R{-}R')^2 + 3 (Z{-}Z')^2\Bigr)\,\eta_0(R',Z')\dd R' \dd Z' \,=\, 
  \frac{R^2}{16} + \frac{3 Z^2}{16} + \frac12\,,   
\]
and a similar calculation gives $Q_2 \eta_0 = \frac{3R^2}{16} - \frac{Z^2}{16}
+ \frac14$. Moreover,
\[
  \bigl(L P_2 \eta_0\bigr)(R,Z) \,=\, \frac{1}{16} \int_{\R^2}
  \log\Bigl(\frac{8}{D}\Bigr)\Bigl(2D^2 + (Z{-}Z')^2 - (R{-}R')^2\Bigr)
  \,\eta_0(R',Z')\dd R' \dd Z'\,,
\]
where $D^2 = (R{-}R')^2 + (Z{-}Z')^2$. Using the fact that $\eta_0$
given by \eqref{phi0def} is radially symmetric, we easily obtain
\[
  \frac{1}{2\pi}\bigl(L P_2 \eta_0\bigr)(R,Z) \,=\, 
  \chi_6(\rho) + (R^2 - Z^2)\chi_7(\rho)\,.  
\]
Altogether, we arrive at
\begin{equation}\label{2nd2}
  \frac{1}{2\pi}\bigl\{\beta_\epsilon P_2\eta_0 + LP_2 \eta_0 
  + Q_2 \eta_0\,, \eta_0\bigr\} \,=\, \frac{\beta_\epsilon}{16\pi}\,
  RZ \eta_0 + RZ \chi_8(\rho)\,.  
\end{equation}

The remaining terms in \eqref{R2def} are easier to treat. In view of 
\eqref{dotz0}, \eqref{eta1exp}, \eqref{phi1exp}, we have
\begin{equation}\label{2nd3}
\begin{split}
  \bigl\{\phi_1\,,\eta_1\bigr\} &- \frac{r_0 \dot{\bar z}_0}{\Gamma}
  \,\partial_Z \eta_1 \,=\, \Bigl\{\phi_1 - \frac{\beta_\epsilon -1}{4\pi}\,R
  \,,\eta_1\Bigr\} - \frac{v}{2\pi}\,\partial_Z \eta_1  \\
  \,&=\, \Bigl\{\frac{R}{2}\,\phi_0 - \partial_R \phi_0 + R\,\phi_{10}(\rho) + 
  \delta Z\,\phi_{11}(\rho)\,, R\,\eta_{10}(\rho) + \delta Z\,\eta_{11}(\rho)\Bigr\}
  - \frac{v}{2\pi}\,\partial_Z \eta_1 \\
  \,&=\, RZ\,\chi_9(\rho) + \delta\Bigl(\chi_{10}(\rho) + (R^2-Z^2)
  \chi_{11}(\rho)\Bigr) + \delta^2RZ\,\chi_{12}(\rho)\,.
\end{split}
\end{equation}
It is also easy to verify that the terms $(\partial_Z \phi_1)\eta_0 + 
(\partial_Z \phi_0)\eta_1 - 2R(\partial_Z\phi_0)\eta_0 + \delta\partial_R(R\eta_0)$ 
are exactly of the same form. Finally, using again \eqref{eta1exp}, \eqref{phi1exp},
we obtain
\begin{equation}\label{2nd4}
  R\Bigl(\bigl\{\phi_1\,,\eta_0\bigr\} + \bigl\{\phi_0\,,\eta_1\bigr\}\Bigr) 
  \,=\, R\Bigl(\frac{\beta_\epsilon-1}{4\pi}\,\partial_Z \eta_0 + 
  Z \chi_{13}(\rho) + \delta R \chi_{14}(\rho)\Bigr)\,.
\end{equation}
If we now combine \eqref{2nd1}, \eqref{2nd2}, \eqref{2nd3}, \eqref{2nd4}, 
we arrive at \eqref{R2exp}. \QED

\subsection{Higher order order calculations}\label{ssecA4}

The calculations carried out in Sections~\ref{ssec35} and \ref{ssec36} do not
require new ideas, but a more compact notation is often helpful. To prove
Lemma~\ref{R3lem} and similar statements, it is important to understand how the
decomposition \eqref{Yndef} of the function space $\cY$ behaves under the
Poisson bracket. If we use polar coordinates $R = \rho\cos\vt$, $Z = \rho\sin\vt$, we
recall that $\cY_n$ is the subspace of $\cY$ spanned by functions of the form
$a(\rho)\cos(n\vt)$ and $b(\rho) \sin(n\vt)$. Since
\[
  \bigl\{f\,,g\bigr\} \,=\, \partial_R f \partial_Z g - \partial_Z f \partial_R g
  \,=\, \frac{1}{\rho}\Bigl(\partial_\rho f \partial_\vt g - \partial_\vt f
  \partial_\rho g\Bigr)\,,
\]
we easily obtain the following result\:

\begin{lem}\label{Poissonlem}
If $a,b : \R_+ \to \R$ are smooth functions and $n,m \in \N$, then
\begin{align*}
  \bigl\{a(\rho)\cos(n\vt)\,,\,b(\rho)\cos(m\vt)\bigr\} \,&=\,
  c_{11}(\rho)\sin((n{-}m)\vt) + c_{12}(\rho)\sin((n{+}m)\vt)\,, \\
  \bigl\{a(\rho)\sin(n\vt)\,,\,b(\rho)\sin(m\vt)\bigr\} \,&=\,
  c_{21}(\rho)\sin((n{-}m)\vt) + c_{22}(\rho)\sin((n{+}m)\vt)\,, \\
  \bigl\{a(\rho)\sin(n\vt)\,,\,b(\rho)\cos(m\vt)\bigr\} \,&=\,
  c_{31}(\rho)\cos((n{-}m)\vt) + c_{32}(\rho)\cos((n{+}m)\vt)\,,
\end{align*}
where $c_{ij} : \R_+ \to \R$ are smooth functions. In particular
$\{\cY_n,\cY_m\} \subset \cY_{n-m} + \cY_{n+m}$ if $m \le n$. 
\end{lem}

It is also necessary to compute the homogeneous polynomials $P_j, Q_j$ in
\eqref{Kexp} for higher values of $j$ than in Lemma~\ref{Kexpansion}. This is a
cumbersome calculation that can be done for instance using computer algebra. For
$j = 3$ we find
\begin{equation}\label{P3Q3}
\begin{split}
  P_3 \,&=\, -\frac{1}{32}(R+R')\Bigl((R-R')^2 + 3 (Z-Z')^2\Bigr)\,,\\
  Q_3 \,&=\, -\frac{1}{48}(R+R')\Bigl((R+R')^2 - 6 (Z-Z')^2\Bigr)\,,
\end{split}
\end{equation}
and the calculation for $j = 4$ yields the more complicated
expressions
\begin{equation}\label{P4Q4}
\begin{split}
  P_4 \,=\, &-\frac{15}{1024}\,(Z{-}Z')^4 + \frac{21}{512}\,(R{-}R')^2 
  (Z{-}Z')^2 + \frac{3}{16}\,R R'\,(Z{-}Z')^2 \\ & + \frac{17}{1024}\,
  (R^2{-}R'^2)^2 - \frac{1}{256}\,R R'\,(R{-}R')^2\,, \\
  Q_4 \,=\, &\frac{31}{2048}\,(Z{-}Z')^4 - \frac{89}{1024}\,(R{+}R')^2(Z{-}Z')^2
  + \frac{1}{256}\,R R'\,(Z{-}Z')^2 \\ &- \frac{19}{6144}\,(R^2{-}R'^2)^2
   + \frac{35}{1536}\,R R'(R{+}R')^2 - \frac{1}{128}\,R^2 R'^2\,.
\end{split}
\end{equation}

The proof of Lemma~\ref{R3lem} is similar to that of Lemma~\ref{R2lem}, and the
details can be omitted. We use the expressions \eqref{eta1exp}, \eqref{eta2exp}
of the vorticities $\eta_1, \eta_2$, the formulas \eqref{phi1exp}, \eqref{phi2exp}
for the stream functions $\phi_1, \phi_2$, and the definition \eqref{BSmdef}
of the Biot-Savart operators, which involve the polynomials \eqref{PQexp} and
\eqref{P3Q3}. Using Lemma~\ref{Poissonlem}, it is straightforward to verify that
the quantity defined in \eqref{R3def} satisfies $\cR_3 \in \cY_1 + \cY_3$ and
takes the form
\[
  \cR_3 \,=\, \chi_1(\rho)\sin(\vt) + \chi_2(\rho)\sin(3\vt) +
  \delta\Bigl(\chi_3(\rho)\cos(\vt) + \chi_4(\rho)\cos(3\vt)\Bigr) + 
  \cO(\delta^2)\,,
\]
where $\chi_1,\chi_2,\chi_3,\chi_4$ are radially symmetric functions which may depend
linearly on $\beta_\epsilon$. To arrive at \eqref{R3exp}, it remains to verify that $\cR_3$
does not contain any term involving $\beta_\epsilon^2$. Indeed, according to
\eqref{PQexp}, \eqref{eta2exp}, we have
\[
  \frac{\beta_\epsilon}{2\pi}\,P_1 \eta_2 \,=\, \frac{\beta_\epsilon}{4\pi}
  \int_{\R^2}(R+R')\,\eta_2(R',Z')\dd R'\dd Z' \,=\, \frac{\beta_\epsilon R}{4\pi}
  \int_{\R^2}\,\eta_{24}(R',Z')\dd R'\dd Z'\,,
\]
so that the first term in \eqref{R3def} does not contain $\beta_\epsilon^2$. 
The only other terms that we have to check are
\[
  \bigl\{\phi_1\,,\eta_2\bigr\} - \frac{r_0}{\Gamma}\,\dot{\bar z}_0
  \partial_Z \eta_2 \,=\, \Bigl\{\phi_1 - \frac{\beta_\epsilon -1 + 2v}{4\pi}\,R
  \,,\eta_2\Bigr\}\,,
\]
but using the expressions \eqref{phi1exp}, \eqref{eta2exp} we immediately see that
the right-hand side does not contain any factor $\beta_\epsilon^2$. Altogether we
arrive at \eqref{R3exp}. \QED 

\section{Appendix to Section~\ref{sec4}}\label{secB}

\subsection{Properties of the energy functional}\label{ssecB1}

\begin{proof}[\bf Proof of Lemma~\ref{Ekinlem}] 
We use the first expression of $E_\epsilon^\kin[\eta]$ in \eqref{Ekindef} and the
representation formula \eqref{BSeps} for the stream function $\phi$. Since
$\supp(\eta) \subset B_\epsilon$ by assumption, we have
\begin{equation}\label{EkinG}
  E_\epsilon^\kin[\eta] \,=\, \frac{1}{4\pi} \int_{B_\epsilon}\int_{B_\epsilon}
  K_\epsilon(R,Z;R',Z')\,\eta(R,Z)\,\eta(R',Z')\dd X\dd X'\,,
\end{equation}
where the integral kernel $K_\epsilon$ is defined in \eqref{Kepsdef}. 
As $R^2+Z^2 \le \epsilon^{-2\sigma_1}$ and ${R'}^2+{Z'}^2 \le \epsilon^{-2\sigma_1}$,
the argument of $F$ in \eqref{Kepsdef} is not larger than $C \epsilon^{2-2\sigma_1}$
for some $C > 0$. Using the asymptotic expansion of $F(s)$ as $s \to 0$ and
proceeding as in Section~\ref{ssec31}, we easily obtain the decomposition
\begin{equation}\label{Gdecomp}
  K_\epsilon(R,Z;R',Z') \,=\, \beta_\epsilon - 2 + \log\frac{8}{D} + \tilde
  K_\epsilon(R,Z;R',Z')\,,
\end{equation}
where $\beta_\epsilon = \log(1/\epsilon)$ and $D^2 = (R{-}R')^2 + (Z{-}Z')^2$. The
remainder $\tilde K_\epsilon$ satisfies the estimate
\begin{equation}\label{Gest}
  |\tilde K_\epsilon(R,Z;R',Z')| \,\le\, C\epsilon\bigl(|R| + |R'|\bigr)
  \Bigl(\beta_\epsilon + 1 + \log\frac{8}{D}\Bigr) + \cO\bigl(\beta_\epsilon
  \epsilon^{2-2\sigma_1}\bigr)\,.
\end{equation}
If we insert the decomposition \eqref{Gdecomp} into \eqref{EkinG}, the contributions
of $\beta_\epsilon - 2$ and $\log(8/D)$ give exactly the first two terms in the
right-hand side of \eqref{Ekinexp}, in view of \eqref{Ekin0def}. Moreover,
taking into account estimate \eqref{Gest} where $\epsilon^{2-2\sigma_1} \le \epsilon$, 
we see that the contributions of $\tilde K_\epsilon$ to the kinetic energy
\eqref{EkinG} are of order $\cO\bigl(\epsilon\beta_\epsilon\|\eta\|_{\cX_\epsilon}^2
\bigr)$, as stated in \eqref{Ekinexp}. 
\end{proof}

\begin{proof}[\bf Proof of Proposition~\ref{Eepsprop}]
Given $\eta \in \cX_\epsilon$, we decompose $\eta = \eta_1 + \eta_2$ where
$\eta_1 = \eta \1_{B_\epsilon}$ and $\1_{B_\epsilon}$ is the indicator
function of the ball $B_\epsilon = \{(R,Z) \in \Omega_\epsilon\,;\, R^2 + Z^2 \le
\epsilon^{-2\sigma_1}\}$. We thus have
\begin{equation}\label{Eepsdec}
  E_\epsilon[\eta] \,=\, \frac12 \int_{\Omega_\epsilon} W_\epsilon\,\eta_1^2\dd X +
  \frac12 \int_{\Omega_\epsilon} W_\epsilon\,\eta_2^2\dd X - \frac12 \int_{\Omega_\epsilon}
  \bigl(\phi_1 + \phi_2\bigr)\bigl(\eta_1 + \eta_2\bigr)\dd X\,,
\end{equation}
where $\phi_j = \BS^\epsilon[\eta_j]$ for $j = 1,2$. We claim that
\begin{equation}\label{Ekinreduc}
  \frac12 \int_{\Omega_\epsilon}\bigl(\phi_1 + \phi_2\bigr)\bigl(\eta_1 + \eta_2
  \bigr)\dd X \,=\, E^\kin_\epsilon[\eta_1] + \cO\bigl(\epsilon^\infty
  \|\eta\|_{\cX_\epsilon}^2\bigr)\,,
\end{equation}
so that
\begin{equation}\label{Eepsreduc}
  E_\epsilon[\eta] \,=\, E_\epsilon[\eta_1] + \frac12 \|\eta_2\|_{\cX_\epsilon}^2 +
  \cO\bigl(\epsilon^\infty\|\eta\|_{\cX_\epsilon}^2\bigr)\,. 
\end{equation}

To prove \eqref{Ekinreduc}, we recall that $\phi_j(R,Z) = \frac{1}{2\pi}\int_{\Omega_\epsilon}
K_\epsilon(R,Z;R',Z')\eta_j(R',Z')\dd X'$, where the kernel $K_\epsilon$ is given
by \eqref{Kepsdef}. Using the crude estimate $|F(s)| \le C\bigl(|\log s| + 1\bigr)$, we
easily obtain
\begin{equation}\label{Gest1}
  \bigl|K_\epsilon(R,Z;R',Z')\bigr| \,\le\, C\bigl(1{+}\epsilon|R|\bigr)^a 
  \bigl(1{+}\epsilon|R'|\bigr)^a \bigl(\beta_\epsilon + \bigl|\log D\bigr| + 1\bigr)\,, 
\end{equation}
for some $a > 1/2$. It follows in particular that
\[
  |\phi(R,Z)| \,\le\, C\bigl(\beta_\epsilon + 1\bigr)(1 + \rho)^b\|\eta\|_{\cX_\epsilon}\,,
  \qquad \rho \,=\, \sqrt{R^2+Z^2}\,,
\]
for some $b > 1/2$, and using H\"older's inequality we deduce
\[
  \int_{\Omega_\epsilon} |\phi(R,Z)|\,|\eta_2(R,Z)|\dd X \,\le\, C\bigl(\beta_\epsilon + 1\bigr)
  \|\eta\|_{\cX_\epsilon}^2 \biggl(\int_{B_\epsilon^c}(1+\rho)^{2b}
  \,W_\epsilon(R,Z)^{-1}\dd X\biggr)^{1/2}\,,
\]
where the last integral is $\cO(\epsilon^\infty)$ in view of \eqref{Wunifbd}. In a 
similar way we have 
\[
  |\phi_2(R,Z)| \,\le\, C\bigl(\beta_\epsilon + 1\bigr)(1+\rho)^b
  \biggl(\int_{B_\epsilon^c}(1+\rho')^{2b} |\eta(R',Z')|^2\dd X'\biggr)^{1/2}
  \!=\, \cO\bigl(\epsilon^\infty\|\eta\|_{\cX_\epsilon}\bigr)(1+\rho)^b\,, 
\]
so that $\int_{\Omega_\epsilon}\!\phi_2 \eta_1\dd x = \cO\bigl(\epsilon^\infty
\|\eta\|_{\cX_\epsilon}^2\bigr)$. Altogether we arrive at \eqref{Ekinreduc}. 

Now, since $\eta_1$ is supported in the ball $B_\epsilon$, it follows from
\eqref{Wapprox} and Lemma~\ref{Ekinlem} that
\begin{equation}\label{eta1first}
  \|\eta_1\|_{\cX_\epsilon}^2 \,=\, \|\eta_1\|_{\cX_0}^2 + \cO\bigl(\epsilon^{\gamma_1}
  \|\eta\|_{\cX_\epsilon}^2\bigr)\,, \qquad E_\epsilon^\kin[\eta_1]  \,=\,
  \frac{\beta_\epsilon{-}2}{4\pi}\,\tilde\mu_0^2 + E_0^\kin[\eta_1] + \cO\bigl(\epsilon
  \beta_\epsilon\|\eta\|_{\cX_\epsilon}^2\bigr)\,. 
\end{equation}
Moreover we know from Proposition~\ref{E0prop} that
\begin{equation}\label{eta1second}
  \|\eta_1\|_{\cX_0}^2 \,\le\, C_4 E_0[\eta_1] + C_5\bigl(\tilde\mu_0^2 + \tilde\mu_1^2
  + \tilde\mu_2^2\bigr)\,,
\end{equation}
where $\tilde\mu_0, \tilde\mu_1, \tilde\mu_2$ are the moments of $\eta_1$, which
satisfy $\tilde\mu_j = \mu_j + \cO\bigl(\epsilon^\infty\|\eta\|_{\cX_\epsilon}\bigr)$.
Combining both estimates in \eqref{eta1first} we obtain
\[
  E_0[\eta_1] \,=\, \frac12 \|\eta_1\|_{\cX_0}^2 - E^\kin_0[\eta_1] \,\le\,
  \frac12 \|\eta_1\|_{\cX_\epsilon}^2 - E^\kin_\epsilon[\eta_1] + \frac{\beta_\epsilon{-}2}{4\pi}
  \,\tilde\mu_0^2 + \cO\bigl(\epsilon^{\gamma_1}\|\eta\|_{\cX_\epsilon}^2\bigr)\,,
\]
namely $E_0[\eta_1] \le E_\epsilon[\eta_1] + \frac{\beta_\epsilon-2}{4\pi}\tilde\mu_0^2 +
\cO\bigl(\epsilon^{\gamma_1}\|\eta\|_{\cX_\epsilon}^2\bigr)$. Using in addition
\eqref{eta1second} we deduce
\[
  \|\eta_1\|_{\cX_\epsilon}^2 \,\le\, \|\eta_1\|_{\cX_0}^2 + \cO\bigl(\epsilon^{\gamma_1}
  \|\eta\|_{\cX_\epsilon}^2\bigr) \,\le\, 
  C_4 E_\epsilon[\eta_1] + C\bigl(\beta_\epsilon\tilde\mu_0^2
  + \tilde\mu_1^2+ \tilde\mu_2^2\bigr) + \cO\bigl(\epsilon^{\gamma_1}\|\eta\|_{\cX_\epsilon}^2
  \bigr)\,.
\]
Finally, invoking \eqref{Eepsreduc} and recalling that $C_4 > 2$, we find
\[
  \|\eta\|_{\cX_\epsilon}^2 \,\le\, \|\eta_1\|_{\cX_\epsilon}^2 + \frac{C_4}{2} \|\eta_2
  \|_{\cX_\epsilon}^2 \,\le\, C_4 E_\epsilon[\eta] + C\bigl(\beta_\epsilon\tilde\mu_0^2
  + \tilde\mu_1^2+ \tilde\mu_2^2\bigr) + \cO\bigl(\epsilon^{\gamma_1}\|\eta\|_{\cX_\epsilon}^2
  \bigr)\,,
\]
and estimate \eqref{Eepsineq} follows, since  $\tilde\mu_j = \mu_j + \cO\bigl(\epsilon^\infty
\|\eta\|_{\cX_\epsilon}\bigr)$ for $j = 0,1,2$. 
\end{proof}

\subsection{Diffusive terms in the energy functional}\label{ssecB2}

We justify here the expression \eqref{I4def} of the quantity $I_4$. Integrating by parts
as in \cite{GS3}, we find
\[
  \int_{\Omega_\epsilon} W_\epsilon \tilde\eta\,\cL\tilde \eta\dd X \,=\, -\int_{\Omega_\epsilon} 
  W_\epsilon |\nabla\tilde\eta|^2\dd X - \int_{\Omega_\epsilon} (\nabla W_\epsilon \cdot
  \nabla \tilde\eta)\tilde\eta\dd X - \int_{\Omega_\epsilon} \tilde V_\epsilon
  \tilde\eta^2\dd X\,,
\]
where $\tilde V_\epsilon = \frac14(R\partial_R + Z\partial_Z)W_\epsilon -
\frac12 W_\epsilon$. Similarly, 
\[
  \epsilon\int_{\Omega_\epsilon} W_\epsilon\tilde\eta\,\partial_R\tilde \zeta \dd X \,=\,
  \epsilon\int_{\Omega_\epsilon} W_\epsilon(1+\epsilon R)\tilde\zeta\,\partial_R\tilde
  \zeta \dd X \,=\, -\frac{\epsilon}{2}\int_{\Omega_\epsilon} \partial_R
  \bigl(W_\epsilon(1+\epsilon R)\bigr)\tilde\zeta^2\dd X\,.
\]
On the other hand, integrating by parts and using the relation \eqref{phidef}
between $\tilde\phi$ and $\tilde\eta$, we obtain
\begin{align*}
  \int_{\Omega_\epsilon} \tilde\phi\Bigl(\cL\tilde\eta + \epsilon\partial_R
  \tilde\zeta\Bigr)\dd X \,&=\,
  \int_{\Omega_\epsilon} \tilde\eta\Bigl(\Delta\tilde\phi -\frac{\epsilon \partial_R
  \tilde\phi}{1+\epsilon R}\Bigr)\dd X -\frac12\int_{\Omega_\epsilon} \tilde \eta
  \bigl(R\partial_R + Z\partial_Z\bigr)\tilde\phi\dd X \\
  \,&=\, -\int_{\Omega_\epsilon} \tilde\eta^2(1+\epsilon R)\dd X -\frac12\int_{\Omega_\epsilon}
  \tilde\eta\bigl(R\partial_R + Z\partial_Z\bigr)\tilde\phi\dd X\,.
\end{align*}
It remains to treat the last term in the right-hand side. Here again, we use the relation
\eqref{phidef} and integrate by parts to obtain
\[
  \frac12\int_{\Omega_\epsilon}\tilde\eta\bigl(R\partial_R + Z\partial_Z\bigr)
  \tilde\phi\dd X \,=\, \frac{\epsilon}{4}\int_{\Omega_\epsilon}\frac{R|\nabla
  \tilde\phi|^2}{(1+\epsilon R)^2}\dd X\,. 
\]
Altogether we arrive at \eqref{I4def}, with $V_\epsilon = \tilde V_\epsilon
- (1+\epsilon R)$. 

\subsection{Coercivity of the diffusive quadratic form}\label{ssecB3}

This section is devoted to the proof of Proposition~\ref{Qepsprop}. Given $\epsilon > 0$
sufficiently small, we take a smooth partition of unity of the form $1 = \chi_3^2 + \chi_4^2$,
where $\chi_3, \chi_4$ are radially symmetric and $\chi_3 = 1$ when $\rho \le \frac12
\epsilon^{-\sigma_1}$, $\chi_3 = 0$ when $\rho \ge \epsilon^{-\sigma_1}$. We can also
assume that $|\nabla \chi_3| + |\nabla \chi_4| \le C \epsilon^{\sigma_1}$. Given $\eta$
as in the statement of Proposition~\ref{Qepsprop}, we define $\eta_3 = \chi_3\eta$,  
$\eta_4 = \chi_4\eta$. We thus have the decompositions $\eta^2 = \eta_3^2 + \eta_4^2$,
$\eta\nabla\eta = \eta_3\nabla\eta_3 + \eta_4\nabla\eta_4$, and 
\begin{equation}\label{Qeps0}
  |\nabla\eta|^2 \,=\, |\nabla\eta_3|^2 + |\nabla\eta_4|^2 - \bigl(|\nabla\chi_3|^2 + 
  |\nabla\chi_4|^2 \bigr)\eta^2\,. \qquad 
\end{equation}
As a consequence, the quadratic form $Q_\epsilon[\eta]$ can be decomposed as
\begin{equation}\label{Qepsdecomp}
  Q_\epsilon[\eta] \,=\, Q_\epsilon[\eta_3] + Q_\epsilon[\eta_4] - \int_{\Omega_\epsilon}
  W_\epsilon \bigl(|\nabla\chi_3|^2 + |\nabla\chi_4|^2 \bigr)\eta^2\dd X\,.
\end{equation}
The last term in \eqref{Qepsdecomp} is bounded by $C \epsilon^{2\sigma_1}\|\eta\|_{\cX_\epsilon}^2$
and is thus negligible when $\epsilon \ll 1$. So our main task is to estimate from
below the terms $Q_\epsilon[\eta_3]$ and $Q_\epsilon[\eta_4]$.

We first consider the function $\eta_3$ which is supported in the region where $\rho \le
\epsilon^{-\sigma_1}$. We recall that the weight $W_\epsilon$ in \eqref{Wdef} satisfies the
estimates \eqref{Wapprox}, which read
\begin{equation}\label{Westim}
  |\nabla W_\epsilon(R,Z) - \nabla A(\rho)| + |W_\epsilon(R,Z) - A(\rho)| \,\le\,
  C \epsilon^{\gamma_1} A(\rho)\,, \qquad \hbox{when}~\,\rho \le \epsilon^{-\sigma_1}\,,
\end{equation}
where $\gamma_1 > 0$. We easily deduce that 
\begin{equation}\label{Qeps1}
  Q_\epsilon[\eta_3] \,\ge\, Q_0[\eta_3] - C \epsilon^{\gamma_1} \bigl(\|\nabla\eta_3\|_{\cX_0}^2 +
  \|\rho\eta_3\|_{\cX_0}^2 + \|\eta_3\|_{\cX_0}^2\bigr)\,,
\end{equation}
where $Q_0$ is the limiting quadratic form \eqref{Q0def}. On the other hand, we know from
Proposition~\ref{Q0prop} that
\begin{equation}\label{Qeps2}
  C_8 Q_0[\eta_3] \,\ge\, \|\nabla\eta_3\|_{\cX_0}^2 + \|\rho\eta_3\|_{\cX_0}^2 + \|\eta_3\|_{\cX_0}^2
  - C_9\bigl(\tilde\mu_0^2 + \tilde\mu_1^2 + \tilde\mu_2^2\bigr)\,,
\end{equation}
where $\tilde\mu_0, \tilde\mu_1, \tilde\mu_2$ are the moments of $\eta_3$, which
satisfy $\tilde\mu_j = \mu_j + \cO\bigl(\epsilon^\infty\|\eta\|_{\cX_\epsilon}\bigr)$.
Combining \eqref{Qeps1}, \eqref{Qeps2} and using \eqref{Westim} once again, we
arrive at
\begin{equation}\label{Qeps3}
  \|\nabla\eta_3\|_{\cX_\epsilon}^2 + \|\rho\eta_3\|_{\cX_\epsilon}^2 + \|\eta_3\|_{\cX_\epsilon}^2
  \,\le\, 2C_8 Q_\epsilon[\eta_3] + C\bigl(\tilde\mu_0^2 + \tilde\mu_1^2 + \tilde\mu_2^2\bigr)\,. 
\end{equation}

We next consider the function $\eta_4$, which is nonzero only if $\rho \ge \frac12 \epsilon^{-\sigma_1}$.
Our starting point is the lower bound
\[
  Q_\epsilon[\eta_4] \,\ge\, \frac14\int_{\Omega_\epsilon} W_\epsilon |\nabla\eta_4|^2\dd X +
  \int_{\Omega_\epsilon} \Bigl(V_\epsilon - \frac{|\nabla W_\epsilon|^2}{3W_\epsilon}\Bigr)\eta_4^2\dd X\,,
\]
which is obtained from \eqref{Qepsdef} by applying Young's inequality to the middle term in
the right-hand side. Using the expression \eqref{Wdef} of the weight function, as well
as the estimates \eqref{Westim} in the inner region $\Omega_\epsilon'$, it is not difficult
to verify that
\[
  \frac{V_\epsilon}{W_\epsilon} - \frac{|\nabla W_\epsilon|^2}{3W_\epsilon^2} ~\ge~ \begin{cases}
  C \rho^2 - \tilde C & \hbox{in}~\,\Omega_\epsilon'\,, \\
  - \tilde C & \hbox{in}~\,\Omega_{\epsilon}''\,, \\
  C \rho^{2\gamma} & \hbox{in}~\,\Omega_\epsilon'''\,, \end{cases}
\]
for some positive constants $C,\tilde C$. It follows that
\begin{equation}\label{Qeps4}
  Q_\epsilon[\eta_4] \,\ge\, \frac14 \|\nabla\eta_4\|_{\cX_\epsilon}^2 + C \int_{\Omega_\epsilon' \cup
  \Omega_\epsilon'''} W_\epsilon\,\rho_\gamma^2\eta_4^2\dd X - \tilde C \int_{\Omega_\epsilon''}
  W_\epsilon\,\eta_4^2\dd X\,.
\end{equation}
If we now combine \eqref{Qeps3} and \eqref{Qeps4}, we obtain
\begin{equation}\label{Qeps5}
\begin{split}
    \|\nabla\eta_3\|_{\cX_\epsilon}^2  + \|\nabla\eta_4\|_{\cX_\epsilon}^2 + \|\eta\|_{\cX_\epsilon}^2
  & + \int_{\Omega_\epsilon'\cup \Omega_\epsilon'''} W_\epsilon\,\rho_\gamma^2 \eta^2\dd X \\ \,&\le\, 
  C_{10}\bigl(Q_\epsilon[\eta_3] + Q_\epsilon[\eta_4]\bigr) + C_{11}\Bigl(\tilde\mu^2 +
  \int_{\Omega_\epsilon''} W_\epsilon \eta^2\dd X\Bigr)\,,
\end{split}
\end{equation}
for some positive constants $C_{10}, C_{11}$, where $\tilde\mu^2 = \tilde\mu_0^2 + \tilde\mu_1^2
+ \tilde\mu_2^2$. Finally, using again \eqref{Qeps0} as well as \eqref{Qepsdecomp}, and 
recalling that $\tilde\mu_j = \mu_j + \cO\bigl(\epsilon^\infty\|\eta\|_{\cX_\epsilon}\bigr)$, 
we deduce \eqref{Qepsineq} from \eqref{Qeps5}. \QED 

\medskip\noindent{\bf Acknowledgments.} ThG is partially supported by the grant
SingFlows ANR-18-CE40-0027 of the French National Research Agency (ANR). The
research of VS is supported in part by grant DMS 1956092 from the National
Science Foundation.

\bigskip\noindent
{\bf Thierry Gallay}\\
Institut Fourier, Universit\'e Grenoble Alpes, CNRS, Institut Universitaire de France\\
100 rue des Maths, 38610 Gi\`eres, France\\
Email\: {\tt Thierry.Gallay@univ-grenoble-alpes.fr}

\bigskip\noindent
{\bf Vladim\'ir \v{S}ver\'ak}\\
School of Mathematics, University of Minnesota\\
127 Vincent Hall, 206 Church St.\thinspace SE, Minneapolis, MN 55455, USA\\
Email\: {\tt sverak@math.umn.edu}

\end{document}